\documentclass{amsart}
\usepackage{amssymb,amsmath, amscd, mathrsfs, graphics,hyperref}

\newtheorem{thm}{Theorem}[section]
\newtheorem{lemma}[thm]{Lemma}
\newtheorem{prop}[thm]{Proposition}
\newtheorem{cor}[thm]{Corollary}

\theoremstyle{definition}
\newtheorem{remark}[thm]{Remark}
\newtheorem{ex}[thm]{Example}
\newtheorem{defn}[thm]{Definition}

\title[Noncommutative hyperbolicity]{Noncommutative 
hyperbolic metrics}

\author{Serban T. Belinschi}
\address{CNRS - Institut de Math\'ematiques de Toulouse}
\thanks{This work was initiated, and largely completed, during the 2016 Thematic Semester in Analysis, 
funded by ANR-11-LABX-0040-CIMI within the program ANR-11-IDEX-0002-02.}

\author{Victor Vinnikov}
\address{Ben Gurion University of the Negev}

\begin{document}

\begin{abstract}
We characterize certain noncommutative domains in terms of 
noncommutative holomorphic equivalence via a pseudometric
that we define in purely algebraic terms. We prove some 
properties of this pseudometric and provide an 
application to free probability.

\end{abstract}

\maketitle

\section{Introduction}

A noncommutative function is a function, defined on a domain in the disjoint union
of square matrices of all sizes over a vector space, that satisfies natural
compatibility conditions: it respects direct sums and similarities.
Noncommutative function theory is the free analogue of classical function theory,
much like operator space theory \cite{ER,Paulsen} is the free analogue
of classical Banach space theory.
Noncommutative functions were first introduced by Taylor \cite{taylor0,taylor}
in his monumental work on noncommutative spectral theory.
Their theory was further developed by Voiculescu \cite{V1,VFAQ2},
and, in a systematic fashion, by Kalyuzhnyi-Verbovetskyi and the second author \cite{ncfound}.
We mention also
the work of Helton, Klep, and McCullough (see, e.g., \cite{HKM1,HKM}), of Popescu (see, e.g., \cite{Po06,Po10}), 
of Muhly and Solel (see, e.g., \cite{MS}), and of Agler and McCarthy (see, e.g., \cite{AMc}).
Noncommutative functions appear naturally in a large variety of settings: noncommutative algebra,
systems and control, spectral theory, and free probability.
They possess very strong regularity properties (reminiscent
of the regularity properties of usual analytic functions) and admit a good difference-differential calculus.

At least some analytic aspects of noncommutative function theory are now quite well understood
(see, e.g., \cite{Pa,Sha,BMV1} and the references therein for some recent developments).
The purpose of this paper is to move in the direction of noncommutative geometric function theory
by introducing a free analogue of Kobayashi (pseudo)metric for general noncommutative domains.

We review the basics of noncommutative functions, noncommutative kernels, and topologies
for noncommutative sets in Section \ref{sec:two}. 
Following that, we introduce in Section \ref{pseudodistance} a ``noncommutative length function'' $\delta_{\mathcal D}$
for a noncommutative set ${\mathcal D}$ (satisfying some mild assumptions);
$\delta_{\mathcal D}$ reflects in a natural way the ``noncommutative geometry'' of ${\mathcal D}$.
It follows immediately that any noncommutative function between two noncommutative sets
is contractive with respect to the corresponding noncommutative length functions.
We show that the noncommutative length function $\delta_{\mathcal D}$
is nondegenerate if and only if there are no nonconstant
noncommutative functions from the whole noncommutative space over ${\mathbb C}$ 
to the multiples of a given matrix level $k$ in ${\mathcal D}$, for all $k \geq 1$.
We use the noncommutative length function to define a pseudodistance on ${\mathcal D}$ in several ways;
we mimick the definition of the Kobayashi pseudodistance using either the Lempert function
or the infinitesimal Kobayashi pseudometric, and possibly allowing to go to higher matrix levels. 

Our definition of the noncommutative length function is motivated by \cite{JLMS};
it would be interesting to clarify possible connections with \cite{A}. 

In Section \ref{sec:four}, we consider the large class of noncommutative sets that are defined by noncommutative kernels.
These can be thought of as analogues of the sets $\{\rho(z,\bar z)>0\}$ for some (global) defining function $\rho$
in several complex variables. In this case it is possible to obtain explicit formulae for the noncommutative
length function by applying the noncommutative difference-differential operators to the kernel.
This applies in particular to generalized noncommutative balls (or halfplanes) defined by a noncommutative
function that were investigated in \cite{AMc1,AMc3,AMc4,BMV1} as suitable domains for interpolation problems.
For these, the noncommutative distance turns out to be dominated by the usual Kobayashi distance on every matrix level,
with equality for noncommutative balls (or halfplanes) over a $C^*$ algebra.

In Section \ref{sec:five}, we seek to apply the noncommutative metric to noncommutative function theory.
We show that under some assumptions, isometries with respect the noncommutative length function are exactly
bijective noncommutative mappings,
and a noncommutative mapping from a noncommutative set to a subset that is at a positive distance from the boundary
of the set has a unique attracting fixed point. 

Finally, in Section \ref{sec:six}, we use the noncommutative metric on the noncommutative upper half plane over
a finite von Neumann algebra to solve the functional equation
$\omega(b) = b + h(\omega(b))$ for $h$ a noncommutative self map of the noncommutative upper halfplane to itself
satisfying a rather mild vanishing condition at infinity. 
This solves the problem of defining free convolution
powers of distributions of unbounded selfadjoint random variables in free probability.

\section{Noncommutative domains, functions and kernels}
\label{sec:two}

\subsection{Noncommutative functions}
Noncommutative functions originate in Joseph L. Taylor's work 
\cite{taylor0,taylor} on spectral theory and functional calculus
for $k$-tuples of non-commuting operators.
We largely follow \cite{ncfound} in our presentation  
of noncommutative sets and functions.  We refer to \cite{ncfound}
for details on, and proofs of, the statements below.

Let us introduce the following notation: if $S$ is 
a nonempty set, we denote by $S^{m\times n}$ the 
set of all matrices with $m$ rows and $n$ columns 
having entries from $S$. If $S=\mathbb F$ is a field,
then we use the standard notation $GL_n(\mathbb F)$ for
the group of matrices $X$ in $\mathbb F^{n\times n}$ which 
are invertible (that is, there exists $X^{-1}\in\mathbb F^{n\times n}$
such that $XX^{-1}=X^{-1}X=I_n$, where $I_n$ is the diagonal
matrix having the multiplicative unit of $\mathbb F$ on the
diagonal and zero elsewhere). We will work almost exclusively 
with subsets of operator spaces and operator systems (linear subspaces of the 
algebra $B(\mathcal H)$ of bounded operators over a Hilbert
space $\mathcal H$ -- which we assume to be separable -- 
which contain the unit $1$ of $B(\mathcal H)$, are norm-closed and selfadjoint - see \cite{ER}). However 
some of our definitions hold in much broader generality.
Given a complex vector space $\mathcal V$, 
a {\em noncommutative set} is a family 
$\Omega_{\rm nc}:=(\Omega_n)_{n\in\mathbb N}$ such that
\begin{enumerate}
\item[(a)] for each $n\in\mathbb N$, $\Omega_n\subseteq\mathcal
V^{n\times n};$
\item[(b)] for each $m,n\in\mathbb N$, we have $\Omega_m\oplus
\Omega_n\subseteq\Omega_{m+n}$.
\end{enumerate}
The noncommutative set $\Omega_{\rm nc}$ is called {\em right 
admissible} if in addition the condition (c) below is satisfied:
\begin{enumerate}
\item[(c)] for each $m,n\in\mathbb N$ and $a\in\Omega_m,b\in
\Omega_n,w\in\mathcal V^{m\times n}$, there is an $\epsilon>0$
such that $\begin{bmatrix}
a & zw\\
0 & b\end{bmatrix}\in\Omega_{m+n}$ for all $z\in\mathbb C,
|z|<\epsilon$.
\end{enumerate}
Left admissible sets are defined similarly, except that $zw$ 
appears in the lower left corner of the matrix.

Given complex vector spaces $\mathcal{V,W}$ and a noncommutative set 
$\Omega_{\rm nc}\subseteq\coprod_{n=1}^\infty\mathcal V^{n\times n}$, a 
{\em noncommutative function} is a family $f:=(f_n)_{n\in\mathbb N}$
such that $f_n\colon\Omega_n\to\mathcal W^{n\times n}$ and
\begin{enumerate}
\item $f_m(a)\oplus f_n(b)=f_{m+n}(a\oplus b)$ for all
$m,n\in\mathbb N$, $a\in\Omega_m,b\in\Omega_n$;
\item for all $n\in\mathbb N$, $f_n(T^{-1}aT)=T^{-1}f_n(a)T$ whenever
$a\in\Omega_n$ and $T\in GL_n(\mathbb C)$ are such that $T^{-1}aT$
belongs to the domain of definition of $f_n$.
\end{enumerate}
These two conditions are equivalent to the requirement that $f$ respects
intertwinings by scalar matrices:
\begin{enumerate}
\item[(I)] For all $m,n\in\mathbb N$, $a\in\Omega_m,b\in\Omega_n$, 
$S\in\mathbb C^{m\times n}$, we have
\begin{equation}\label{inter}
aS=Sb\implies f_m(a)S=Sf_n(b).
\end{equation}
\end{enumerate}

If $\mathcal{V,W}$ are operator spaces,
it is shown in \cite[Theorem 7.2]{ncfound}) that, under very 
mild openness conditions on $\Omega_{\rm nc}$, local boundedness for 
$f$ implies each $f_n$ is analytic as a map between Banach spaces. More
specifically, if $\Omega_{\rm nc}$ is finitely open (that is, for all $n\in
\mathbb N$, the intersection of $\Omega_n$ with any finite 
dimensional complex subspace is open) and $f$ is locally 
bounded on slices (that is, for every 
$n\in\mathbb N$, for every $a\in\Omega_n$ and 
$b\in\mathcal V^{n\times n}$, there exists an 
$\varepsilon>0$ such that the set $\{f_n(a+zb)\colon
z\in\mathbb C,|z|<\varepsilon\}$ is bounded in $\mathcal W^{n\times n}$),
then each $f_n$ is G\^{a}teaux complex differentiable on $\Omega_n$
(see Section \ref{top} below). 
Indeed, this is a consequence of the following 
essential property of noncommutative functions: if $\Omega_{\rm nc}$ is 
admissible, $a\in \Omega_n, c\in\Omega_m, b\in\mathcal V^{n\times m}$ 
such that $\begin{bmatrix}
a & b \\
0 & c
\end{bmatrix}\in\Omega_{n+m}$, then there exists a linear map
$\Delta f_{n,m}(a,c)\colon\mathcal V^{n\times m}\to
\mathcal W^{n\times m}$ such that 
\begin{equation}\label{FDQ}
f_{n+m}\left(\begin{bmatrix}
a & b \\
0 & c
\end{bmatrix}\right)=\begin{bmatrix}
f_n(a) & \Delta f_{n,m}(a,c)(b) \\
0 & f_m(c)
\end{bmatrix}.
\end{equation}
This implies in particular that $f_{n+m}$ extends to
the set of all elements $\begin{bmatrix}
a & b \\
0 & c
\end{bmatrix}$ such that $a\in \Omega_n, c\in\Omega_m, 
b\in\mathcal V^{n\times m}$ (see \cite[Section 2.2]{ncfound}).
Two properties of this operator that are important for us are
\begin{equation}\label{FDC}
\Delta f_{n,n}(a,c)(a-c)=f(a)-f(c)=\Delta f_{n,n}(c,a)(a-c),\quad
 \Delta f_{n,n}(a,a)(b)=f_n'(a)(b),
\end{equation}
the derivative of $f_n$ in $a$ aplied to the element
$b\in\mathcal V^{n\times m}$. Moreover, $\Delta f(a,c)$ 
as functions of $a$ and $c$, respectively, satisfy 
properties similar to the ones described in items 
(1), (2) above -- see \cite[Sections 2.3--2.5]{ncfound}
for details (for convenience, from now on we shall suppress the 
indices denoting the level for noncommutative functions, as 
it will almost always be obvious from the context). 

\begin{ex}\label{example21}
There are many examples of noncommutative functions. We provide here three.
\begin{enumerate}
\item The best known is provided by the classical 
theory of analytic functions of one complex variable: if $D$ is a simply connected domain in $\mathbb C$ 
and $f\colon D\to\mathbb C$ is analytic, then $f$ is the first level of an nc map
$f\colon\coprod_{n=1}^\infty\{A\in\mathbb C^{n\times n}\colon\sigma(A)\subset D\}\to
\coprod_{n=1}^\infty\mathbb C^{n\times n}$ given by the classical analytic functional calculus:
$f_n(A)=(2\pi i)^{-1}\int_\gamma(A-\zeta I_n)^{-1}f(\zeta)\,{\rm d}\zeta$, for some simple closed
curve $\gamma$ which surrounds once counterclockwise the spectrum $\sigma(A)$ of $A$.

\item If $P(X_1,\dots,X_k)$ is a polynomial in $k$ non-commuting indeterminates $X_1,\dots,X_k$
and $\mathcal A$ is a $C^*$-algebra, then the evaluation $P(a_1,\dots,a_k)$, 
$a_j\in\mathcal A^{n\times n}$, $n\in\mathbb N$, is an nc function. More
generally, this can be extended to power series $P$ with (finite or infinite)
radius of convergence (see, for instance, \cite{Popescu0}).

\item If $\mathcal A$ is a unital $C^*$-algebra and $B\subseteq\mathcal A$ is an inclusion of 
$C^*$-algebras which share the same unit, assume that $E\colon\mathcal A\to B$ is a unit-preserving
conditional expectation. If $X=X^*\in\mathcal A$, then the map 
$G_X$ defined by $G_{X,n}(b)=({\rm Id}_{\mathbb C^{n\times n}}\otimes E)
\left[(b-I_n\otimes X)^{-1}\right],$ $b\in B^{n\times n}$, is an nc function (see \cite{V2,V1}). Its domain
is the set of all $b$ such that $b-I_n\otimes X$ is invertible. The {\em noncommutative upper 
half-plane} $\coprod_{n=1}^\infty\{b\in B^{n\times n}\colon(b-b^*)/2i>0\}$ is a natural nc 
subdomain on which $G_X$ is defined.

\end{enumerate}

\end{ex}

\subsection{Noncommutative kernels}

This section follows mostly \cite{BMV}. Let $\Omega_{\rm nc}$
be a noncommutative subset of the operator space $\mathcal V$.
Consider two other operator spaces $\mathcal V_0$ and $\mathcal V_1$.
Denote by $\mathcal L(\mathcal V_0,\mathcal V_1)$ the space of
linear operators from $\mathcal V_0$ to $\mathcal V_1$. A {\em
global kernel} on $\Omega_{\rm nc}$ is a function
$K\colon\Omega_{\rm nc}\times\Omega_{\rm nc}\to\mathcal L(\mathcal V_0,\mathcal V_1)_{\rm nc}$
such that
\begin{eqnarray}
& & a\in\Omega_m,c\in\Omega_n\implies K(a,c)\in\mathcal L(\mathcal V_0^{m\times n},\mathcal V_1^{m\times n})\label{quatro}\\
& & K\left(\begin{bmatrix}
a & 0\\
0 & \tilde{a}
\end{bmatrix},\begin{bmatrix}
c & 0\\
0 & \tilde{c}
\end{bmatrix}\right)\left(\begin{bmatrix}
P_{1,1} & P_{1,2}\\
P_{2,1} & P_{2,2}
\end{bmatrix}\right)=\begin{bmatrix}
K(a,c)(P_{1,1}) & K(a,\tilde{c})(P_{1,2})\\
K(\tilde{a},c)(P_{2,1}) & K(\tilde{a},\tilde{c})(P_{2,2})
\end{bmatrix},\label{cinco}
\end{eqnarray}
for any $m,\tilde{m},n,\tilde{n}\in\mathbb N$, $a\in\Omega_m,\tilde{a}\in\Omega_{\tilde{m}},
c\in\Omega_n,\tilde{c}\in\Omega_{\tilde{n}},$ $P_{1,1}\in\mathcal V_0^{m\times n},
P_{1,2}\in\mathcal V_0^{m\times\tilde{n}},P_{2,1}\in\mathcal V_0^{\tilde{m}\times n},
P_{2,2}\in\mathcal V_0^{\tilde{m}\times\tilde{n}}$ (that is, $\begin{bmatrix}
P_{1,1} & P_{1,2}\\
P_{2,1} & P_{2,2}
\end{bmatrix}\in\mathcal V_0^{(m+\tilde{m})\times(n+\tilde{n})}$).
Obviously, condition \eqref{cinco} can be extended to evaluations of $K$ in diagonal matrices with
arbitrarily many blocks on the diagonal. The kernel $K$ is called an
{\em affine noncommutative kernel} if in addition to condition
\eqref{quatro}, it respects intertwinings:
\begin{eqnarray}
& & a\in\Omega_m,\tilde{a}\in\Omega_{\tilde{m}},S\in\mathbb C^{\tilde{m}\times m}\text{ are such that }
Sa=\tilde{a}S,\nonumber\\
& & c\in\Omega_n,\tilde{c}\in\Omega_{\tilde{n}},T\in\mathbb C^{n\times\tilde{n}}\text{ are such that }
cT=T\tilde{c},\nonumber\\
& & P\in\mathcal V_0^{m\times n}\implies SK(a,c)(P)T=K(\tilde{a},\tilde{c})(SPT).\label{sei}
\end{eqnarray}
Conditions \eqref{quatro} and \eqref{sei} are equivalent to 
conditions \eqref{quatro}, \eqref{cinco} and
\begin{eqnarray}
& & a,\tilde{a}\in\Omega_m,S\in GL_m(\mathbb C)\text{ are such that }SaS^{-1}=\tilde{a},\nonumber\\
& & c,\tilde{c}\in\Omega_n,T\in GL_n(\mathbb C)\text{ are such that }T^{-1}cT=\tilde{c},\nonumber\\
& & P\in\mathcal V_0^{m\times n}\implies K(\tilde{a},\tilde{c})(P)=SK(a,c)(S^{-1}PT^{-1})T.\label{sette}
\end{eqnarray}
If $f\colon\Omega_{\rm nc}\to\mathcal W_{\rm nc}$ is a noncommutative 
map, then $\Omega_{\rm nc}\times\Omega_{\rm nc}\ni(a,c)\mapsto
\Delta f(a,c)\in\mathcal L(\mathcal V,\mathcal W)_{\rm nc}$ satisfies
the above conditions (see \cite[Proposition 2.15]{ncfound}).

We call $K$ a {\em noncommutative (nc) kernel} if $K$ satisfies 
\eqref{quatro}  and respects intertwinings in the following sense:
\begin{eqnarray}
& & a\in\Omega_m,\tilde{a}\in\Omega_{\tilde{m}},S\in\mathbb C^{\tilde{m}\times m}\text{ are such that }
Sa=\tilde{a}S,\nonumber\\
& & c\in\Omega_n,\tilde{c}\in\Omega_{\tilde{n}},T\in\mathbb C^{\tilde{n}\times n}\text{ are such that }
Tc=\tilde{c}T,\nonumber\\
& & P\in\mathcal V_0^{m\times n}\implies SK(a,c)(P)T^*=K(\tilde{a},\tilde{c})(SPT^*).\label{otto}
\end{eqnarray}
Conditions \eqref{quatro} and \eqref{otto} are equivalent to 
conditions \eqref{quatro}, \eqref{cinco} and
\begin{eqnarray}
& & a,\tilde{a}\in\Omega_m,S\in GL_m(\mathbb C)\text{ are such that }SaS^{-1}=\tilde{a},\nonumber\\
& & c,\tilde{c}\in\Omega_n,T\in GL_n(\mathbb C)\text{ are such that }TcT^{-1}=\tilde{c},\nonumber\\
& & P\in\mathcal V_0^{m\times n}\implies K(\tilde{a},\tilde{c})(P)=SK(a,c)(S^{-1}P(T^{-1})^*)T^*.\label{nove}
\end{eqnarray}
Observe that if $K$ is an affine nc kernel, 
then $(a,c)\mapsto K(a,c^*)$ is an nc kernel.

We say that a noncommutative kernel $K$ is a {\em completely
positive noncommutative (cp nc) kernel} if in addition
\begin{equation}\label{dieci}
a\in\Omega_m,P\ge0\text{ in }\mathcal V_0^{m\times m}\implies K(a,a)(P)\ge0\text{ in }\mathcal 
V_1^{m\times m}\text{ for all }m\in\mathbb N.
\end{equation}
If $\mathcal V_0,\mathcal V_1$ are $C^*$-algebras, 
then \eqref{dieci} is equivalent to requiring that for all $N\in\mathbb N$, $m_1,m_2,\dots,m_N\in
\mathbb N$,
\begin{equation}\label{undici}
a^{(j)}\in\Omega_{m_j},P_j\in\mathcal V_0^{N\times m_j},b_j\in\mathcal V_1^{m_j},1\le j\le N\implies 
\sum_{i,j=1}^Nb_i^*K(a^{(i)},a^{(j)})(P_i^*P_j)b_j\ge0
\end{equation}
(see \cite[Proposition 2.2]{BMV}). If $K(a,a)$ is completely positive, 
then it is also completely bounded and $\|K(a,a)\|=\|K(a,a)\|_{\rm cb}
=\|K(a,a)(1)\|$.

\begin{ex}
Let $\mathcal A$ be a $C^*$-algebra.
The simplest non-constant nc kernel is $\mathcal A_{\rm nc}\times\mathcal A_{\rm nc}\ni(a,c)
\mapsto a\cdot c^*\in\mathcal L(\mathcal A,\mathcal A)_{\rm nc}.$ That is, for $m,n\in
\mathbb N$, $a\in\mathcal A^{m\times m},c\in\mathcal A^{n\times n}$ and $P\in
\mathcal A^{m\times n},$ we have $(a,c)\mapsto(P\mapsto aPc^*)$. More generally, if
$G,H$ are nc functions from $\Omega_{\rm nc}\subseteq\mathcal V_{\rm nc}$ to $\mathcal A_{\rm nc}$,
then $(a,c)\mapsto G(a)\cdot H(c)^*$ is an nc kernel. One can further pre-compose this 
kernel with a completely bounded map $\Psi\colon\mathcal A\to\mathcal A$:
$$
\Omega_m\times\Omega_n\ni(a,c)\mapsto\left[\mathcal A^{m\times n}\ni P
\mapsto G(a)({\rm Id}_{\mathbb C^{m\times n}}\otimes\Psi)(P)H(c)^*\in\mathcal A^{m\times n}\right]
$$
is an nc kernel. If $G=H$ and $\Psi$ is completely positive, then this is a cp nc kernel. In a certain 
sense, {\em all} nc kernels are of this form (we refer to \cite[Theorem 3.1]{BMV} for the precise statement).
Note also that $(a,c)\mapsto[P
\mapsto G(a)({\rm Id}_{\mathbb C^{m\times n}}\otimes\Psi)(P)H(c^*)^*]$ is an affine nc kernel.
\end{ex}

\begin{ex}\label{kernel-ex}
One of the main objectives of this paper is to analyze certain metric properties of 
noncommutative sets. An important class of such sets is given precisely by noncommutative kernels. 
Let $\mathcal A$ be a $C^*$-algebra, $\mathcal V$ be an operator space and $\Omega_{\rm nc}\subset
\mathcal V_{\rm nc}$ be an nc set. Assume that $K\colon\Omega_{\rm nc}\times\Omega_{\rm nc}
\to\mathcal L(\mathcal A)_{\rm nc}$ is a noncommutative kernel. We may define the set
$$
\mathcal D_K:=\coprod_{n=1}^\infty\underbrace{\{a\in\Omega_n\colon K(a,a)(I_n)>0\}}_{\mathcal D_n}.
$$
Observe that if $K$ were assumed instead to be an affine 
nc kernel, then the above definition would change to 
$\mathcal D_n=\{a\in\Omega_n\colon K(a,a^*)(I_n)>0\}$.
Clearly $\mathcal D_K$ may be empty or equal to $\Omega_{\rm nc}$.

If $a\in\Omega_m,\tilde{a}\in\Omega_{\tilde{m}}$, then, by 
\eqref{quatro} and \eqref{cinco}, $K(a\oplus\tilde{a},a\oplus\tilde{a})
\in\mathcal L(\mathcal A^{(m+\tilde{m})\times(m+\tilde{m})})$ and
\begin{eqnarray*}
K(a\oplus\tilde{a},a\oplus\tilde{a})(I_{m+\tilde{m}})& = &
\begin{bmatrix}
K(a,a)(I_m) & K(a,\tilde{a})(0)\\
K(\tilde{a},a)(0) & K(\tilde{a},\tilde{a})(I_{\tilde{m}})
\end{bmatrix}\\
& = & \begin{bmatrix}
K(a,a)(I_m)& 0 \\
0 & K(\tilde{a},\tilde{a})(I_{\tilde{m}})
\end{bmatrix}>0.
\end{eqnarray*}
Thus, under the weaker assumptions that $K$ is a global kernel,
we are guaranteed that $\mathcal D_K$ is a noncommutative set.
Under our assumption that $K$ is a noncommutative kernel, we
have in addition that for any $S\in GL_m(\mathbb C)$, 
$$
K(SaS^{-1},(S^{-1})^*aS^*)(I_m)=SK(a,a)(S^{-1}I_mS)S^{-1}=SK(a,a)(I_m)S^{-1}.
$$
Thus, if $S$ is unitary (that is, $S^*=S^{-1}$), then $K(SaS^{*},SaS^*)(I_m)>0$
whenever $K(a,a)(I_m)>0$. We conclude that {\em if $K$ is an nc kernel 
on $\Omega_{\rm nc}$, then $\mathcal D_K$ is a noncommutative set
which is invariant with respect to conjugation by scalar unitary matrices.}

Some of the more famous examples of noncommutative sets are given by nc kernels: 

\begin{enumerate}

\item[(i)] The noncommutative upper half-plane 
$H^+(\mathcal A)=\coprod_{n=1}^\infty H^+(\mathcal A^{n\times n}))$, 
where $H^+(\mathcal A^{n\times n})=\{a\in\mathcal A^{n\times n}\colon\Im a>0\}$
(we remind the reader that $\Im b=(b-b^*)/2i,\Re b=(b+b^*)/2$, so that
$b=\Re b+i\Im b$). The kernel in this case is $K(a,c)(P)=(aP-(cP^*)^*)/2i$,
$a\in\mathcal A^{m\times m},c\in\mathcal A^{n\times n}$, 
$P\in\mathcal A^{m\times n}$. It is easy to verify that this is a 
globally defined nc kernel. This set is important in free probability
(see \cite{V2,V1}).

\item[(ii)] The unit ball $B_1(\mathcal A)=\coprod_{n=1}^\infty B_1(\mathcal A^{n\times n})$, 
where $B_1(\mathcal A^{n\times n})=\{a\in\mathcal A^{n\times n}\colon\|a\|<1\}$ (the
norm considered being the $C^*$-norm on $\mathcal A^{n\times n}$).
Here the kernel is even simpler: $K(a,c)(P)=1-aPc^*.$

\item[(iii)] More generally, if $G$ is a noncommutative 
function with values in $\mathcal A$, 
we could define $H^+(\mathcal A)_G$
by using the kernel $K(a,c)(P)=(G(a)P-(G(c)P^*)^*)/2i$ and 
$B_1(\mathcal A)_G$ by using the kernel $K(a,c)(P)=1-G(a)PG(c)^*$.
\end{enumerate}

However, some are not:
\begin{enumerate}

\item[(iv)] 
Consider $\mathcal N(\mathcal A)=\coprod_{n=1}^\infty\{a\in\mathcal A^{n\times n}\colon
a^n=0\}$. Clearly $\mathcal N(\mathcal A)$ is closed under direct sums, and, moreover,
if $S\in GL_n(\mathbb C)$ and $a\in\{a\in\mathcal A^{n\times n}\colon
a^n=0\}$, then $(SaS^{-1})^n=Sa^nS^{-1}=0$. So this set is in fact 
invariant under conjugation by {\em all} of $GL_n(\mathbb C)$, not 
just by the unitary group. This is because $\mathcal N(\mathcal A)$ is 
``thin,'' in the sense that it has empty interior in all the natural topologies 
on nc sets (see below). Thus, one cannot expect that $\mathcal N(\mathcal A)$
is of the form $\mathcal D_K$ for an nc kernel $K$.
\end{enumerate}
\end{ex}

\subsection{Three topologies on noncommutative sets}\label{top}

As already stated, operator spaces constitute
the natural framework for noncommutative 
function theory. We recall that (see, for instance,
\cite{ER}) if $\mathcal V$
is an operator space, then 
$$
\|a\oplus\tilde{a}\|_{m+\tilde{m}}=\max\{\|a\|_m,\|\tilde{a}\|_{\tilde{m}}\},\quad
m,\tilde{m}\in\mathbb N,a\in\mathcal V^{m\times m},\tilde{a}\in
\mathcal V^{\tilde{m}\times\tilde{m}},
$$
and
$$
\|SaT\|_n\leq\|S\|\|a\|_m\|T\|,\quad m,n\in\mathbb N,a\in\mathcal V^{m\times m},
S\in\mathbb C^{n\times m},T\in\mathbb C^{m\times n}.
$$
A topology naturally compatible with these norm conditions is the
{\em uniformly-open topology}. It has as basis balls defined the following
way: if $c\in\mathcal V^{s\times s}$ and $r\in(0,+\infty)$, then
$$
B_{\rm nc}(c,r)=\coprod_{n=1}^\infty\left\{a\in\mathcal V^{sn\times sn}\colon
\left\|a-\oplus_{j=1}^nc\right\|_{sn}<r\right\}.
$$
This topology is not Hausdorff.
A noncommutative function $f$ defined on a noncommutative
set $\Omega_{\rm nc}\subseteq\mathcal V_{\rm nc}$ with values 
in an operator space is said to be {\em uniformly analytic} if 
$\Omega_{\rm nc}$ is uniformly open, and $f$ is uniformly 
locally bounded and complex differentiable at
each level. It is shown in \cite[Corollary 7.28]{ncfound} 
that $f$ is analytic if and only if it is uniformly 
locally bounded (that is, the requirement of complex
differentiability at each level is automatically satisfied by 
an nc function which is uniformly locally bounded on
a uniformly open nc set).

The second important topology (already mentioned above) 
is the {\em finitely open topology}: a set 
$\Omega_{\rm nc}\subseteq\mathcal V_{\rm nc}$ is called
finitely open if for any $n\in\mathbb N$, the intersection
of $\Omega_n$ with any finite dimensional subspace 
$\mathcal X$ of $\mathcal V^{n\times n}$ is open in 
the Euclidean topology of $\mathcal X$. It is shown in
\cite[Theorem 7.2]{ncfound} that if $f$ is a noncommutative
function defined on $\Omega_{\rm nc}$ which is locally
bounded on slices, then $f$ is analytic on slices, in the
sense that for any $n\in\mathbb N$ and any finite 
dimensional subspace $\mathcal X$ of 
$\mathcal V^{n\times n}$, $f|_{\mathcal X}$ is
analytic as a function of several complex variables.

Finally, one can also consider the topology in which 
a set $\Omega_{\rm nc}$ is open in $\mathcal V_{\rm nc}$
if and only if $\Omega_n$ is open in the topological vector space topology of
$\mathcal V^{n\times n}$ for all $n\in\mathbb N$. Observe
that such a set is also finitely open. We refer to it as the
{\em level topology.}


\section{A (pseudo)distance on noncommutative sets}\label{pseudodistance}

Let $\mathcal V$ be a complex topological vector space. As we 
progress through the paper, we put more and more structure on $\mathcal V$,
but for our first definition, we need nothing more than the axioms
of a complex topological vector space. For now we endow 
$\mathcal V^{n\times m}$, $n,m\in\mathbb N$, with the usual (product) 
topology. Let $\mathcal D$ be a noncommutative subset of 
$\mathcal V_{\rm nc}$ and consider the following properties:
\begin{enumerate}
\item For any $n\in\mathbb N$, $\mathcal D_n$ is open in 
$\mathcal V^{n\times n}$;\label{propr1set}
\item If $U$ is a unitary $n\times n$ complex matrix and $a\in
\mathcal D_n$, then $UaU^*\in\mathcal D_n$;\label{propr2set}
\item If $a\in\mathcal V^{n\times n},c\in\mathcal V^{m\times m}$ 
are such that $\begin{bmatrix} 
a & 0\\ 
0 & c
\end{bmatrix}\in\mathcal D_{n+m}$, then $a\in\mathcal D_n$, $c\in
\mathcal D_m$. (Note that this is a sort of ``converse'' of 
part (b) of the definition of noncommutative sets.)\label{propr3set}
\end{enumerate}
Let $\mathcal S_{n,m}=\{g\colon\mathcal V^{n\times m}\to[0,+\infty]
\colon g(tb)=tg(b)\forall t\ge0\}$ (with the convention $0\times(+
\infty)=+\infty$), and define 
$\mathcal S=\displaystyle\coprod_{n,m\in\mathbb N}\mathcal S_{n,m}$.
Define a function $\delta_\mathcal D\colon\mathcal D\times\mathcal D\to
\mathcal S$ such that $\delta_\mathcal D(a,c)\in\mathcal 
S_{n,m}$ whenever $a\in\mathcal D_n,c\in\mathcal D_m$, by
\begin{equation}\label{InfDistance1}
\delta_\mathcal D(a,c)(b)=\left[\sup\left\{t\in[0,+\infty]\colon
\begin{bmatrix} 
a & sb\\
0 & c
\end{bmatrix}\in\mathcal D_{n+m}\text{ for all }s\in[0,t]\right\}
\right]^{-1},
\end{equation}
with the convention $1/0=+\infty$.
Observe first that $\delta_\mathcal D(a,c)$ is indeed well-defined
because noncommutative sets respect direct sums: 
$\begin{bmatrix} 
a & sb\\
0 & c
\end{bmatrix}\in\mathcal D_{n+m}$ at least for $s=0$. Second, 
$\delta_\mathcal D(a,c)(b)=0\iff\begin{bmatrix} 
a & sb\\
0 & c
\end{bmatrix}\in\mathcal D_{n+m}$ for all $s\in[0,+\infty)$.
Third, if $s_0\in(0,+\infty)$ is given, then, as indicated
in the definition, $\delta_\mathcal D(a,c)(s_0b)=s_0
\delta_\mathcal D(a,c)(b)$. Indeed, if $\delta_\mathcal D(a,c)(b)=0$
or $+\infty$, then the statement is obvious. Else, if
$\begin{bmatrix} 
a & sb\\
0 & c
\end{bmatrix}\in\mathcal D_{n+m}$ for all $s\in[0,
\delta_\mathcal D(a,c)(b)^{-1})$, then 
$\begin{bmatrix} 
a & sb\\
0 & c
\end{bmatrix}=\begin{bmatrix} 
a & \frac{s}{s_0}s_0b\\
0 & c
\end{bmatrix}$, so that $\begin{bmatrix} 
a & r(s_0b)\\
0 & c
\end{bmatrix}\in\mathcal D_{n+m}$ for all $r\in
[0,(s_0\delta_\mathcal D(a,c)(b))^{-1})$, which shows
that $s_0\delta_\mathcal D(a,c)(b)=\delta_\mathcal D(a,c)(s_0b)$.

\begin{remark}\label{rem-cont}
Given a complex vector space $\mathcal V$ endowed with a topology for  which the multiplication with
positive scalars is continuous (a requirement automatically satisfied by a topological vector space), 
the quantity $\delta$ is upper  semicontinuous in its three variables whenever it is defined on an nc 
set which satisfies property \eqref{propr1set} above. Indeed, consider such an nc set
$\Omega\subseteq\mathcal V_{\rm nc}$. It is enough to prove the statement
at level one. Thus, consider three nets $\{a_\iota\}_{\iota\in I},
\{c_\iota\}_{\iota\in I}$, and $\{b_\iota\}_{\iota\in I}$ converging to
$a,c\in\Omega_1$ and $b\in\mathcal V$, respectively. Let
$t\in(0,+\infty)$ be chosen so that $\begin{bmatrix}
a & sb \\
0 & c
\end{bmatrix}\in\Omega_2$ for all $s\leq t$. Since $\Omega_2$ is open in
the topology of $\mathcal V^{2\times 2}$, there exists an $\iota_0\in I$
such that $\begin{bmatrix}
a_\iota & [0,t]b_\iota \\
0 & c_\iota
\end{bmatrix}\subset\Omega_2$ for all $\iota\ge\iota_0$ (we have used
here the compactness of $[0,t]$). Thus, $t^{-1}>\delta(a,c)(b)$
implies that $t^{-1}>\delta(a_\iota,c_\iota)(b_\iota)$ for
all $\iota$ large enough. This implies that
\begin{equation}\label{limsup}
\limsup_{\iota\in I}\delta(a_\iota,c_\iota)(b_\iota)\leq\delta(a,c)(b),\quad a,c\in\Omega_1,b\in\mathcal V.
\end{equation}
This shows that $\delta$
is upper semicontinuous on nc sets that satisfy property \eqref{propr1set} under very
mild conditions on the topology of the underlying vector space.
Remarkably, under the supplementary hypothesis that the intersection
$\partial\Omega_{2k}\cap\begin{bmatrix}
a & \mathbb R_+b \\
0 & c
\end{bmatrix}$ is discrete for all $b\in\mathcal V^{k\times k}$,
$a,c\in\Omega_k$, the exact same argument applied to the
complement of $\Omega$ shows that $\delta$
is lower semicontinuous, and thus continuous.

\end{remark}

The following proposition is straightforward,
but, unless some of the hypotheses \eqref{propr1set} --
\eqref{propr3set} from above are assumed, it may well
be vacuous.

\begin{prop}\label{contr}
Let $\mathcal V,\mathcal W$ be two complex topological vector
spaces and let $\mathcal D$ and $\mathcal E$ be two noncommutative 
subsets of $\mathcal V_{\rm nc}$ and $\mathcal W_{\rm nc},$
respectively. Assume that $f\colon\mathcal D\to\mathcal E$
is a function such that 
\begin{enumerate}
\item[(a)] for any $a\in\mathcal D_n$, we have $f(a)\in\mathcal E_n$;
\item[(b)] $f$ respects direct sums;
\item[(c)] if $a\in\mathcal D_n,c\in\mathcal D_m$ and $b\in
\mathcal V^{n\times m}$ are such that $\begin{bmatrix} 
a & b\\
0 & c
\end{bmatrix}\in\mathcal D_{n+m},$ then there exists 
a function of three variables denoted $\Delta f(a,c)(b)$ 
such that $\Delta f(a,c)(tb)=t\Delta f(a,c)(b)$ for all $t\in[0,+\infty)$
with the property that $tb$ is in the domain of $\Delta f(a,c)(\cdot)$,
and $f$ satisfies 
$$
f\left(\begin{bmatrix} 
a & b\\
0 & c
\end{bmatrix}\right)=\begin{bmatrix} 
f(a) & \Delta f(a,c)(b)\\
0 & f(c)
\end{bmatrix}.
$$
\end{enumerate}
Then 
$$
\delta_{\mathcal D}(a,c)(b)\ge\delta_{\mathcal E}(f(a),f(c))
(\Delta f(a,c)(b)),\quad a\in\mathcal D_n,c\in\mathcal D_m,
b\in\mathcal V^{n\times m}.
$$
\end{prop}

Note that the hypothesis on the homogeneity of $\Delta f(a,c)(b)$ in 
$b$ is meaningful only if there exists some interval $(t,r)$ such that 
$\begin{bmatrix} 
a & sb\\
0 & c
\end{bmatrix}\in\mathcal D_{n+m}$ for all $s\in(t,r)$. Otherwise,
one can simply define $\Delta f(a,c)(sb)$ as $s\Delta f(a,c)(b)$.

\begin{proof}
The statement is tautological: consider $a\in\mathcal D_n,c\in
\mathcal D_m$ and $b\in\mathcal V^{n\times m}$ such that 
$\begin{bmatrix} 
a & sb\\
0 & c
\end{bmatrix}\in\mathcal D_{n+m}$ for all $s\in[0,t_0)$. If 
$t_0=+\infty$, then $\delta_{\mathcal D}(a,c)(b)=0$ and
$$
f\left(\begin{bmatrix} 
a & sb\\
0 & c
\end{bmatrix}\right)=\begin{bmatrix} 
f(a) & \Delta f(a,c)(sb)\\
0 & f(c)
\end{bmatrix}=\begin{bmatrix} 
f(a) & s\Delta f(a,c)(b)\\
0 & f(c)
\end{bmatrix}
$$
for all $s\in[0,+\infty)$, so that $\delta_{\mathcal E}(f(a),f(c))
(\Delta f(a,c)(b))=0.$
If $t_0=0$ (i.e. $\delta_{\mathcal D}(a,c)(b)=+\infty$), then the inequality
$\delta_{\mathcal D}(a,c)(b)\ge\delta_{\mathcal E}(f(a),f(c))
(\Delta f(a,c)(b))$ is obvious thanks to hypothesis (b). Finally, if $t_0=
\delta_{\mathcal D}(a,c)(b)^{-1}\in(0,+\infty)$, then 
$$
f\left(\begin{bmatrix} 
a & sb\\
0 & c
\end{bmatrix}\right)=\begin{bmatrix} 
f(a) & \Delta f(a,c)(sb)\\
0 & f(c)
\end{bmatrix}=\begin{bmatrix} 
f(a) & s\Delta f(a,c)(b)\\
0 & f(c)
\end{bmatrix}\in\mathcal E_{n+m}
$$
for all $s\in[0,t_0)$, which implies 
$t_0=\delta_{\mathcal D}(a,c)(b)^{-1}\leq\delta_{\mathcal E}
(f(a),f(c))(\Delta f(a,c)(b))^{-1}$. This concludes the proof.
\end{proof}

\begin{remark}
\begin{enumerate}
\item
If we assume hypotheses \eqref{propr1set} for $\mathcal D$,
then for any $a\in\mathcal D_n,c\in\mathcal D_m,b\in\mathcal V^{n
\times m}$ we are guaranteed that there exists a $t_0\in(0,+\infty]$
such that $\begin{bmatrix} 
a & sb\\
0 & c
\end{bmatrix}\in\mathcal D_{n+m}$ for all $s\in[0,t_0)$. Thus,
under a very mild assumption of openness in a complex topological
vector space, we are guaranteed that $\delta_{\mathcal D}(a,c)(b)$
is finite (possibly zero).
\item Assumption \eqref{propr2set} on $\mathcal D$ is sufficient 
(although not necessary) in order to guarantee that $\begin{bmatrix} 
a & zb\\
0 & c
\end{bmatrix}\in\mathcal D_{n+m}$ for all $z\in\mathbb C,$
$|z|<\delta_{\mathcal D}(a,c)(b)^{-1}$. Indeed, one simply
conjugates $\begin{bmatrix} 
a & sb\\
0 & c
\end{bmatrix}$ with the unitary $\begin{bmatrix} 
e^{i\theta/2}1_n & 0\\
0 & e^{-i\theta/2}1_m
\end{bmatrix}\in\mathbb C^{(n+m)\times(n+m)},$ where $\theta$ is the 
argument of $z$.
\item If, in Proposition \ref{contr}, the sets $\mathcal D$
and $\mathcal E$ are assumed to satisfy hypotheses 
\eqref{propr1set} and \eqref{propr2set}, and in addition 
$b\mapsto\Delta f(a,c)(b)$ satisfies $\Delta f(a,c)(zb)=
z\Delta f(a,c)(b)$, then we are guaranteed that the statement of
the proposition is not vacuous. In particular, 
\end{enumerate}
\end{remark}

\begin{cor}\label{trentatre}
If $f\colon\mathcal D\to\mathcal E$ is a
locally bounded noncommutative function on a finitely
open subset, then $f$ satisfies
$\delta_{\mathcal E}(f(a),f(c))
(\Delta f(a,c)(b))\le\delta_{\mathcal D}(a,c)(b),$ $a\in\mathcal D_n,c\in\mathcal D_m,
b\in\mathcal V^{n\times m},m,n\in\mathbb N.$
\end{cor}

Next, we study some of the properties of $\delta_\mathcal D$
in more detail.

\begin{lemma}\label{lemma3.3}
Assume that the noncommutative subset $\mathcal D$ of $\mathcal V_{\rm nc}$ satisfies
properties \eqref{propr2set} and \eqref{propr3set}.
For any unitary matrices $U\in\mathbb C^{n\times n},V\in\mathbb 
C^{m\times m}$, $a_1,a_2\in\mathcal D_n,c_1,c_2\in\mathcal D_m$, 
$b_{11}\in\mathcal V^{n\times n},b_{12}\in\mathcal V^{n\times m},
b_{21}\in\mathcal V^{m\times n},b_{22}\in\mathcal V^{m\times m}$,
we have
\begin{equation}\label{unit}
\delta_\mathcal D(Ua_1U^*,Vc_2V^*)(Ub_{12}V^*)=\delta_\mathcal D(a_1,
c_2)(b_{12})
\end{equation}
\begin{equation}\label{diagonal}
\delta_\mathcal D\left(\begin{bmatrix}
a_1 & 0\\
0 & c_1\end{bmatrix},\begin{bmatrix}
a_2 & 0\\
0 & c_2\end{bmatrix}\right)\begin{bmatrix}
b_{11} & 0\\
0 & b_{22}\end{bmatrix} =  \max\{\delta_\mathcal D(a_1,c_1)(b_{11}),
\delta_\mathcal D(a_2,c_2)(b_{22})\},
\end{equation}
\begin{equation}\label{counterdiagonal}
\delta_\mathcal D\left(\begin{bmatrix}
a_1 & 0\\
0 & c_1\end{bmatrix},\begin{bmatrix}
a_2 & 0\\
0 & c_2\end{bmatrix}\right)\begin{bmatrix}
0 & b_{12}\\
b_{21}& 0\end{bmatrix} =  \max\{\delta_\mathcal D(a_1,c_2)(b_{12}),
\delta_\mathcal D(c_1,a_2)(b_{21})\}.
\end{equation}
\end{lemma}

\begin{proof}
Relation \eqref{unit} follows trivially from hypothesis 
\eqref{propr2set}: for $s\ge0$, we have the chain of equivalences
\begin{eqnarray*}
\begin{bmatrix}
a_1 & sb_{12}\\
0 & c_2\end{bmatrix}\in\mathcal D_{n+m} & \iff &
\begin{bmatrix}
U & 0\\
0 & V\end{bmatrix}\begin{bmatrix}
a_1 & sb_{12}\\
0 & c_2\end{bmatrix}\begin{bmatrix}
U^* & 0\\
0 & V^*\end{bmatrix}\in\mathcal D_{n+m}\\
& \iff & \begin{bmatrix}
Ua_1U^* & sUb_{12}V^*\\
0 & Vc_2V^*\end{bmatrix}\in\mathcal D_{n+m}.
\end{eqnarray*}
A slight variation of this trick proves \eqref{diagonal} and 
\eqref{counterdiagonal}. 
Let 
$$
U_0=\begin{bmatrix}
0_{m\times n} & I_m & 0_{m\times n} & 0_{m\times m}\\
0_{n\times n} & 0_{n\times m} & I_n & 0_{n\times m} \\
I_n & 0_{n\times m} & 0_{n\times n} & 0_{n\times m} \\
0_{m\times n} & 0_{m\times m} & 0_{m\times n} & I_m\end{bmatrix},
$$
a complex $(2n+2m)\times(2n+2m)$ unitary matrix. For any $s\ge0$, we have
\begin{eqnarray*}
\begin{bmatrix}
a_1 & 0 & 0 & sb_{12}\\
0 & c_1 & sb_{21} & 0 \\
0 & 0 & a_2 & 0 \\
0 & 0 & 0 & c_2\end{bmatrix}\in\mathcal D_{2(n+m)} & \iff &
U_0\begin{bmatrix}
a_1 & 0 & 0 & sb_{12}\\
0 & c_1 & sb_{21} & 0 \\
0 & 0 & a_2 & 0 \\
0 & 0 & 0 & c_2\end{bmatrix}U_0^*\in\mathcal D_{2(n+m)}\\
& \iff & \begin{bmatrix}
c_1 & sb_{21} & 0 & 0\\
0 & a_2 & 0 & 0\\
0 & 0 & a_1 & sb_{12}\\
0 & 0 & 0 & c_2\end{bmatrix}\in\mathcal D_{m+n+n+m}\\
& \iff & \begin{bmatrix}
c_1 & sb_{21}\\
0 & a_2\end{bmatrix},\begin{bmatrix}
a_1 & sb_{12}\\
0 & c_2\end{bmatrix}\in\mathcal D_{n+m},
\end{eqnarray*}
where we have 
used property \eqref{propr3set} in the last equivalence
and property \eqref{propr2set} in the first.
This proves \eqref{counterdiagonal}. Relation
\eqref{diagonal} is proved the same way.
\end{proof}

The next lemma shows that, in a certain way, $\delta_\mathcal D$
is itself a sort of noncommutative function.

\begin{lemma}\label{lemma3.4}
Assume that $\mathcal D\subseteq\mathcal V_{\rm nc}$ satisfies properties 
\eqref{propr2set} and \eqref{propr3set}. For any $n,m\in\mathbb N$, 
$a\in\mathcal D_n,c\in\mathcal D_m,b\in\mathcal V^{n\times m}$,
and any $k\in\mathbb N$, we have
$$
\delta_\mathcal D(I_k\otimes a,I_k\otimes c)(Z\otimes b)=
\delta_\mathcal D(a,c)(b)\|ZZ^*\|^\frac12,\quad Z\in\mathbb C^{k\times
k}.
$$
\end{lemma}

\begin{proof}
We shall prove this lemma in two steps. In the first step, we assume 
that $a=c$ (and implicitly $m=n$). Consider unitary
matrices $U,V^*\in\mathbb C^{k\times k}$ which diagonalize $Z$:
$UZV^*=\mathrm{diag}(\lambda_1,\dots,\lambda_k)$, where
$0\le\lambda_1\le\dots\le\lambda_k=\|Z^*Z\|^\frac12$ are
the singular values of $Z$. 
Then 
$$
\begin{bmatrix}
U\otimes 1 & 0\\
0 & V\otimes 1\end{bmatrix}
\begin{bmatrix}
I_k\otimes a & Z\otimes b\\
0 & I_k\otimes a\end{bmatrix}
\begin{bmatrix}
U^*\otimes 1 & 0\\
0 & V^*\otimes 1\end{bmatrix}=
\begin{bmatrix}
I_k\otimes a & UZV^*\otimes b\\
0 & I_k\otimes a\end{bmatrix}.
$$
Thus, by property \eqref{propr2set}, $\begin{bmatrix}
I_k\otimes a & sZ\otimes b\\
0 & I_k\otimes a\end{bmatrix}\in\mathcal D_{2kn}$ if and only if
the matrix
$$
\begin{bmatrix}
a & 0 & \cdots & 0 & s\lambda_1b & 0 & \cdots & 0\\
0 & a & \cdots & 0 & 0 & s\lambda_2b & \cdots & 0\\
\vdots & \vdots & \vdots & \vdots & \vdots & \vdots & \vdots & \vdots\\
0 & 0  & \cdots & a & 0 & 0 & \cdots & s\lambda_kb\\
0 & 0  & \cdots & 0 & a & 0 & \cdots & 0\\
0 & 0  & \cdots & 0 & 0 & a & \cdots & 0\\
\vdots & \vdots & \vdots & \vdots & \vdots & \vdots & \vdots & \vdots\\
0 & 0 & 0 & \cdots & 0 & 0 & \cdots & a
\end{bmatrix}\in\mathcal D_{2kn}.
$$
Successive permutations transform this into the condition
$$
\begin{bmatrix}
a & s\lambda_1b & \cdots & 0 & 0  & \cdots & 0 & 0\\
0 & a & \cdots & 0 & 0 & \cdots & 0 & 0\\
\vdots & \vdots & \vdots & \vdots & \vdots & \vdots & \vdots & \vdots\\
0 & 0  & \cdots & a & s\lambda_jb & \cdots & 0 & 0\\
0 & 0  & \cdots & 0 & a & \cdots & 0 & 0\\
\vdots & \vdots & \vdots & \vdots & \vdots & \vdots & \vdots & \vdots\\
0 & 0  & \cdots & 0 & 0 & \cdots & a & s\lambda_kb\\
0 & 0 & 0 & \cdots & 0 & \cdots & 0 & a
\end{bmatrix}\in\mathcal D_{2kn},
$$
i.e.
$$
\mathrm{diag}\left(
\begin{bmatrix}
a & s\lambda_1b\\
0 & a 
\end{bmatrix},\dots,
\begin{bmatrix}
a & s\lambda_kb\\
0 & a 
\end{bmatrix}
\right)\in\mathcal D_{2kn}.
$$
By property \eqref{propr3set} we have that this happens if
and only if each block 
$\begin{bmatrix}
a & s\lambda_jb\\
0 & a 
\end{bmatrix}
$ belongs to $\mathcal D_{2n}$. Since the largest singular
value $\lambda_k$ of $Z$
equals $\|Z^*Z\|^\frac12,$ the first step is proved.

In order to prove the second step, we use equation \eqref{counterdiagonal} 
of Lemma \ref{lemma3.3}, which guarantees that
$$
\delta_\mathcal D(I_k\otimes a,I_k\otimes c)(Z\otimes b)=
\delta_\mathcal D\left(\begin{bmatrix}
I_k\otimes a & 0\\
0 & I_k\otimes c 
\end{bmatrix},\begin{bmatrix}
I_k\otimes a & 0\\
0 & I_k\otimes c
\end{bmatrix}
\right)\left(\begin{bmatrix}
0 & Z\otimes b\\
0 & 0
\end{bmatrix}\right).
$$
By conjugating with a permutation matrix, it follows, again
via Lemma \ref{lemma3.3} and the first step, that 
\begin{eqnarray*}
\lefteqn{\delta_\mathcal D\left(\begin{bmatrix}
I_k\otimes a & 0\\
0 & I_k\otimes c 
\end{bmatrix},\begin{bmatrix}
I_k\otimes a & 0\\
0 & I_k\otimes c
\end{bmatrix}
\right)\left(\begin{bmatrix}
0 & Z\otimes b\\
0 & 0
\end{bmatrix}\right)}\\
& = & \delta_\mathcal D\left(I_k\otimes\begin{bmatrix}
a & 0\\
0 & c 
\end{bmatrix},I_k\otimes \begin{bmatrix}
a & 0\\
0 & c
\end{bmatrix}
\right)\left( Z\otimes\begin{bmatrix}
0 & b\\
0 & 0
\end{bmatrix}\right)\\
& = & \delta_\mathcal D\left(\begin{bmatrix}
a & 0\\
0 & c 
\end{bmatrix},\begin{bmatrix}
a & 0\\
0 & c
\end{bmatrix}
\right)\left(\begin{bmatrix}
0 & b\\
0 & 0
\end{bmatrix}\right)\|Z^*Z\|^\frac12\\
& = & \delta_\mathcal D\left(a,c
\right)(b)\|Z^*Z\|^\frac12.
\end{eqnarray*}
\end{proof}

With these two lemmas, we can prove now the main result of
this section. For simplicity, denote
$$
\tilde\delta_\mathcal D(a,c):=\delta_\mathcal D(a,c)(a-c).
$$

\begin{thm}\label{sep}
Assume that $\mathcal D\subseteq\mathcal V_{\rm nc}$ satisfies properties 
\eqref{propr2set} and \eqref{propr3set}. 
The following statements are equivalent for any $m,n\in\mathbb N$:
\begin{enumerate}
\item[(i)] For any $a\in\mathcal D_n,c\in\mathcal D_m$, $\delta_\mathcal D(a,c)(b)=0\implies b=0$;
\item[(ii)] For any $a\in\mathcal D_n$, $\delta_\mathcal D(a,a)(b)=0\implies b=0$;
\item[(iii)] For any $a,c\in\mathcal D_n$, $\tilde\delta_\mathcal D(a,c)=0\implies a=c$;
\item[(iv)] For any $k\in\mathbb N$, there exists no non-constant 
noncommutative function $f\colon\mathbb C_{\rm nc}\to(\mathcal D_k)_{\rm nc}$.
\end{enumerate}
By $(\mathcal D_k)_{\rm nc}$ we denote all the levels of $\mathcal D$
which are multiples of $k$.
\end{thm}

\begin{proof}
Implications (i)$\implies$(ii) and (i)$\implies$(iii) are obvious.
If there exists a noncommutative $f\colon\mathbb C_{\rm nc}\to(\mathcal 
D_k)_{\rm nc}$ for some $k\in\mathbb N$, then, by Proposition
\ref{contr}, it follows that for any $n\in\mathbb N$, 
$Z,W\in\mathbb C^{n\times n}$, 
$\tilde\delta_\mathcal D(f(Z),f(W))\leq\tilde\delta_{\mathbb C_{\rm nc}}(Z,W)=0$.
Thus, if $f$ is not constant, (iii) is violated. Thus, 
(iii)$\implies$(iv).
The implication (ii)$\implies$(i) follows from equation 
\eqref{counterdiagonal} of Lemma \ref{lemma3.3}
by writing
$$
\delta_\mathcal D\left(\begin{bmatrix}
a & 0\\
0 & c 
\end{bmatrix},\begin{bmatrix}
a & 0\\
0 & c
\end{bmatrix}
\right)\left(\begin{bmatrix}
0 & b\\
0 & 0
\end{bmatrix}\right)= \delta_\mathcal D\left(a,c\right)(b).
$$
Finally, to prove (iv)$\implies$(ii), assume that we found $a_0\in
\mathcal D_{n_0}$ and $b_0\in\mathcal V^{n_0\times n_0}\setminus\{0\}$
such that $\delta_\mathcal D(a_0,a_0)(b_0)=0$.
We build the linear noncommutative function
$f(Z)=\begin{bmatrix}
I_p\otimes a_0 & 0\\
0 & I_p\otimes a_0
\end{bmatrix}+Z\otimes\begin{bmatrix}
0 & b_0\\
0 & 0
\end{bmatrix}$, $Z\in\mathbb C^{p\times p}$.
By a conjugation with a permutation matrix and an application of 
Lemma \ref{lemma3.4}, we conclude that $f$ takes values in
$\mathcal D_{2pn_0}$, so that (iv) does not hold. This completes
the proof.
\end{proof}

The function $\tilde\delta_\mathcal D$ allows us to define two distances
(possibly degenerate) on $\mathcal D$, by mimicking the
definition of the Kobayashi distance, with $\tilde\delta_\mathcal D$
playing the role of Lempert function. 
\begin{defn}\label{Kobayashi}
If $\mathcal D$ is a noncommutative set in an operator space 
satisfying assumptions \eqref{propr2set} and \eqref{propr3set},
then for any $n\in\mathbb N,a,c\in\mathcal D_n$,
$$
\tilde{d}_\mathcal D(a,c)=\inf\left\{\sum_{j=1}^N\tilde\delta_\mathcal D
(a_{j-1},a_j)\colon a_j\in\mathcal D_n,0\le j\le N,a_0=a,a_N=c,N\in
\mathbb N\right\}.
$$
Wel call such a finite sequence $a=a_0,a_1,\dots,a_N=c$
a {\em division} of $\tilde{d}_\mathcal D(a,c)$.
\end{defn}
The function $\tilde{d}_\mathcal D\colon\mathcal D\times\mathcal D\to[0,+
\infty]$ fails to separate the points of $\mathcal D$ if one (and hence all)
the conditions of Theorem \ref{sep} are satisfied.

It is quite easy to show that $\tilde{d}$ is a distance.
Indeed, since $\tilde\delta_\mathcal D(a,c)=\tilde\delta_\mathcal D(c,a)$,
it follows that $\tilde{d}_\mathcal D(a,c)=\tilde{d}_\mathcal D(c,a).$
So only the triangle inequality remains to be proved.
Let $a,c,v\in\mathcal D_n$. If 
$a_0=a,a_1,\dots,a_N=c$ and $a_{N}=c,\dots,a_{N+p-1},
a_{N+p}=v$ are divisions for $\tilde{d}_\mathcal D(a,c)$ and
$\tilde{d}_\mathcal D(c,v)$, respectively, then 
$a_0=a,a_1,\dots,a_N,a_{N+1},\dots,
a_{N+p}=v$ is a division for $\tilde{d}_\mathcal D(a,v)$.
In particular, 
\begin{eqnarray*}
\lefteqn{\sum_{j=1}^{N}\tilde\delta_\mathcal D
(a_{j-1},a_j)+\sum_{j=N+1}^{N+p}\tilde\delta_\mathcal D
(a_{j-1},a_j)}\\
& \ge & \inf\left\{\sum_{j=1}^M\tilde\delta_\mathcal D
(d_{j-1},d_j)\colon d_0,\dots,d_j\text{ division of }\tilde{d}_\mathcal D(a,v)
\right\}
\end{eqnarray*}
for all divisions $a_0=a,a_1,\dots,a_N=c$ for $\tilde{d}_\mathcal D(a,c)$
and $a_{N}=c,\dots,a_{N+p-1},a_{N+p}=v$ for
$\tilde{d}_\mathcal D(c,v)$. Taking infimum separately
after each division provides
$$
\tilde{d}_\mathcal D(a,c)+\tilde{d}_\mathcal D(c,v)\ge \tilde{d}_\mathcal D(a,v).
$$

A slight drawback of $\tilde{d}_\mathcal D$ is that it may depend to a 
certain extent on the level $n$ in which the points $a$ and $c$ live. 
The following modification of Definition \ref{Kobayashi} seems to be 
more natural in the noncommutative framework.
\begin{defn}\label{Kobayashi8}
If $\mathcal D$ is a noncommutative set in an operator space 
satisfying assumptions \eqref{propr2set} and \eqref{propr3set},
then for any $n\in\mathbb N,a,c\in\mathcal D_n$,
\begin{eqnarray*}
\lefteqn{\tilde{d}_{\mathcal D,\infty}(a,c)=}\\
& & \inf\left\{\sum_{j=1}^N\tilde\delta_\mathcal D
(a_{j-1},a_j)\colon a_j\in\mathcal D_{np},0\le j\le N,a_0=I_p\otimes a,a_N=I_p\otimes c,N,p\in
\mathbb N\right\}.
\end{eqnarray*}
\end{defn}
Clearly, according to Lemma \ref{lemma3.4}, $\tilde{d}_{\mathcal D,\infty}(a,c)
\leq\tilde{d}_{\mathcal D}(a,c)$. At this moment we do not know for what 
sets $\mathcal D$ the two distances are equal. It will be seen below that
they coincide on the unit ball of a $C^*$-algebra.

The most general version of the Schwarz-Pick Lemma tells us that an 
analytic map between two hyperbolic domains is a contraction with
respect to the corresponding Kobayashi metrics. The following
corollary is a direct consequence of Proposition \ref{contr} and
the above definition.

\begin{cor}
Let $\mathcal D,\mathcal E$ be two noncommutative 
sets satisfying assumptions \eqref{propr2set} and \eqref{propr3set}.
Let $f\colon\mathcal D\to\mathcal E$ be a noncommutative
function. Then $f$ is a contraction with respect to the above-defined
distances:
$$
\tilde{d}_{\mathcal E}(f(a),f(c))\leq\tilde{ d}_{\mathcal D}(a,c)\ \text{ and }\ 
\tilde{d}_{\mathcal E,\infty}(f(a),f(c))\leq\tilde{d}_{\mathcal D,\infty}(a,c),\quad
a,c\in\mathcal D_n, n\in\mathbb N.
$$
\end{cor}
Note that assuming also hypothesis \eqref{propr1set} in the above
corollary guarantees that the two sides of the inequality above are 
both finite (possibly zero).

Until now we have made no assumptions on the openness of $\mathcal D$.
As seen in Remark \ref{rem-cont}, hypotheses
\eqref{propr1set} --- \eqref{propr3set} guarantee 
that $\delta_\mathcal D$ is upper semicontinuous in its three variables,
and in particular so is $\tilde{\delta}_\mathcal D$. 
Thus, we may define infinitesimal versions of $\tilde{d}_\mathcal D$ and $\tilde{d}_{\mathcal D,\infty}$. 
\begin{defn}\label{inf-Kobayashi}
If $\mathcal D$ is a noncommutative set in an operator space 
satisfying assumptions \eqref{propr1set}---\eqref{propr3set},
then for any $n\in\mathbb N,a,c\in\mathcal D_n$,
\begin{eqnarray*}
\lefteqn{
{d}_\mathcal D(a,c)=\inf\left\{\int_{[0,1]}\delta({\bf a}(t),{\bf a}
(t))({\bf a}'(t))\,{\rm d}t:\right.}\\
& & \left.\frac{}{}{\bf a}\colon[0,1]\to\mathcal D_n\text{ continuously
differentiable, }{\bf a}(0)=a,{\bf a}(1)=c
\right\}.
\end{eqnarray*}
Similarly,
\begin{eqnarray*}
\lefteqn{
{d}_{\mathcal D,\infty}(a,c)=\inf\left\{\int_{[0,1]}\delta({\bf a}(t),{\bf a}
(t))({\bf a}'(t))\,{\rm d}t:\right.}\\
& & \left.p\in\mathbb N,\frac{}{}{\bf a}\colon[0,1]\to\mathcal D_{np}\text{ continuously
differentiable, }{\bf a}(0)=I_p\otimes a,{\bf a}(1)=I_p\otimes c
\right\}.
\end{eqnarray*}
\end{defn}
The openness of $\mathcal D_n$ implies that $\delta({\bf a}(t),{\bf a}
(t))({\bf a}'(t))$ is finite for all $t\in[0,1]$. Since an upper semicontinuous function
attains its supremum on a compact set, it follows that the set $\{\delta({\bf a}(t),{\bf a}
(t))({\bf a}'(t))\colon t\in[0,1]\}$ is bounded. Thus, the integrals defining $d_\mathcal D$
are finite, so that $d_\mathcal D$ is well-defined and finite (possibly zero).
Clearly ${d}_\mathcal D(a,c)\geq{d}_{\mathcal D,\infty}(a,c)$. 
Both $d_\mathcal D$ and ${d}_{\mathcal D,\infty}$ are (possibly degenerate) distances: as before,
it is only the triangle inequality that needs to be verified. If $a,v,c\in\mathcal D_n$, then
the above infimum over all paths from $a$ to $c$ is necessarily no greater than the
infimum over all paths from $a$ to $c$ which go through $v$. Since $\delta_\mathcal D$
is continuous and paths which are continuous and differentiable everywhere except at one point
can be approximated arbitrarily well by paths which are differentiable everywhere, it follows 
that ${d}_\mathcal D(a,c)\leq{d}_\mathcal D(a,v)+{d}_\mathcal D(v,c)$, and the same for 
${d}_{\mathcal D,\infty}$.

Another application of Proposition \ref{contr} shows that noncommutative functions are contractions also 
with respect to ${d}_\mathcal D$ and ${d}_{\mathcal D,\infty}$. We record this fact below.

\begin{cor}
Let $\mathcal D,\mathcal E$ be two noncommutative 
sets satisfying assumptions \eqref{propr1set}---\eqref{propr3set}.
Let $f\colon\mathcal D\to\mathcal E$ be a noncommutative
function. Then $f$ is a contraction with respect to the above-defined
distances:
$$
{d}_{\mathcal E}(f(a),f(c))\leq{ d}_{\mathcal D}(a,c)\ \text{ and }\ 
{d}_{\mathcal E,\infty}(f(a),f(c))\leq{ d}_{\mathcal D,\infty}(a,c),\quad
a,c\in\mathcal D_n, n\in\mathbb N.
$$
\end{cor}

We establish next the relation between $\tilde{d}_\mathcal D$ and $d_\mathcal D$ under the assumptions
\eqref{propr1set} --- \eqref{propr3set}. As a direct consequence of the upper semicontinuity of 
$\delta$ (Remark \ref{rem-cont}), we obtain for any differentiable path ${\bf a}$ defined on $[0,1]$
and any $t\in[0,1]$ the relation 
$$
\limsup_{h\to0}\delta_\mathcal D({\bf a}(t),{\bf a}(t+h))\left(\frac{{\bf a}(t+h)-{\bf  a}(t)}{h}\right)\leq
\delta_\mathcal D({\bf a}(t),{\bf a}(t))({\bf a}'(t)).
$$
(When $t=0$ or $t=1$, the limit should of course be taken one-sided.) 
In particular given an arbitrary path ${\bf a}$, a division of
$[0,1]$ translates into a division of $\tilde{d}(a,c)$. Given $\varepsilon>0$,
for any $t\in[0,1]$ there exists $\eta_{t,\varepsilon}>0$ such that 
$\delta_\mathcal D({\bf a}(t),{\bf a}(t+h))\left(\frac{{\bf a}(t+h)-{\bf  a}(t)}{h}\right)<
\delta_\mathcal D({\bf a}(t),{\bf a}(t))({\bf a}'(t))+\varepsilon$ for any $|h|<\eta_{t,\varepsilon}$.
The family $\{(t-\eta_{t,\varepsilon},t+\eta_{t,\varepsilon})\}_{0\le t\le1}$
is an open cover of $[0,1]$, so that we may extract a finite subcover
$(t_1-\eta_{t_1,\varepsilon},t_1+\eta_{t_1,\varepsilon}),\dots,
(t_N-\eta_{t_N,\varepsilon},t_N+\eta_{t_N,\varepsilon})$, $t_1<\cdots<t_N$. Let 
$t_0=0,t_{N+1}=1$.  
By choosing the smallest among $\eta_{t_j,\varepsilon}$, $1\le j\le N$, and increasing the number of 
points $t_j$ if necessary, we may assume that $\eta_{t_1,\varepsilon}=\cdots=\eta_{t_N,\varepsilon}
=\eta_\varepsilon>0$ and $t_j\in(t_{j-1}-\eta_\varepsilon,t_{j-1}+\eta_\varepsilon)\cap
(t_{j+1}-\eta_\varepsilon,t_{j+1}+\eta_\varepsilon)$. Then
\begin{eqnarray}
\tilde{d}_{\mathcal D}(a,c)& \leq & \sum_{j=0}^N\tilde{\delta}_{\mathcal D}({\bf a}(t_j),{\bf a}(t_{j+1}))
\label{17}\\
& = & \sum_{j=0}^N(t_{j+1}-t_j)\nonumber
{\delta}_{\mathcal D}({\bf a}(t_j),{\bf a}(t_{j+1}))\left(\frac{{\bf a}(t_{j+1})-{\bf a}(t_j)}{t_{j+1}-t_j}\right)\\
& < & \sum_{j=0}^N(t_{j+1}-t_j)\delta_\mathcal D({\bf a}(s_j),{\bf a}(s_j))({\bf a}'(s_j))+\varepsilon\quad
\ (s_j\in[t_j,t_{j+1}])\nonumber\\
& \le & \sum_{j=0}^N(t_{j+1}-t_j)\int_{[t_j,t_{j+1}]}\delta_\mathcal D({\bf a}(t),{\bf a}(t))({\bf a}'(t))\,{\rm 
d}t+\varepsilon\label{18}\\
& = & \int_{[0,1]}\delta({\bf a}(t),{\bf a}
(t))({\bf a}'(t))\,{\rm d}t+\varepsilon.\nonumber
\end{eqnarray}
We have used in \eqref{17} the definition of $\tilde{d}_\mathcal D$, and in relation \eqref{18}
the fact that we may choose $s_j$ arbitrarily in $[t_j,t_{j+1}]$, and we decide to choose 
an $s_j$ such that 
$$
\delta_\mathcal D({\bf a}(s_j),{\bf a}(s_j))({\bf a}'(s_j))\leq 
\frac{1}{t_{j+1}-t_j}\int_{[t_j,t_{j+1}]}\delta_\mathcal D({\bf a}(t),{\bf a}(t))({\bf a}'(t))\,{\rm 
d}t.
$$
Since ${\bf a}$ has been arbitrarily chosen, it follows that $\tilde{d}_\mathcal D(a,c)\leq
d_\mathcal D(a,c)$ for all $a,c$ belonging to the same level of $\mathcal D$. Thus,
\begin{equation}\label{ds}
\tilde{d}_\mathcal D\leq
d_\mathcal D\quad \text{for all }\mathcal D\subset\mathcal V_{\rm nc}\text{ satisfying hypotheses }
\eqref{propr1set} - \eqref{propr3set}.
\end{equation}
All of the above arguments hold when $\tilde{d}_\mathcal D$ is replaced by 
$\tilde{d}_{\mathcal D,\infty}$ and ${d}_\mathcal D$ by ${d}_{\mathcal D,\infty}$. We
record the result below:
\begin{equation}\label{ds8}
\tilde{d}_{\mathcal D,\infty}\leq
d_{\mathcal D,\infty}\quad \text{for all }\mathcal D\subset\mathcal V_{\rm nc}\text{ satisfying hypotheses }
\eqref{propr1set} - \eqref{propr3set}.
\end{equation}

Since we have shown that $\delta$ and $\tilde\delta$ generate distances, it is 
natural to ask what topology one may expect those distances to determine on the original space.
In the most general case, we are able to make only the following statement:

\begin{prop}\label{dreiwolf}
Assume that $\mathcal D$ is a noncommutative subset of a topological vector space
$\mathcal V$ which satisfies assumptions \eqref{propr1set}---\eqref{propr3set}.  
If $n\in\mathbb N$ is
given and a subset $A$ of $\mathcal D_n$ is open in the topology generated by $\tilde{d}_\mathcal D$,
then it is open in the product topology induced by $\mathcal V$ on $\mathcal V^{n\times n}$.
\end{prop}
\begin{proof}
Assume that $a\in\mathcal D_n$ is given. For any net $\{a_\iota\}_{\iota\in I}\subseteq\mathcal D_n$
which converges to $a$ in the product topology of $\mathcal V^{n\times n}$, we have by Remark 
\ref{rem-cont} that
$$
0=\tilde\delta_\mathcal D(a,a)\geq\limsup_{\iota\in I}\tilde\delta_\mathcal D(a_\iota,a)
\geq\limsup_{\iota\in I}\tilde{d}_\mathcal D(a_\iota,a)\ge0.
$$
Thus, $\lim_{\iota\in I}\tilde{d}_\mathcal D(a_\iota,a)=0$ whenever $\{a_\iota\}_{\iota\in I}\subseteq
\mathcal D_n$ converges to $a$ in the product topology of $\mathcal V^{n\times n}$. This completes
our proof.
\end{proof}



\section{A smooth (pseudo)metric}
\label{sec:four}

We have seen above that simple properties of noncommutative
sets allow us to define a distance which is often nondegenerate,
and with respect to which analytic noncommutative functions
are natural contractions. These results have a ``metric space''
flavour. In this section we consider the case when the distance
defined has a ``differential geometry'' flavour.

\subsection{Hypotheses}\label{SecHyp}
Let $\mathcal V$ be an operator space and $\mathcal{J,K}$ be
${C}^*$-algebras. Let $\mathcal O_{\rm nc}
\subseteq\mathcal V_{\rm nc}$ be a noncommutative set,
and assume that $G\colon\mathcal O_{\rm nc}\times\mathcal O_{\rm nc}
\to\mathcal L(\mathcal J , \mathcal K)$ is an affine noncommutative
kernel. Recall that if $(a,c)\mapsto G(a,c)$ is an affine nc kernel,
then $(a,c)\mapsto G(a,c^*)$ is an nc kernel. We prefer to work 
with the affine kernel $G$ because we will often need to take its
derivative (or, rather, difference-differential) on both the first and second coordinate
(which we denote by ${}_0\Delta G(a;a',c)$ and ${}_1\Delta G(a,c;c')$,
respectively). Consider the following properties:
\begin{enumerate}
\item $\mathcal O_{\rm nc}$ is uniformly open and $G$ is locally uniformly
bounded. Thus, $G$ is uniformly analytic in each of its two variables.

\item $\mathcal O_{\rm nc}$ is finitely open and $G$ is locally
bounded on slices. Thus, $G$ is analytic on slices in each of its two variables.

\item  $\mathcal O_{\rm nc}$ is open in the level topology and $G$ is locally
bounded on slices. Thus, $G$ is analytic on slices in each of its two variables.

\item For any $n\in\mathbb N$ and $a,c
\in\mathcal O_n$ such that $a^*,c^*\in\mathcal O_n$, we have 
$G(a,c)(v)^*=G(c^*,a^*)(v)$ for $v=v^*\in\mathcal J^{n\times n}$;

\item $\{a\in \mathcal O_{\rm nc}\colon G(a,a^*)(1)>0\}\neq
\varnothing.$ 

\item At each level at which the set $\{a\in \mathcal O_{\rm nc}
\colon G(a,a^*)(1)>0\}$ is nonempty, we have $\|G(a,a^*)(1)^{-1}\|\to+
\infty$ as $a$ tends to the norm-topology boundary of $\{a\in 
\mathcal O_{\rm nc}\colon G(a,a^*)(1)>0\}$.

\item Let $\Omega$ be a connected component of $\{a\in\mathcal O_{\rm 
nc}\colon G(a,a^*)(1)>0\}$. For any given $a\in\Omega_{n}$, $c\in
\overline{\Omega}_n$, we have $G(a,c^*)(1)$ invertible as an element
in the ${C}^*$-algebra $\mathcal K^{n\times n}$.

\item The function $G$ is analytic on a neighbourhood of $\Omega_n
\times\Omega_n$ for each $n\in\mathbb N$.
\end{enumerate}

In our results below, we will assume various subsets of the
above hypotheses. We would like to emphasize at this moment already
that they are not very restrictive, and important families of kernels 
satisfy all of them.

Pick a point $a_0\in\{a\in\mathcal O_{\rm nc}\colon G(a,a^*)(1)>0\}$
at the first nonempty level. Let $\mathcal D_{G,\rm nc}$ be the
connected component of $a_0$ (i.e. at each multiple $k$ of the level
in which $a_0$ occurs, we consider the connected component of 
$a_0\otimes 1_k$). In all applications we are currently aware of,
the set $\mathcal O_{\rm nc}$ is considerably bigger than 
$\mathcal D_{G,\rm nc}$. It seems in fact that at the present level
of knowledge in this field, analyticity of $G$ on the boundary of
$\mathcal D_{G,\rm nc}$ is necessary in order to obtain powerful
results about arbitrary functions defined on it. Given the case of
single-variable analytic functions, that is probably not so 
surprising. However, for the purposes of the next section, this 
hypothesis is not needed.

We would like to emphasize that if $G(a,a^*)$ is completely 
positive, then the condition $G(a,a^*)(1)>0$ can be replaced by 
the condition $G(a,a^*)(x)>0$ for any $x>0$.
Indeed, one implication is obvious. Conversely, if $x>0$, then it
is invertible and $x\ge\|x^{-1}\|^{-1}1>0$, so that 
$0<\|x^{-1}\|^{-1}G(a,a^*)(1)\le G(a,a^*)(x)$. 
For our purposes, completely positive kernels are ``bad'': 
they generate a degenerate pseudometric. However,
in the following, to the extent possible, we shall perform our 
computations in such a way as to be able to draw conclusions for both 
the case $G(a,a^*)$ completely positive and $G(a,a^*)(1)>0$ (without 
the assumption that $G(a,a^*)$ is positive).

\subsection{The smooth pseudometric}
The following proposition gives a noncommutative version of 
a hyperbolic pseudometric. This version is given in terms of the 
defining functions of the domains in question and its definition
is purely algebraic. It is clear that noncommutative domains admit
hyperbolic pseudometrics level-by-level. However, there would be
apropri no reason to think that they are related to the pseudometric
we define here. We will see later that in some cases our pseudometric
indeed generates the Kobayashi metric, while in others it does not.
As in Theorem \ref{sep}, and as in the classical theory of several 
complex variables, for the pseudometric to be nondegenerate, 
it is necessary that the domains do not contain holomorphic 
images of complex lines (i.e. copies of $\mathbb C$) at any level.

Consider $G$ satisfying properties [(1), (2) or (3)], (4), and (5), and define 
$\mathcal D_{G,\rm nc}$ as above. Without loss of generality, we assume 
that $\mathcal D_{G,1}\neq\varnothing.$ Recall that the spectrum of 
an operator $V$ on a Hilbert space is denoted by $\sigma(V)$. For 
$a\in\mathcal D_{G,n},c\in\mathcal D_{G,m},b\in\mathcal 
V^{n\times m}$, we have 
\begin{eqnarray}
\lefteqn{\delta_{\mathcal D_{G,\rm nc}}(a,c)(b)}\nonumber\\
& = & \max\left\{0,\ \sup\sigma\left(G(a,a^*)(1)^{-1/2}\left[
{}_0\Delta G(a;c,c^*)(b,1)
\left[G(c,c^*)(1)\right]^{-1}\right.\right.\right.\nonumber\\
& & \left.\left.\left.\frac{}{}\mbox{}\times{}_1\Delta G(c,c^*;a^*)
(1,b^*)-{}_1\Delta {}_0\Delta G(a;c,c^*;a^*)(b,1,b^*)
\right]G(a,a^*)(1)^{-1/2}\right)\right\}^\frac12,
\label{Infinitesimal}
\end{eqnarray}
and, when $m=n$, 
\begin{equation}
\label{Distance}
\tilde\delta_{\mathcal D_{G,\rm nc}}(a,c)=\delta_{\mathcal D_{G,\rm nc}}
(a,c)(a-c).
\end{equation}
It will be seen below that 
\begin{eqnarray*}
\tilde\delta_{\mathcal D_{G,\rm nc}}(a,c) & = & \max\left\{0,\ 
\sup\sigma\left(G(a,a^*)(1)^{-\frac12}G(a,c^*)
(1)G(c,c^*)(1)^{-1}\right.\right.\\
& & \left.\left.\mbox{}\times
G(c,a^*)(1)G(a,a^*)(1)^{-\frac12}-1\right)
\right\}^\frac12.\nonumber
\end{eqnarray*}
It will be apparent that these two objects coincide 
with the ones defined in Section \ref{pseudodistance}
for the particular case of domains defined via 
inequalities of the type described in hypothesis
(5) above.

Consider another function $H$ defined on some 
noncommutative subset of $\mathcal W_{\rm nc}$,
which satisfies the same properties as $G$. We
define $\mathcal D_{H,\rm nc}$ the same
way as $\mathcal D_{G,\rm nc}$

\begin{prop}\label{prop:3.1}
Let $\mathcal D_{G,\rm nc},\mathcal D_{H,\rm nc}$ be two
domains defined as above. Let $f\colon\mathcal D_{G,\rm nc}\to
\mathcal D_{H,\rm nc}$ be a noncommutative map. 
For any $n,m\in\mathbb N$, $a\in
\mathcal D_{G,n},c\in\mathcal D_{G,m},b\in\mathcal 
V^{n\times m}$, we have
\begin{equation}\label{deriv}
\delta_{\mathcal D_{H,\rm nc}}(f(a),f(c))(\Delta f(a,c)(b))
\leq\delta_{\mathcal D_{G,\rm nc}}(a,c)(b).
\end{equation}
If $m=n$, then
$$
\tilde\delta_{\mathcal D_{H,\rm nc}}(f(a),f(c))
\leq\tilde\delta_{\mathcal D_{G,\rm nc}}(a,c).
$$
and
\begin{eqnarray}
\lefteqn{\left\|H(f(a),f(a)^*)(1)^{-\frac12}
H(f(a),f(c)^*)(1)H(f(c),f(c)^*)(1)^{-1}H(f(c),f(a)^*)(1)\right.}
\nonumber\\
& & \left.\mbox{}\times H(f(a),f(a)^*)(1)^{-\frac12}\right\|\nonumber\\
& \leq & \left\|G(a,a^*)(1)^{-\frac12}G(a,c^*)
(1)G(c,c^*)(1)^{-1}G(c,a^*)(1)G(a,a^*)(1)^{-\frac12}\right\|
\label{basic3}.
\end{eqnarray}
In addition,
\begin{eqnarray}
\lefteqn{H(f(a),f(c)^*)(1)H(f(c),
f(c)^*)(1)^{-1}H(f(c),f(a)^*)(1)-H(f(a),f(a)^*)(1)}\nonumber\\
& \leq & H(f(a),f(a)^*)(1)\times\nonumber\\
& & \mbox{}\left\|G(a,a^*)(1)^{-\frac12}G(a,c^*)
(1)G(c,c^*)(1)^{-1}G(c,a^*)(1)G(a,a^*)(1)^{-\frac12}-1\right\|
\label{basic2}.
\end{eqnarray}
\end{prop}

\begin{remark}\label{rmk3.3}
If in addition $H(u,v^*)(1)H(v,v^*)(1)^{-1}H(v,u^*)(1)-H(u,u^*)
(1)\ge0$ for all $u,v\in\mathcal D_{H,{\rm nc}}$, then
relation \eqref{basic2} is equivalent to
\begin{eqnarray}
\lefteqn{\left\|H(f(a),f(a)^*)(1)^{-\frac12}H(f(a),f(c)^*)(1)
\right.}\nonumber\\
& & \left.\mbox{}\times H(f(c),f(c)^*)(1)^{-1}H(f(c),f(a)^*)(1)
H(f(a),f(a)^*)(1)^{-\frac12}-1\right\|\label{basicnorm}\\
& \leq & \mbox{}\left\|G(a,a^*)(1)^{-\frac12}G(a,c^*)(1)
G(c,c^*)(1)^{-1}G(c,a^*)(1)G(a,a^*)(1)^{-\frac12}-1\right\|
.\nonumber
\end{eqnarray}
\end{remark}

\begin{remark}\label{generalizeds}
The condition $H(u,v^*)(1)H(v,v^*)(1)^{-1}H(v,u^*)(1)-H(u,u^*)
(1)\ge0$ for all $u,v\in\mathcal D_{H,{\rm nc}}$ is satisfied by 
a large and important class of noncommutative domains. In 
particular, it is satisfied by generalized noncommutative 
half-planes and generalized noncommutative balls. Indeed, 
consider the upper half-plane from Example \ref{kernel-ex}(i),
$H^+(\mathcal A^{n\times n})=\{b\in\mathcal A^{n\times n}
\colon\Im b>0\}$. It is given by the affine kernel 
$H(a,c)(P)=(2i)^{-1}(aP-Pc)$. Then the inequality reduces to
$$
\frac{a-c^*}{2i}\left(\frac{c-c^*}{2i}\right)^{-1}\frac{c-a^*}{2i}-
\frac{a-a^*}{2i}\ge0.
$$
But 
$$
\frac{a-c^*}{2i}\left(\frac{c-c^*}{2i}\right)^{-1}\frac{c-a^*}{2i}-
\frac{a-a^*}{2i}=\frac14(a-c)\left(\frac{c-c^*}{2i}\right)^{-1}(a-c)^*
\ge0
$$
whenever $a\neq c$ in the upper half-plane. One can generalize this
to kernels of the form $H(a,c)(P)=(2i)^{-1}(h(a)P-Ph(c^*)^*)$, for
some noncommutative function $h$. 

A better-known class of kernels is given by the formula 
$H(a,c)(P)=1-h(a)Ph(c^*)^*$ with $h$ a noncommutative function.
Then $H(a,c^*)(1)=1-h(a)h(c)^*$, so from the point of view of
the above inequality it is enough to consider the case when $h$
is the identity function. If $H(a,c)(P)=1-aPc$, this comes down to:
\begin{eqnarray*}
\lefteqn{(1-ac^*)(1-cc^*)^{-1}(1-ca^*)-(1-aa^*)}\\
& = &1-ca^*+(c-a)c^*(1-cc^*)^{-1}
c(c-a)^*+(c-a)c^*-1+aa^*\\
& = & (c-a)c^*(1-cc^*)^{-1}
c(c-a)^*+cc^*+aa^*-ac^*-ca^*\\
& = & (c-a)c^*(1-cc^*)^{-1}
c(c-a)^*+(c-a)(c-a)^*\ge0
\end{eqnarray*}
whenever $a\neq c$ satisfy $\|a\|,\|c\|<1$. 

Note that this also proves 
that
$$
\left\|H(a,a^*)(1)^{-\frac12}H(a,c^*)(1)H(c,c^*)(1)^{-1}
H(c,a^*)(1)H(a,a^*)(1)^{-\frac12}-1\right\|=0
$$
for $H(a,c)(P)=1-h(a)Ph(c^*)^*$ if and only if $h(a)=h(c)$. 
So the pseudodistance defined by this formula separates 
points if and only if $h$ is injective. The same
fact holds for the generalized half-plane.
\end{remark}


\begin{proof}[Proof of Proposition \ref{prop:3.1}]
The proof is based on showing that formula \eqref{Infinitesimal}
for $\delta_{\mathcal D_{G,\rm nc}}$ coincides with the definition 
\eqref{InfDistance1} in the particular case of a domain defined by a 
noncommutative kernel as in assumption (5). An application of 
Proposition \ref{contr} will allow us to conclude. On the way to proving 
\eqref{deriv}, we will obtain formulas allowing us to argue that 
\eqref{basic3} and \eqref{basic2} hold by applying the same 
principle as in the proof of Proposition \ref{contr}. In some cases,
for future reference, we will perform computations which are 
slightly more involved than absolutely necessary.

Thus, let us start by evaluating $G$ on elements
$\begin{bmatrix}
a & {b}\\
0 & c
\end{bmatrix}$, $a\in\mathcal D_{G, n},c\in\mathcal D_{G,m}$. We have 
$
G\left(\begin{bmatrix}
a & {b}\\
0 & c
\end{bmatrix},\begin{bmatrix}
a^* & 0\\
{b}^* & c^*
\end{bmatrix}\right)\left(I_{n+m}\right)>0
$
whenever $\begin{bmatrix}
a & {b}\\
0 & c
\end{bmatrix}\in\mathcal D_{G,n+m}.$ As $a\in\mathcal
D_{G,n},c\in\mathcal D_{G,m}$, we can use the properties of 
nc functions/kernels to write explicitly the entries of this matrix. For 
future reference, we consider the general case, with $P_{11}\in
\mathcal J^{n\times n},P_{12}\in\mathcal J^{n\times m},
P_{21}\in\mathcal J^{m\times n},P_{22}\in\mathcal J^{m\times m}$,
$a_1,a_2\in\mathcal D_{G_1,n},c_1,c_2\in\mathcal D_{G_1,m},
b_1\in\mathcal V^{n\times m},b_2\in\mathcal V^{m\times n}$.
According to condition \eqref{sette} in the definition of affine nc kernels,
\begin{eqnarray*}
\lefteqn{G\left(\begin{bmatrix}
a_1 & {b}_1\\
0_{m\times n} & c_1
\end{bmatrix},\begin{bmatrix}
a_2 & 0_{n\times m}\\
{b}_2 & c_2
\end{bmatrix}\right)\left(\begin{bmatrix}
P_{11} & P_{12}\\
P_{21} & P_{22}
\end{bmatrix}\right)}\\
& = & G\left(\begin{bmatrix}
a_1 & {b}_1\\
0 & c_1
\end{bmatrix},\begin{bmatrix}
0 & 1_n\\
1_m & 0
\end{bmatrix}
\begin{bmatrix}
c_2 & {b}_2\\
0_{n\times m} & a_2
\end{bmatrix}\begin{bmatrix}
0 & 1_m\\
1_n & 0
\end{bmatrix}\right)\left(\begin{bmatrix}
P_{12} & P_{11}\\
P_{22} & P_{21}
\end{bmatrix}\begin{bmatrix}
0 & 1_m\\
1_n & 0
\end{bmatrix}\right)\\
& = & G\left(\begin{bmatrix}
a_1 & {b}_1\\
0 & c_1
\end{bmatrix},\begin{bmatrix}
c_2 & {b}_2\\
0 & a_2
\end{bmatrix}\right)\left(\begin{bmatrix}
P_{12} & P_{11}\\
P_{22} & P_{21}
\end{bmatrix}\right)\begin{bmatrix}
0 & 1_m\\
1_n & 0
\end{bmatrix},
\end{eqnarray*}
On the other hand, 
\begin{eqnarray*}
\lefteqn{G\left(\begin{bmatrix}
a_1 & {b}_1\\
0 & c_1
\end{bmatrix},\begin{bmatrix}
c_2 & b_2\\
0 & a_2
\end{bmatrix}\right)\left(\begin{bmatrix}
P_{12} & P_{11}\\
P_{22} & P_{21}
\end{bmatrix}\right)=}\\
&  & \begin{bmatrix}
G\left(a_1,\begin{bmatrix}
c_2 & b_2\\
0 & a_2
\end{bmatrix}\right)\left(\begin{bmatrix}
P_{12} & P_{11}
\end{bmatrix}
\right)+{}_0\Delta G\left(a_1;c_1,\begin{bmatrix}
c_2 & b_2\\
0 & a_2
\end{bmatrix}\right)\left(b_1,\begin{bmatrix}
P_{22} & P_{21}
\end{bmatrix}
\right)\\
G\left(c_1,\begin{bmatrix}
c_2 & b_2\\
0 & a_2
\end{bmatrix}\right))\left(\begin{bmatrix}
P_{22} & P_{21}
\end{bmatrix}
\right)
\end{bmatrix}.
\end{eqnarray*}
We identify each of the two components of this column vector.
\begin{eqnarray*}
\lefteqn{G\left(a_1,\begin{bmatrix}
c_2 & b_2\\
0 & a_2
\end{bmatrix}\right)\left(\begin{bmatrix}
P_{12} & P_{11}
\end{bmatrix}
\right)}\\
& = & \begin{bmatrix}
G(a_1,c_2)([P_{12}]) & {}_1\Delta G(a_1,c_2;a_2)([P_{12}],b_2)+
G(a_1,a_2)([P_{11}])
\end{bmatrix},
\end{eqnarray*}
and
\begin{eqnarray*}
\lefteqn{G\left(c_1,\begin{bmatrix}
c_2 & b_2\\
0 & a_2
\end{bmatrix}\right)\left(\begin{bmatrix}
P_{22} & P_{21}
\end{bmatrix}
\right)}\\
& = & \begin{bmatrix}
G(c_1,c_2)([P_{22}]) & {}_1\Delta G(c_1,c_2;a_2)([P_{22}],b_2)+
G(c_1,a_2)([P_{21}])
\end{bmatrix}.
\end{eqnarray*}
Finally,
\begin{eqnarray*}
\lefteqn{{}_0\Delta G\left(a_1,c_1,\begin{bmatrix}
c_2 & b_2\\
0 & a_2
\end{bmatrix}\right)\left(b_1,\begin{bmatrix}
P_{22} & P_{21}
\end{bmatrix}
\right)= }\\
& & \left[
\begin{array}{ll}
{}_0\Delta G\left(a_1;c_1,c_2\right)(b_1,[P_{22}]) & 
{}_1\Delta{}_0\Delta G(a_1;c_1,c_2;a_2)(b_1,[P_{22}],b_2)\\
 & \ \quad +{}_0\Delta G(a_1;c_1,a_2)(b_1,[P_{21}])
\end{array}\right],
\end{eqnarray*}
a row vector with two components. To centralize all results, if 
$$
G\left(\begin{bmatrix}
a_1 & {b}_1\\
0 & c_1
\end{bmatrix},\begin{bmatrix}
c_2 & b_2\\
0 & a_2
\end{bmatrix}\right)\left(\begin{bmatrix}
P_{12} & P_{11}\\
P_{22} & P_{21}
\end{bmatrix}\right)=\begin{bmatrix}
G_{11} & G_{12}\\
G_{21} & G_{22}
\end{bmatrix},
$$
then 
\begin{equation}\label{G11}
G_{11}=G(a_1,c_2)([P_{12}])+{}_0\Delta G\left(a_1;c_1,c_2\right)(b_1,
[P_{22}]),
\end{equation}
\begin{eqnarray}\label{G12}
G_{12}& = & {}_1\Delta G(a_1,c_2;a_2)([P_{12}],b_2)\\
& & \mbox{}+
G(a_1,a_2)([P_{11}])+{}_1\Delta{}_0\Delta G(a_1;c_1,c_2;a_2)
(b_1,[P_{22}],b_2)\nonumber\\
& & \mbox{}+{}_0\Delta G(a_1;c_1,a_2)(b_1,[P_{21}]),\nonumber
\end{eqnarray}
\begin{equation}\label{G21}
G_{21}=G(c_1,c_2)([P_{22}]),
\end{equation}
\begin{equation}\label{G22}
G_{22}={}_1\Delta G(c_1,c_2;a_2)([P_{22}],b_2)+
G(c_1,a_2)([P_{21}]).
\end{equation}

For any C${}^*$-algebra $\mathcal A$, 
$\begin{bmatrix}
u & v\\
v^* & w
\end{bmatrix}\in\mathcal A^{(n+m)\times(n+m)}$ is strictly positive
if and only if $u>0,w>0$ and $v^*u^{-1}v<w$ (or, equivalently,
$u>vw^{-1}v^*$ - see \cite[Chapter 3]{Paulsen}). 
Thus, $\begin{bmatrix}
a & b\\
0 & c
\end{bmatrix}\in\mathcal D_{G,n+m}$ if and only if\footnote{
If the requirement in the definition of $\mathcal
D_{G,\rm nc}$
were that
$G\left(\begin{bmatrix}
a & {b}\\
0 & c
\end{bmatrix},\begin{bmatrix}
a^* & 0\\
{b}^* & c^*
\end{bmatrix}\right)$ is completely positive (equivalently,
$G\left(\begin{bmatrix}
a & {b}\\
0 & c
\end{bmatrix},\begin{bmatrix}
c^* & {b}^*\\
0 & a^*
\end{bmatrix}\right)\left(\begin{bmatrix}
P_{12} & P_{11}\\
P_{22} & P_{21}
\end{bmatrix}\right)\begin{bmatrix}
0 & 1 \\
1 & 0
\end{bmatrix}>0$ for all 
$\begin{bmatrix}
P_{11} & P_{12}\\
P_{21} & P_{22}
\end{bmatrix}>0$ 
in $\mathcal J^{(n+m)\times(n+m)}$ -- so  $P_{12}
=P_{21}^*$) for all $n,m$, then, according to the above formula applied 
to $a_1=a,a_2=a^*,c_1=c,c_2=c^*,b_1=b,b_2=b^*$, the requirement 
$\begin{bmatrix}
G_{12} & G_{11}\\
G_{22} & G_{21}
\end{bmatrix}>0$, would become
\begin{eqnarray*}
0& < & G(c,c^*)([P_{22}]),\\
0 & < & G(a,a^*)([P_{11}])+{}_1\Delta {}_0\Delta G(a;c,c^*;a^*)(b,[P_{22}],b^*)\\
& & \mbox{}+{}_1\Delta G(a,c^*;a^*)([P_{12}],b^*)
+{}_0\Delta G(a;c,a^*)(b,[P_{21}]),
\end{eqnarray*}
and
\begin{eqnarray}
\lefteqn{\left[G(a,c^*)([P_{12}])+{}_0\Delta G(a;c,c^*)(b,[P_{22}])
\right]\left[G(c,c^*)([P_{22}])\right]^{-1}}\nonumber\\
& & \mbox{}\times\left[{}_1\Delta G(c,c^*;a^*)([P_{22}],b^*)+G
(c,a^*)([P_{21}])\right]\nonumber\\
& < & 
G(a,a^*)([P_{11}])+{}_0\Delta G(a;c,a^*)(b,[P_{21}])
+{}_1\Delta G(a,c^*;a^*)([P_{12}],b^*)\nonumber\\
& & \mbox{}+{}_1\Delta {}_0\Delta G(a;c,c^*;a^*)(b,[P_{22}],b^*).
\label{long}
\end{eqnarray}
}
$$
a\in\mathcal D_{G,n},c\in\mathcal D_{G,m}\ \text{ and } \ 
G\left(\begin{bmatrix}
a & {b}\\
0 & c
\end{bmatrix},\begin{bmatrix}
c^* & {b}^*\\
0 & a^*
\end{bmatrix}\right)\left(\begin{bmatrix}
0 & 1_n\\
1_m & 0
\end{bmatrix}\right)\begin{bmatrix}
0 & 1_m \\
1_n & 0
\end{bmatrix}>0.
$$ 
The requirement of positivity for $G$ applied to a block-diagonal
$P=\begin{bmatrix}
P_{11} & P_{12}\\
P_{21} & P_{22}
\end{bmatrix}=\begin{bmatrix}
P_{11} & 0\\
0 & P_{22}
\end{bmatrix}$, means
\begin{eqnarray*}
0& < & G(c,c^*)([P_{22}]),\\
0 & < & G(a,a^*)([P_{11}])+{}_1\Delta {}_0\Delta G(a;c,c^*;a^*)
(b,[P_{22}],b^*),
\end{eqnarray*}
and
\begin{eqnarray}
\lefteqn{{}_0\Delta G(a;c,c^*)(b,[P_{22}])
\left[G(c,c^*)([P_{22}])\right]^{-1}{}_1\Delta G(c,c^*;a^*)
([P_{22}],b^*)}\nonumber\\
& < & 
G(a,a^*)([P_{11}])+{}_1\Delta {}_0\Delta G(a;c,c^*;a^*)
(b,[P_{22}],b^*).\label{useful}
\end{eqnarray}
(Note that if $G_1(x,x^*)$ were cp, by letting $P_{11}$ go to zero in 
the above, we'd 
conclude that the map $P_{22}\mapsto{}_1\Delta {}_0\Delta G_1
(a,c,c^*,a^*)(b,[P_{22}],b^*)$ is necessarily a completely positive
map whenever $b$ is so that 
$\begin{bmatrix}
a & b\\
0 & c
\end{bmatrix}\in\mathcal D_{G,n+m}$.)

Given $a,c$ as above, by the openness of
$\mathcal O_{\rm nc}$, which is a consequence of 
condition (5) and of the analyticity of $G$, we 
know that there is an $\epsilon>0$ depending on
$a,c$ so that $\begin{bmatrix}
a & b\\
0 & c
\end{bmatrix}\in\mathcal D_{G,n+m}$ for all $b\in
\mathcal V^{n\times m}$ with $\|b\|<\epsilon$. Fix a direction 
$b_0\in\mathcal V^{n\times m}$. Then, recalling the definition \eqref{InfDistance1} for $\delta$,
$$
\varepsilon_0:=\left[\delta_{\mathcal D_{G,\rm nc}}(a,c)(b)\right]^{-1}=\sup\left\{t\in(0,+\infty)\colon
\begin{bmatrix}
a & sb_0\\
0 & c
\end{bmatrix}\in\mathcal D_{G,n+m}\forall s<t\right\}\in(0,+\infty].
$$
Observe that if 
$$
{}_1\Delta {}_0\Delta G(a;c,c^*;a^*)(b_0,1,b_0^*)
\ge{}_0\Delta G(a;c,c^*)(b_0,1)
G(c,c^*)(1)^{-1}{}_1\Delta G(c,c^*;a^*)(1,b_0^*)
$$ 
then $\mathcal D_{G,n+m}$ contains a complex 
line. Indeed, one simply divides
by $|z|^2$ in \eqref{useful}.

We argue that $\delta_{\mathcal D_{G,\rm nc}}(a,c)(b)$
is indeed given in this case by formula \eqref{Infinitesimal}.
 If $\varepsilon_0<+\infty$, then it can be written as
\begin{eqnarray*}
\lefteqn{\delta_{\mathcal D_{G,\rm nc}}(a,c)(b)^2=\varepsilon_0^{-2}=}\\
& & \sup\left\{\varphi\left(G(a,a^*)(1)^{-1/2}\left[
{}_0\Delta G(a;c,c^*)(b_0,1)
\left[G(c,c^*)(1)\right]^{-1}{}_1\Delta G(c,c^*;a^*)
(1,b_0^*)\right.\right.\right.\\
& & \quad\ \ \mbox{}-{}_1\Delta {}_0\Delta G(a;c,c^*;a^*)
(b_0,1,b_0^*)\left.\left.\left.\frac{}{}
\right]G(a,a^*)(1)^{-1/2}\right)
\colon\varphi\colon\mathcal K^{n\times n}\to\mathbb C\text{ state}
\right\}.
\end{eqnarray*}
Thus, formula \eqref{Infinitesimal} holds. By Proposition \ref{contr}, we conclude 
relation \eqref{deriv}. If $\varepsilon_0=+\infty$, there is nothing to prove.

Consider now the case $b_0=\epsilon(a-c)$ for some arbitrary
$\epsilon\in\mathbb C$. We apply Equation \eqref{FDC} to write
\begin{itemize}
\item ${}_0\Delta G(a;c,a^*)(\epsilon(a-c),[P_{21}])=
G(a,a^*)([\epsilon P_{21}])-G(c,a^*)([\epsilon P_{21}])$;
\item ${}_1\Delta G(a,c^*;a^*)([P_{12}],\overline{\epsilon}(a-c)^*)
=G(a,a^*)([\overline{\epsilon}P_{12}])-G(a,c^*)
([\overline{\epsilon}P_{12}]);$
\item ${}_0\Delta G(a;c,c^*)(\epsilon(a-c),[P_{22}])=
G(a,c^*)([\epsilon P_{22}])-G(c,c^*)([\epsilon P_{22}])$;
\item ${}_1\Delta G(c,c^*;a^*)
([P_{22}],\overline{\epsilon}(a-c)^*)=G(c,a^*)([\overline{\epsilon}
P_{22}])-G(c,c^*)([\overline{\epsilon}P_{22}]);$
\item $
{}_1\Delta {}_0\Delta G(a;c,c^*;a^*)
(\epsilon(a-c),[P_{22}],\overline{\epsilon}(a-c)^*)=$

\centerline{$
G(a,a^*)([|\epsilon|^2P_{22}])-G(c,a^*)([|\epsilon|^2P_{22}])
-G(a,c^*)([|\epsilon|^2P_{22}])+G(c,c^*)([|\epsilon|^2P_{22}]).$}
\end{itemize}
We record for future reference the expressions for $G_{ij}$ 
corresponding to $b_0=\epsilon(a-c)$.
\begin{eqnarray}
G_{11} & = & G(a,c^*)([P_{12}])+\epsilon(G(a,c^*)([P_{22}])-
G(c,c^*)([P_{22}]))\nonumber\\
G_{12} & = & G(a,a^*)([\epsilon P_{21}]+[\overline{\epsilon}P_{12}])
-G(a,c^*)([\overline{\epsilon}P_{12}])-G(c,a^*)([\epsilon P_{21}])
\nonumber\\
& & \mbox{}+G(a,a^*)([P_{11}])\label{a-c}\\
& & \mbox{}+|\epsilon|^2(G(a,a^*)([P_{22}])+
G(c,c^*)([P_{22}])-G(a,c^*)([P_{22}])-G(c,a^*)([P_{22}]))
\nonumber\\
G_{21} & = & G(c,c^*)([P_{22}])\nonumber\\
G_{22} & = & G(c,a^*)([P_{21}])+\overline{\epsilon}(G(c,a^*)([P_{22}])
-G(c,c^*)([P_{22}]))\nonumber
\end{eqnarray}

For $\epsilon=1$, we obtain that for any state $\psi$ on $\mathcal K^{n
\times n}$ and $\varepsilon>0$, there is a state $\varphi$ on 
$\mathcal K^{n\times n}$ depending on $\varepsilon$ such 
that
\begin{eqnarray}
\lefteqn{\psi\left(H(f(a),f(a)^*)(1)^{-\frac12}H(f(a),f(c)^*)(1)^{}
\left[H(f(c),f(c)^*)(1)\right]^{-1}\right.}\nonumber\\
& & \mbox{}\times H(f(c),f(a)^*)(1)\left.
H(f(a),f(a)^*)(1)^{-\frac12}-1\right)-\varepsilon\nonumber\\
& \leq & 
\varphi\left(G(a,a^*)(1)^{-\frac12}G(a,c^*)(1)
\left[G(c,c^*)(1)\right]^{-1}G(c,a^*)(1)G(a,a^*)(1)^{-\frac12}-1
\right).
\nonumber
\end{eqnarray}
Recall that $\psi,\varphi$ are states, so that $\varphi(1)=
\psi(1)=1$, which implies that
\begin{eqnarray}
\lefteqn{\psi\left(H(f(a),f(a)^*)(1)^{-\frac12}H(f(a),f(c)^*)(1)^{}
\left[H(f(c),f(c)^*)(1)\right]^{-1}\right.}\nonumber\\
& & \mbox{}\times H(f(c),f(a)^*)(1)\left.
H(f(a),f(a)^*)(1)^{-\frac12}\right)-\varepsilon\nonumber\\
& \leq & 
\varphi\left(G(a,a^*)(1)^{-\frac12}G(a,c^*)(1)
\left[G(c,c^*)(1)\right]^{-1}G(c,a^*)(1)G(a,a^*)(1)^{-\frac12}
\right).
\nonumber
\end{eqnarray}
Clearly the elements under the states above are nonnegative, so this
reduces to
\begin{eqnarray}
\lefteqn{\left\|H(f(a),f(a)^*)(1)^{-\frac12}H(f(a),f(c)^*)(1)^{}
\left[H(f(c),f(c)^*)(1)\right]^{-1}\right.}\nonumber\\
& & \mbox{}\times H(f(c),f(a)^*)(1)\left.
H(f(a),f(a)^*)(1)^{-\frac12}\right\|\nonumber\\
& \leq & 
\left\|G(a,a^*)(1)^{-\frac12}G(a,c^*)(1)
\left[G(c,c^*)(1)\right]^{-1}G(c,a^*)(1)G(a,a^*)(1)^{-\frac12}
\right\|.
\nonumber
\end{eqnarray}
The last inequality of our proposition, \eqref{basic2}, is
a trivial consequence of the selfadointness of the elements
involved, together with the previous results.
\end{proof}

We are not automatically able to conclude the norm-inequality
\eqref{basicnorm} only because the norm of the left-hand
side might be achieved at the lower bound of the spectrum.
However, assuming the hypothesis of Remark \ref{rmk3.3} guarantees
this is not the case. It wouldn't be unreasonable to suppose
that this hypothesis is satisfied in most cases of interest. So we 
discuss next three things related to it.

\begin{remark}
First, not surprisingly, the inequality 
$G(a,a^*)(1)-G(a,c^*)(1)G(c,c^*)(1)^{-1}$
$G(c,a^*)(1)\geq0$, opposite to the one introduced in Remark \ref{rmk3.3}, cannot hold
under the assumption of no complex lines in $\mathcal D_{G,\rm nc}$. 
Indeed, if we put $b=\epsilon(a-c)$ in formulas 
\eqref{G11}, \eqref{G12}, \eqref{G21} and \eqref{G22} for $a_1=
a_2^*=a,c_1=c_2^*=c$, we obtain, according to \eqref{a-c} with
$\epsilon>0,P_{11}=P_{22}=1,P_{12}=P_{21}=0$, the matrix inequality
{\footnotesize
$$
\begin{bmatrix}
G(a,a^*)(1+\epsilon^2)-\epsilon^2G(a,c^*)(1)-\epsilon^2G(c,a^*)(1)+
\epsilon^2G(c,c^*)(1) & \epsilon[G(a,c^*)(1)-G(c,c^*)(1)]\\
\epsilon[G(c,a^*)(1)-G(c,c^*)(1)] & G(c,c^*)(1)
\end{bmatrix}>0.
$$
}
Multiplying with $\begin{bmatrix}
1 & \epsilon\\
0 & 1
\end{bmatrix}$ left and its adjoint right does not change the
positivity of the matrix, so that
\begin{equation}\label{gac}
\begin{bmatrix}
(1+\epsilon^2)G(a,a^*)(1) & \epsilon G(a,c^*)(1)\\
\epsilon G(c,a^*)(1) & G(c,c^*)(1)
\end{bmatrix}>0,
\end{equation}
for any $\epsilon>0$ such that $\begin{bmatrix}
a & \epsilon(a-c)\\
0 & c \end{bmatrix}\in\mathcal D_{G,2n}.$
Since $\mathcal D_{G,2n}$ contains no complex line, it follows that
there is an $\epsilon_0(a,c)>0$ maximal beyond which the matrix 
inequality above fails. Thus, necessarily
$G(a,a^*)(1)-G(a,c^*)(1)G(c,c^*)(1)^{-1}G(c,a^*)(1)\not\geq0$.

\end{remark}

\begin{remark}

Second, we observe that certain obvious transformations of
$G$ have similarly obvious effects on $\mathcal D_{G,\rm nc}.$ For
example, composing $G$ with a completely positive 
unital map increases $\mathcal D_{G,\rm nc}$. Indeed, if 
$\Phi$ is such a map, then $G(a,a^*)(1)>\varepsilon1\implies
\Phi(G(a,a^*)(1))>\varepsilon\Phi(1)=\varepsilon1$, so that
$\mathcal D_{G,\rm nc}\subseteq\mathcal D_{\Phi\circ G,\rm nc}$.
Subtracting a positive multiple of 1 from $G$ decreases 
$\mathcal D_{G,\rm nc}$, adding increases it. However,
if for any $t\in\mathbb R\setminus\sigma(G(c,c^*)(1))$ we let
$$
f(t)=[G(a,c^*)(1)-t1]
\left[G(c,c^*)(1)-t1\right]^{-1}[G(c,a^*)(1)-t1]-[G(a,a^*)(1)-t1],
$$
then 
\begin{eqnarray*}
\lefteqn{f'(t)=
\left[(G(c,c^*)(1)-t1)^{-1}(G(c,a^*)(1)-t1)-1\right]}\\
& & \quad\quad\mbox{}\times\left[(G(a,c^*)(1)-t1)(G(c,c^*)(1)-t1)^{-1}-1\right]\ge0,
\end{eqnarray*}
(recall hypothesis (3) which states that $G(a,c)(1)^*=G(c^*,a^*)(1))$,
and
\begin{eqnarray*}
\lefteqn{f''(t)=2\left[1-(G(a,c^*)(1)-t1)(G(c,c^*)(1)-t1)^{-1}\right]}\\
& & \mbox{}\times(G(c,c^*)(1)-t1)^{-1}
\left[1-(G(c,c^*)(1)-t1)^{-1}(G(c,a^*)(1)-t1)\right]\ge0,
\end{eqnarray*}
for all $t\in\mathbb R\setminus\sigma(G(c,c^*)(1)$. This means 
that for any state $\varphi$, the map $t\mapsto\varphi\circ f$ is
convex and increasing on each connected component of 
$\mathbb R\setminus\sigma(G(c,c^*)(1)$. Clearly, 
$\lim_{t\to\pm\infty}\|f(t)-(G(a,c^*)(1)+G(c,c^*)(1)+G(c,a^*)(1)-G(a,a^*)(1))\|=0$,
so if $a$ is such that $G(a,c^*)(1)+G(c,c^*)(1)+G(c,a^*)(1)-G(a,a^*)(1)\ge0$
(for example, $c\in\mathcal D_{G,n}$ and $a$ close to $c$), then 
$f(t)\ge0$ for all real $t$ in the connected component of $-\infty$. Conversely,
if $G(a,c^*)(1)+G(c,c^*)(1)+G(c,a^*)(1)-G(a,a^*)(1)\not\ge0$, then
$f(t)\not\ge0$ for all real $t$ in the connected component of $+\infty$.
\end{remark}

\begin{remark}\label{MatrixConvexity}
Finally, the function $\delta$ has been defined in terms of the length of a ``ray'' 
in a given direction. In this remark we look at the whole set of points $b$ for which the upper 
triangular matrix $\begin{bmatrix}
a & b\\
0 & c\end{bmatrix}$ belongs to the chosen noncommutative set.
Consider a nc set $\mathcal D$ which satisfies property \eqref{propr2set}. 
Fix $m,n\in\mathbb N$ and $a\in\mathcal D_n,c\in\mathcal D_m$. Let 
$$
\daleth(a,c)_{\rm nc}=\coprod_{k\in\mathbb N}\left\{b\in\mathcal V^{nk\times mk}\colon
\begin{bmatrix}
I_k\otimes a & b\\
0 & I_k\otimes c\end{bmatrix}\in\mathcal D_{km+kn}\right\}.
$$
It is quite easy to see that this set is noncommutative: if $b=b_1\oplus b_2$ with $b_j\in
\daleth(a,c)_{k_j}$, $j=1,2$, then 
$$
\begin{bmatrix}
I_{k_1+k_2}\otimes a & b\\
0 & I_{k_1+k_2}\otimes c\end{bmatrix}=
\begin{bmatrix}
I_{k_1}\otimes a & 0  & b_1 & 0 \\
0 & I_{k_2}\otimes a & 0 & b_2 \\
0 & 0 & I_{k_1}\otimes c & 0 \\
0 & 0 & 0 & I_{k_2}\otimes c\end{bmatrix},
$$
By permuting rows 2 and 3 and columns 2 and 3 (which comes to the conjugation with a 
scalar matrix), we obtain
$$
\begin{bmatrix}
I_{k_1}\otimes a  & b_1 & 0 & 0 \\
0 & I_{k_1}\otimes c & 0 & 0 \\
0 & 0 & I_{k_2}\otimes a & b_2 \\
0 & 0 & 0 & I_{k_2}\otimes c\end{bmatrix},
$$
which belongs to $\mathcal D_{(n+m)(k_1+k_2)}$ because $\mathcal D$ is a noncommutative set.

Similarly, $
\daleth(a,c)_{\rm nc}$ is invariant by conjugation with scalar unitary matrices: if $b\in
(\mathcal V^{n\times m})^{k\times k}$, then for any unitary matrix $U\in\mathbb C^{k\times k}$,
$$
\begin{bmatrix}
I_k\otimes a & UbU^*\\
0 & I_k\otimes c\end{bmatrix}=\begin{bmatrix}
U\otimes I_n & 0\\
0 & U\otimes I_m\end{bmatrix}\begin{bmatrix}
I_k\otimes a & b\\
0 & I_k\otimes c\end{bmatrix}\begin{bmatrix}
U^*\otimes I_n & 0\\
0 & U^*\otimes I_m\end{bmatrix}, 
$$
which belongs to $\mathcal D_{k(m+n)}$ by assumption \eqref{propr2set} on the set $\mathcal D$.

Given the unitary invariance of the set $\daleth(a,c)_{\rm nc}$ and Lemma
\ref{lemma3.4}, one is justified in asking whether $\daleth(a,c)_{\rm nc}$
is in fact matrix convex. That turns out to be false in general.

Let us recall Wittstock's definition of a matrix convex set (see \cite[Section 3]{EW}): a matrix convex set
is a noncommutative set $K=\coprod_nK_n$ such that for any $S\in\mathbb C^{r\times n}$
satisfying $S^*S=I_n$, we have $S^*K_rS\subseteq K_n$. Since $\daleth(a,c)_{\rm nc}$
is invariant by conjugation with scalar unitary matrices, matrix convexity of $\daleth(a,c)_{\rm nc}$
is equivalent to the following statement: for any $k<k'\in\mathbb N$ and $b\in\daleth(a,c)_{k'}$,
we have $\begin{bmatrix}
I_k & 0\end{bmatrix} b\begin{bmatrix}
I_k\\
0\end{bmatrix}\in\daleth(a,c)_{k},$ i.e. the upper right $k\times k$ corner of $b$ is
an element of $\daleth(a,c)_{\rm nc}$ whenever $b$ is. There is a simple counterexample
to this statement: consider the unit disk $\mathbb D$ in the complex plane, and the noncommutative
set 
$$
\mathcal D=\coprod_{k\in\mathbb N}\{A\in\mathbb C^{k\times k}\colon \sigma(A)\subset\mathbb D,
\|A\|<k\}.
$$
This is clearly a noncommutative set (if $A_j\in\mathcal D_{k_j}$, then $\|A_1\oplus A_2\|=\max
\{\|A_1\|,\|A_2\|\}<\max\{k_1,k_2\}<k_1+k_2$) which is unitarily invariant ($\|U^*AU\|=\|A\|$).
However, 
$$
\begin{bmatrix}
0 & 0 & 3 & 0\\
0 & 0 & 0 & \frac12\\
0 & 0 & 0 & 0 \\
0 & 0 & 0 & 0
\end{bmatrix}\in\mathcal D_4,\quad\text{while}\quad \begin{bmatrix}
0 & 3\\
0 & 0 
\end{bmatrix}\not\in\mathcal D_2,
$$
which means that $\begin{bmatrix}
3 & 0\\
0 & \frac12
\end{bmatrix}\in\daleth(0,0)_{2}$, while $3\not\in\daleth(0,0)_{1}$.

However, there are important classes of nc sets for which the set $\daleth$ is matrix convex.
One such example is the class of generalized half-planes (see Remark \ref{generalizeds}). 
Consider an injective nc map $h\colon\mathcal V_{\rm nc}\to\mathcal A_{\rm nc}$ for some
unital $C^*$-algebra $\mathcal A$. Recall that a generalized half-plane is 
$$
H^+_h(\mathcal V)=\coprod_{n=1}^\infty\{a\in\mathcal V^{n\times n}\colon h(a)+h(a)^*>0\}.
$$
Then elements $b\in\daleth(a,c)_{\rm nc}$ must satisfy 
$$
(\Re h(a))^{-1/2}
\Delta h(a,c)(b)(\Re h(c))^{-1}\Delta h(a,c)(b)^*(\Re h(a))^{-1/2}<4\cdot1.
$$
That is, for any $k'\in\mathbb N$,
\begin{eqnarray*}
\lefteqn{(I_{k'}\otimes\Re h(a))^{-1/2}\Delta h(I_{k'}\otimes a,I_{k'}\otimes c)(b)
(I_{k'}\otimes\Re h(c))^{-1}}\\
& & \mbox{}\times\Delta h(I_{k'}\otimes a,I_{k'}\otimes c)(b)^*(I_{k'}\otimes\Re h(a))^{-1/2}\\
& = & 
\left[\sum_{l=1}^{k'}(\Re h(a))^{-1/2}\Delta h(a,c)(b_{il})(\Re h(c))^{-1}
\Delta h(a,c)(b_{jl})^*(\Re h(a))^{-1/2}\right]_{1\le i,j\le k'}\\
&<&4I_{k'}\otimes1.
\end{eqnarray*}
If one fixes such a $k'>1$ in $\mathbb N$ and a $b\in\daleth(a,c)_{k'}$, proving matrix convexity 
comes to proving that the upper right $k\times k$ corner of $b$ is in $\daleth(a,c)_k$ for all $0<k<k'$.
That is,
$$
\left[\sum_{l=1}^{k}(\Re h(a))^{-1/2}\Delta h(a,c)(b_{il})(\Re h(c))^{-1}
\Delta h(a,c)(b_{jl})^*(\Re h(a))^{-1/2}\right]_{1\le i,j\le k}<4I_k\otimes1.
$$
Denoting $P_k$ the projection onto the first $k$ coordinates of $\mathbb C^{k'\times k'}$,
the above relation is equivalent to 
\begin{eqnarray*}
\lefteqn{(P_kI_{k'}\otimes\Re h(a))^{-1/2}\Delta h(I_{k'}\otimes a,I_{k'}\otimes c)(b)
(P_kI_{k'}\otimes\Re h(c))^{-1}}\\
& & \mbox{}\times \Delta h(I_{k'}\otimes a,I_{k'}\otimes c)(b)^*(P_kI_{k'}\otimes\Re 
h(a))^{-1/2}<4P_kI_{k'}\otimes1=4I_k\otimes 1.
\end{eqnarray*}
This is implied by the general fact that $AA^*<4I_{k'}\implies PAPA^*P\le4P=4I_k.$ Indeed,
clearly $AA^*<4I_{k'}\implies PAA^*P\le4P=4I_k$ and $P\leq I_{k'}\implies(PA)P(P^*A)^*\le
PAA^*P\le4I_k.$ Thus, $P_kbP_k\in\daleth(a,c)_k$ for all $0<k<k'$.
Clearly this proof applies as well to generalized balls.

\end{remark}

As mentioned in Section \ref{pseudodistance}, the definition of $\tilde{d}_\mathcal D$
is similar to the definition of the Kobayashi distance. We show next that generally
${d}_\mathcal D$ is dominated by the Kobayashi distance, with equality for the
unit ball or half-plane of a $C^*$-algebra. We denote by $k_{\mathcal D,n}$
the Kobayashi distance on the domain $\mathcal D_n$, and by $\kappa_{\mathcal D,n}$
the infinitesimal Kobayashi metric on $\mathcal D_n$.

\begin{prop}\label{Prop4.7}
Let $\mathcal V$ be an operator system and $\mathcal A$ a $C^*$-algebra. Consider 
an injective noncommutative function $h$ defined on an open noncommutative subset 
of  $\mathcal V_{\rm nc}$ whose range contains the unit ball of $\mathcal A_{\rm nc}$. 
Define
$$
\mathcal D_{h,\rm nc}=\coprod_{n=1}^\infty\{a\in\mathcal V^{n\times n}\colon
h(a)h(a)^*<1\}\subseteq\mathcal V_{\rm nc}.
$$ 
Then ${d}_{\mathcal D_{h,n}}(a,c)\leq k_{\mathcal D_{h,n}}(a,c)$ for all  
$a,c\in\mathcal D_{h,n},n\in\mathbb N$.
\end{prop}

\begin{proof}
According to relation \eqref{NCdisk}, the infinitesimal 
Poincar\'e (or hyperbolic) metric on the unit disk $\mathbb D$, 
$\kappa_\mathbb D(z,v)$, coincides with $\delta_\mathbb D(z,z)(v)$:
they both equal $\frac{|v|}{1-\overline{z}z}$. Thus, the metric
generated by $\delta_\mathbb D(z,z)(v)$ coincides with the one
generated by $\kappa_\mathbb D(z,v)$, the Poincar\'e metric.  
According to the definition of the Kobayashi metric, 
\begin{eqnarray}
\lefteqn{\kappa_{\mathcal D_{h,n}}(a,b)}\nonumber\\
& = & \inf\{|w|\colon\exists f\colon\mathbb D\to\mathcal D_{h,n}\text{ analytic},f(0)=a,f'(0)(w)=b\}\nonumber\\
& = & \inf\{\kappa_\mathbb D(0,w)\colon\exists f\colon\mathbb D\to\mathcal D_{h,n}\text{ analytic},f(0)=a,f'(0)(w)=b\}.\label{InfKob}
\end{eqnarray}
By von Neumann's inequality, any holomorphic function $f\colon\mathbb D\to\mathcal D_{h,n}$
has a noncommutative extension $f\colon\mathbb D_{\rm nc}\to(\mathcal D_{h,n})_{\rm nc}:$
$h(f(T))h(f(T))^*<1$ whenever $T\in \mathbb C^{n\times n},\|T\|<1$. This implies 
$$
\delta_{\mathcal D_{h}}(f(z),f(z))(f'(z)(w))=\delta_{\mathcal D_{h}}(f(z),f(z))(\Delta f(z,z)(w))\leq\delta_{\mathbb D_{\rm nc}}(z,z)(w)=\kappa_\mathbb D(z,w)
$$
for all  $z,w\in\mathbb D$. For any $\varepsilon>0$, we choose 
$f_\varepsilon\colon\mathbb D\to\mathcal D_{h,n}$ such that $f_\varepsilon(0)=a,f'_\varepsilon(0)\in\text{Span}_\mathbb C\{b\}$, and
$\frac{\|b\|}{\|f'_\varepsilon(0)\|}=\kappa_\mathbb D(0,\frac{\|b\|}{\|f'_\varepsilon(0)\|})\le
\kappa_{\mathcal D_{h,n}}(f_\varepsilon(0),f'_\varepsilon(0)(\frac{\|b\|}{\|f'_\varepsilon(0)\|}))+\varepsilon=\kappa_{\mathcal D_{h,n}}(a,b)+\varepsilon.$
Thus,
\begin{eqnarray*}
\delta_{\mathcal D_h}(a,a)(b) & = & \delta_{\mathcal D_h}(f_\varepsilon(0),f_\varepsilon(0))\left(f'_\varepsilon(0)\left(\frac{\|b\|}{\|f'_\varepsilon(0)\|}\right)\right)\\
& \leq & \delta_{\mathbb D_{\rm nc}}(0,0)\left(\frac{\|b\|}{\|f'_\varepsilon(0)\|}\right)=\kappa_\mathbb D\left(0,\frac{\|b\|}{\|f'_\varepsilon(0)\|}\right)\\
& \leq & \kappa_{\mathcal D_{h,n}}(a,b)+\varepsilon,
\end{eqnarray*}
for any $\varepsilon>0$. Thus, by integrating and taking the infimum, $d_{\mathcal D_{h,n}}\leq k_{\mathcal D_{h,n}}$. 
\end{proof}

\begin{remark}
Expression  \eqref{NCdisk} and the proof of the above proposition guarantee 
that $d$ and $k$ coincide level-by-level on the noncommutative unit ball. Of 
course, this statement remains true for any noncommutative domain which is 
biholomorphically equivalent to the noncommutative unit ball, including the
noncommutative upper half-plane (see also Theorem \ref{classif} below). In 
general, however, one should not expect equality in the inequality from 
Proposition \ref{Prop4.7}. The essential ingredient in the proof of Proposition
\ref{Prop4.7} is the fact that any analytic map $f\colon\mathbb D\to\mathcal D_{h,n}$ lifts to 
a nc map on $\mathbb D_{\rm nc}$ (the conclusion of Proposition \ref{Prop4.7} holds
for any domain $\mathcal D$ with this property). This ingredient is missing when we try to
obtain the reverse inequality in Proposition \ref{Prop4.7}: 
maps $f\colon\mathcal D_{h,n}\to\mathbb D$ need not satisfy a von Neumann inequality. 
Indeed, in most cases, one cannot even extend such an $f$ to a nc map.
\end{remark}

The definition of the classical infinitesimal Kobayashi metric \eqref{InfKob} on a domain 
is given in terms of the infimum of the norms of derivatives of analytic functions from 
the Poincar\'e disk into the domain. From the infinitesimal metric, one builds the Kobayashi
distance by integrating along curves. In the noncommutative context, distance 
$d_{\mathcal D,\infty}$ from Definition \ref{inf-Kobayashi} allows for a 
similar definition in terms of derivatives of noncommutative functions on
the noncommutative unit disk. Consider again a nc domain $\mathcal D\subset\mathcal V_{\rm nc}$
and let $x\in\mathcal D_{n},b\in\mathcal V^{n\times n}$. We claim that 
\begin{eqnarray*}
\lefteqn{\delta_{\mathcal D_{}}(x,x)(b)=}\\
& & \left[\sup\left\{t\ge0\colon\exists f\colon\mathbb D_{\rm nc}
\to\left(\mathcal D_{2n}\right)_{\rm nc}, 
f(0)=I_2\otimes x,\Delta f(0,0)(1)=t\begin{bmatrix}0&b\\0&0\end{bmatrix}\right\}\right]^{-1}.
\end{eqnarray*}
As above, by $\left(\mathcal D_{n}\right)_{\rm nc}$ we denote the subset of 
$\mathcal D$ formed of all levels which are multiples of $n$.
Indeed, inequality $\leq$ follows easily: as shown in Proposition \ref{contr}, 
and Lemma \ref{lemma3.3}, if $f$ is as in the right-hand side of the above relation, then
\begin{eqnarray*}
t\delta_{\mathcal D}(x,x)(b) & = & \delta_{\mathcal D}(x,x)(tb)\\
&=&\delta_{\mathcal D}\left(\begin{bmatrix}x&0\\0&x\end{bmatrix},\begin{bmatrix}x&0\\0&x\end{bmatrix}\right)
\left(\begin{bmatrix}0&tb\\0&0\end{bmatrix}\right)\\
& = & \delta_{\mathcal D}(f(0),f(0))(\Delta f(0,0)(1))\leq\delta_\mathbb D(0,0)(1)=1.
\end{eqnarray*}
We show the reverse inequality by finding an ``extremal'' function. Let $\iota=
\frac{1}{\delta_{\mathcal D}(x,x)(b)}$.
Consider the nc function $f=\{f_p\}_{p\in\mathbb N},$ $f_p(Z)=\begin{bmatrix}
I_p\otimes x & 0 \\
0 & I_p\otimes x\end{bmatrix}+Z\otimes\begin{bmatrix}
0 & \iota b \\
0 & 0\end{bmatrix}$, $Z\in\mathbb C^{p\times p}$, $\|Z\|<1$, $p\in\mathbb N$.
According to Lemmas \ref{lemma3.4} and \ref{lemma3.3}, we have 
\begin{eqnarray*}
\delta_{\mathcal D}(x,x)(b) & > & \delta_{\mathcal D}(x,x)(b)\|Z\|\\
& = & \delta_{\mathcal D}\left(\begin{bmatrix}
I_p\otimes x & 0 \\
0 & I_p\otimes x\end{bmatrix},\begin{bmatrix}
I_p\otimes x & 0 \\
0 & I_p\otimes x\end{bmatrix}\right)\left(Z\otimes\begin{bmatrix}
0 & b \\
0 & 0\end{bmatrix}\right),
\end{eqnarray*}
whenever $Z$ is a contraction. Thus, $f$ takes values in $\left(\mathcal D_{H,2n}\right)_{\rm nc}$
and for $Z=1\in\mathbb C,$ we actually reach the supremum in the above equality.

\section{A classification of noncommutative domains of 
holomorphy}
\label{sec:five}

We turn now towards a classification of noncommutative domains (with respect to the level topology) 
which contain no complex lines at any level, up to noncommutative holomorphic 
equivalence, in terms of $\tilde\delta$. In this section we assume that
$\mathcal V$ is a Banach space, and when dealing with domains defined 
by kernels, we further assume that $\mathcal V$ is an operator space.
Observe that if $f\colon\mathcal D\to
\mathcal E$ is a noncommutative automorphism
(i.e. a map which is bijective at each level, with analytic inverse), then inequality
stated in Corollary \ref{trentatre} must hold in both directions
(for $f$ and $f^{\langle-1\rangle}$), so they must become equalities.
That is,
\begin{eqnarray}\label{conformal}
\tilde\delta_\mathcal D(a,c)=\tilde\delta_\mathcal E(f(a),f(c)),\quad a,c\in\mathcal D_n,n\in\mathbb N.
\end{eqnarray}
Conversely, assume that there is a function $f$ as above such that
equality \eqref{conformal} holds for all $a,c\in
\mathcal D_{n}$, $n\in\mathbb N$.
Then it follows trivially that $f$ is injective. Indeed, if not, there would be an $n\in\mathbb N$
and points $a\neq c\in\mathcal D_n$ such that $f(a)=f(c)$. Then we would have
$0=\tilde\delta_\mathcal E(f(a),f(c))=\tilde\delta_\mathcal D(a,c)$, a contradiction, according to 
Theorem \ref{sep}, to the hypothesis that $\mathcal D$ contains no 
complex lines.

Proving the surjectivity of $f$ as a consequence of equality \eqref{conformal} is not possible in full 
generality. We make the following assumption about our domains:

Given a noncommutative set $\mathcal D$ in the noncommutative extension $\mathcal V_{\rm nc}$ of a 
Banach space $\mathcal V$, which is invariant under conjugation with scalar matrices,
$$
\text{For any }n\in\mathbb N\text{ and }a\in\mathcal D_n,\text{ if }\{c_k\}_{k\in\mathbb N}
\subset\mathcal D_n\text{ satisfies }\lim_{k\to\infty}\inf_{x\in\mathcal D_n^{\rm c}}\|x-c_k\|=0,
\text{ then }
$$
\vskip-0.5truecm\begin{equation}\label{hypo}
\lim_{k\to\infty}\tilde{\delta}_\mathcal D(a,c_k)=+\infty.
\end{equation}
This hypothesis does not exclude the possibility that $\tilde{\delta}_\mathcal D\equiv+\infty$.

\begin{thm}\label{classif}
Consider two noncommutative domains $\mathcal D$ and $\mathcal E$ in a given space 
$\mathcal V_{\rm nc}$ which are invariant under conjugation by unitary scalar matrices and contain 
no complex lines, and a noncommutative function $f\colon\mathcal D\to\mathcal E$. Assume that both 
$\mathcal D$ and $\mathcal E$ satisfy hypothesis \eqref{hypo}. Then the following are equivalent:
\begin{enumerate}

\item $f$ satisfies $\tilde{\delta}_\mathcal D(a,c)=\tilde{\delta}_\mathcal E(f(a),f(c)),a,c\in\mathcal D.$

\item $f$ is a bijective noncommutative map, with
noncommutative inverse.
\end{enumerate}

\end{thm}
The reader might worry about a trivial counterexample: the map from the nc disk to the
nc bidisk sending $z$ to $(z,0)$. However, we excluded this possibility by the way we formulated our
statement: in this case, the nc disk is equal to its boundary in the ``environment'' in which the bidisk
lives, so according to \eqref{hypo}, its $\tilde\delta$ would have to be constantly equal to infinity. 

\begin{proof} (2)$\implies$(1): This implication is trivial. We know that
$\tilde{\delta}_\mathcal D(a,c)\ge\tilde{\delta}_\mathcal E(f(a),f(c))$. If $a',c'\in\mathcal E$,
then by (2) there exist $a,c\in\mathcal D$ such that $f(a)=a',f(c)=c'$,
which means $f^{\langle-1\rangle}(a')=a,f^{\langle-1\rangle}(c')=c$. By Proposition \ref{contr},
$$
\tilde{\delta}_\mathcal E(f(a),f(c))=\tilde{\delta}_\mathcal E(a',c')\ge\tilde{\delta}_\mathcal D
\left(f^{\langle-1\rangle}(a'),f^{\langle-1\rangle}(c')\right)=\tilde{\delta}_\mathcal D(a,c).
$$

(1)$\implies$(2): We have already seen that under condition (1), $f$ is injective. Thus, we
need to show that $f$ is also surjective. Once we showed that, the noncommutativity of
the correspondence $a'\mapsto f^{\langle-1\rangle}(a')$ allows us to conclude.
The essential part of the proof is in the following quite obvious lemma, 
which we nevertheless state separately, since it might be of independent
interest.

\begin{lemma}\label{equality}
Consider a noncommutative domain $\mathcal D$ and a noncommutative subset $\mathcal D'
\subseteq\mathcal D$. Assume that both $\mathcal D$ and $\mathcal D'$ are invariant under 
conjugation by scalar unitary matrices and satisfy hypothesis \eqref{hypo}. If $\tilde{\delta}_\mathcal D
(a,c)=\tilde{\delta}_{\mathcal D'}(a,c)$ for all $a,c\in\mathcal D'$, then $\mathcal D=\mathcal D'$.
\end{lemma}
\begin{proof}
The proof of this lemma is utterly trivial: assume towards contradiction that there exist points 
in $\mathcal D\setminus\mathcal D'$. Pick a point $x\in\mathcal D\cap\partial\mathcal D'$ (by
$\partial\mathcal D'$ we understand the boundary of the set $\mathcal D'$ at the corresponding level 
$n$ in the norm topology of the Banach space $\mathcal V^{n\times n}$) and a point $a\in\mathcal D'$.
By the definition of the boundary, there exists a sequence $\{c_k\}_{k\in\mathbb N}\subset\mathcal D'$ 
converging to $x$ in norm. In particular, $\{c_k\}_{k\in\mathbb N}$ satisfies the condition of hypothesis 
\eqref{hypo}, so that $\tilde{\delta}_{\mathcal D'}(a,c_k)\to+\infty$ as $k\to\infty$. By Remark
\ref{rem-cont}, we have
$$
\infty>\tilde{\delta}_{\mathcal D}(a,x)\ge\limsup_{n\to\infty}\tilde{\delta}_{\mathcal D}(a,c_k)=
\limsup_{n\to\infty}\tilde{\delta}_{\mathcal D'}(a,c_k)=\infty,
$$
an obvious contradiction. Thus, $\mathcal D'=\mathcal D$, as claimed.
\end{proof}

Consider the set 
$f(\mathcal D)\subset\mathcal E$. For any $x,y\in f(\mathcal D)$, there exist
unique $a,c\in\mathcal D$ such that $f(a)=x,f(c)=y$. 
It follows from Proposition \ref{contr} that $\delta_{f(\mathcal D)}(x,y)(x-y)\le\delta_\mathcal D(a,c)(a-c)$. 
Since $f(\mathcal D)\subset\mathcal E,$ we necessarily have $\tilde\delta_{f(\mathcal D)}(x,y)
\ge\tilde\delta_\mathcal E(x,y)$. Together with the hypothesis of (1), we obtain
$$
\tilde\delta_\mathcal D(a,c)=\tilde\delta_\mathcal E(f(a),f(c))=
\tilde\delta_\mathcal E(x,y)\le\tilde\delta_{f(\mathcal D)}(x,y)\le\tilde\delta_\mathcal D(a,c),
$$
so that $\tilde\delta_{f(\mathcal D)}(x,y)=\tilde\delta_\mathcal E(x,y)$ for all $x,y\in
f(\mathcal D)$. By Lemma \ref{equality}, we conclude that $f(\mathcal D)=\mathcal E$.
\end{proof}

\begin{remark}\label{bdryblow-up}
It turns out that in Theorem \ref{classif} we cannot dispense with the requirement 
that $\tilde\delta$ blows up at the boundary. The following counterexample, similar to the
one in Remark \ref{MatrixConvexity}, shows what goes wrong if this requirement is dropped.
Consider a domain $D\subseteq\frac12\mathbb D\subset\mathbb C$ and define the
nc set
$$
\mathcal D=\coprod_{n=1}^\infty\{A\in\mathbb C^{n\times n}\colon\sigma(A)\subset D,\|A\|<1\}.
$$
The proof from Remark \ref{MatrixConvexity} applies to show that $\mathcal D$ is a unitarily
invariant noncommutative set which is open at each level. However, a direct computation shows that
$\left\|\begin{bmatrix}
a & b\\ 
0 & c\end{bmatrix}\right\|<1$ if and only if $aa^*+bb^*<1,cc^*<1$, and 
$bc^*(1-cc^*)^{-1}cb^*<1-aa^*-bb^*$ (this holds in an arbitrary $C^*$-algebra). 
Since the other restriction in the definition of $\mathcal D$ is on the spectrum of the
matrix $A$, it only affects $a$ and $c$; there is no other restriction on $b$.
The last inequality is equivalent to
$b((1-c^*c)^{-1}-1)b^*<1-aa^*-bb^*$, which is in its own turn equivalent to
\begin{equation}\label{NCBall}
(1-aa^*)^{-\frac12}b(1-c^*c)^{-1}b^*(1-aa^*)^{-\frac12}<1.
\end{equation}
Thus,
\begin{equation}\label{NCdisk}
\delta_\mathcal D(a,c)(b)=\left\|(1-aa^*)^{-\frac12}b(1-c^*c)^{-\frac12}\right\|
\end{equation}
However, for any choice of selfadjoints $a$ and $c$, we have that $\tilde{\delta}_\mathcal D(a,c)
\leq \frac43$. Thus, $\tilde{\delta}$ stays bounded (by $4/3$) on the intersection of the selfadjoints with
$\mathcal D$.

On the other hand, we have $\mathcal D\subsetneq\mathbb D_{\rm nc}$, the nc unit ball of $\mathbb C$,
and $\delta_{\mathbb D_{\rm nc}}|_\mathcal D=\delta_\mathcal D$.
\end{remark}

In the context of Lemma \ref{equality}, we record here the ``opposite'' case: we show that, under 
certain conditions, strict inclusion of domains leads to strict inequalities between the associated 
distances.

\begin{prop}\label{ka}
Consider an operator space $\mathcal V$. Let $\mathcal D'\subset\mathcal D$ be an inclusion of 
noncommutative domains in $\mathcal V_{\rm nc}$. Assume that 
\begin{enumerate}
\item $M:=\sup_{n\in\mathbb N}\sup_{x\in\mathcal D'_n}\|x\|<+\infty$;
\item $m:=\inf_{n\in\mathbb N}
\inf\{\|x-w\|\colon{x\in\mathcal D'_n,w\in\mathcal V^{n\times n}\setminus\mathcal D_n}\}>0.$
\end{enumerate}
Then there exists a constant $k\in[0,1)$ such that $k\delta_{\mathcal D'}\ge\delta_\mathcal D$. In particular,
$k\tilde\delta_{\mathcal D'}\ge\tilde\delta_\mathcal D$.
\end{prop}
\begin{proof}
Let $n$ be a fixed level, and pick $a,c\in\mathcal D'_n$, $b\in\mathcal V^{n\times n},b\neq0$. By definition,
$$
\delta_{\mathcal D'}(a,c)(b)^{-1}=\sup\left\{t>0\colon\begin{bmatrix}
a & sb\\
0 & c
\end{bmatrix}\in\mathcal D'_{2n} \text{ for all }s<t\right\},
$$
$$
\delta_{\mathcal D}(a,c)(b)^{-1}=\sup\left\{t>0\colon\begin{bmatrix}
a & rb\\
0 & c
\end{bmatrix}\in\mathcal D_{2n} \text{ for all }r<t\right\}.
$$
We know that the distance from $\mathcal D'_{2n}$ to $\mathcal V^{2n\times 2n}\setminus\mathcal 
D_{2n}$ is at least $m$, so that
$$
(\delta_{\mathcal D}(a,c)(b)^{-1}-s)\|b\|=\left\|\begin{bmatrix}
a & sb\\
0 & c
\end{bmatrix}-\begin{bmatrix}
a & \delta_{\mathcal D}(a,c)(b)^{-1}b\\
0 & c
\end{bmatrix}\right\|\ge m
$$
whenever $\begin{bmatrix}
a & sb\\
0 & c
\end{bmatrix}\in\mathcal D'_{2n}$,
and thus, $\delta_{\mathcal D}(a,c)(b)^{-1}-\delta_{\mathcal D'}(a,c)(b)^{-1}\ge
\frac{m}{\|b\|}$ for any $a,c\in\mathcal D'_{n},b\neq0$. It follows that
$$
\frac{\delta_{\mathcal D'}(a,c)(b)}{\delta_{\mathcal D}(a,c)(b)}\ge
1+\frac{m\delta_{\mathcal D'}(a,c)(b)}{\|b\|}=1+m
\delta_{\mathcal D'}(a,c)\left(\frac{b}{\|b\|}\right).
$$
We bound from below $\delta_{\mathcal D'}(a,c)(b)$ when $\|b\|=1$ and $a,c\in\mathcal D'_{2n}$.
We have $\delta_{\mathcal D'}(a,c)(b)>\xi\iff \delta_{\mathcal D'}(a,c)(b)^{-1}<\xi^{-1}$; but
any element in $\mathcal D'_{2n}$ has norm bounded from above by $M$, so $\begin{bmatrix}
a & sb\\
0 & c
\end{bmatrix}\in\mathcal D'_{2n}$ implies $|s|=\|sb\|\leq M$. Thus, $\delta_{\mathcal D'}(a,c)(b)\ge
M^{-1}.$ We obtain 
$
\frac{\tilde\delta_{\mathcal D'}(a,c)}{\tilde\delta_{\mathcal D}(a,c)}\ge
1+\frac{m}{M},
$
for the constant $k=\frac{M}{m+M}<1$. Taking $b=a-c,a\neq c$, yields the result for $\tilde\delta$.
\end{proof}

For our four distances, we have
\begin{cor}\label{kar}
Under the assumptions, and with the notations, of Proposition \ref{ka},
we have $k{d}_{\mathcal D'}\ge{d}_{\mathcal D}|_{\mathcal D'}, 
k{d}_{\mathcal D',\infty}\ge{d}_{\mathcal D,\infty}|_{\mathcal D'},
k\tilde{d}_{\mathcal D'}\ge\tilde{d}_{\mathcal D}|_{\mathcal D'}$, and 
$k\tilde{d}_{\mathcal D',\infty}\ge\tilde{d}_{\mathcal D,\infty}|_{\mathcal D'}.$
\end{cor}
\begin{proof}
For any $n\in\mathbb N$, $a,c\in\mathcal D'_{n}, p\in\mathbb N$, and division $I_p\otimes a=a_0,a_1,\dots,a_N=I_p\otimes c\in
\mathcal D'_{pn}$, we have
$$
k\sum_{j=1}^N\tilde\delta_{\mathcal D'}(a_{j-1},a_j)>\sum_{j=1}^N\tilde\delta_{\mathcal D}(a_{j-1},a_j).
$$
Taking infimum after all divisions at all levels $p\in\mathbb N$ in the left hand side provides $k\tilde{d}_{\mathcal D',\infty}(a,c)$. Increasing the number of divisions
in the right hand side can only decrease the infimum, so that $k\tilde{d}_{\mathcal D',\infty}
(a,c)\ge\tilde{d}_{\mathcal D,\infty}(a,c)$. Same argument, with ``divisions'' replaced by ``continuously differentiable paths''
yields the result for ${d}_{\mathcal D',\infty}(a,c),{d}_{\mathcal D,\infty}(a,c)$.
The result for $d,\tilde{d}$ is proved the same way, except we do not take infimum after $p$.
\end{proof}

As a side benefit, we obtain from the proof of Proposition \ref{ka} that on bounded domains in
operator spaces, $\tilde\delta$ and the norm are locally equivalent. We have already seen in 
Proposition \ref{dreiwolf} that if $\|a_k-a\|\to0$, then $\tilde\delta_\mathcal D(a_k,a)\to0$
and thus $\tilde{d}_\mathcal D(a_k,a)\to0$. Now assume that in a bounded domain $\mathcal D$
we have a sequence $\{a_k\}_{k\in\mathbb N}\subset\mathcal D$ and a point $a\in\mathcal D_n$
so that $\tilde{d}_\mathcal D(a_k,a)\to0$ as $k\to\infty$. We have seen in the proof of
Proposition \ref{ka} that $\delta_\mathcal D(a,c)(b)\ge M^{-1}$ if $\mathcal D$ is included
in a norm-ball of radius $M$, uniformly in $a,c\in\mathcal D_n,b\in\mathcal V^{n\times n}$, $\|b\|=1$,
$n\in\mathbb N$. Thus, $\tilde\delta_{\mathcal D}(a,c)\ge M^{-1}\|a-c\|$, so that
for any division $a=a_0,a_1,\dots,a_N=c$ of $\tilde\delta_{\mathcal D}(a,c)$, we have
$\sum_{j=1}^N\tilde\delta_{\mathcal D}(a_{j-1},a_j)\ge M^{-1}\sum_{j=1}^N\|a_j-a_{j-1}\|
\ge M^{-1}\|a-c\|.$ Thus, $\tilde{d}_\mathcal D(a,c)\ge M^{-1}\|a-c\|$. Applying this to 
$c=a_k$ yields $\lim_{k\to\infty}\|a-a_k\|=0$. Obviously, nothing changes if we amplify $a,c$ by $I_p$.
We have proved

\begin{prop}\label{coincidence}
If $\mathcal D$ is a bounded nc domain in an operator space $\mathcal V$ and $n\in\mathbb N$, then 
on any subset $A\subset\mathcal D_n$ which is at a positive distance from $\mathcal D_n^c$, the 
topologies induced by $d_\mathcal D,d_{\mathcal D,\infty},\tilde{d}_\mathcal D,\tilde{d}_{\mathcal D,\infty}$,
and the norm of $\mathcal V^{n\times n}$ coincide.
\end{prop}

\begin{remark}
A very similar proof shows that the result stated in Proposition \ref{coincidence} holds also
for bounded strict subsets of half-planes. 
\end{remark}

These results yield easily a noncommutative version of the Earle-Hamilton Theorem \cite{EH}.
The proof of \cite[Theorem 2]{H} translates verbatim (with, in the notations of \cite{H}, 
$\alpha(x,v)$ replaced by $\delta_\mathcal D(x,x)(v)$ and the norm unit ball in a Banach space replaced by the nc unit ball in an operator space).

\begin{prop}
Let $\mathcal V,\mathcal D$, and $\mathcal D'$ be as in Proposition \ref{ka}.
If $f\colon\mathcal D\to\mathcal D'$ is a noncommutative function, then there exists a unique
attracting fixed point $w\in\mathcal D'$ for $f$ in the sense of \cite{AKV}.
\end{prop}

\section{An application to a problem in free probability}
\label{sec:six}

In this section, we use some of the tools introduced before in order to study a problem in 
free probability. We consider a $W^*$-noncommutative probability space $(M, E,B)$, where 
$B\subseteq M$ is a unital inclusion of $W^*$-algebras and $E\colon M\to B$
is a weakly-continuous unit-preserving conditional expectation. Elements in $M$ are called operator-valued
(or, sometimes, $B$-valued) random variables. If $X=X^*\in M$, we define the 
{\em distribution of $X$ with respect to $E$} to be the collection of multilinear maps
$$
\mu_X=\{m_{n,X},n\in\mathbb N\},
$$
called {\em moments}, where $m_{0,X}=1\in B\subseteq M$, $m_{1,X}=E[X]\in B$, and
$$
m_{n,X}\colon\underbrace{B\times\cdots\times B}_{n-1\textrm{ times}}\to B,
\ m_{n,X}(b_1,\dots,b_{n-1})=E[Xb_1Xb_2\cdots Xb_{n-1}X],n>1.
$$
Such distributions are encoded analytically by the noncommutative Cauchy-Stieltjes
transform (see Example \ref{example21}(3)):
$$
G_{X,n}(b)=({\rm Id}_{\mathbb C^{n\times n}}\otimes E)\left[(b-I_n\otimes X)^{-1}\right],\quad n\in\mathbb N,b\in B^{n\times n},\Im b>0.
$$
This is a noncommutative function mapping the noncommutative upper half-plane
of $B$ into the noncommutative lower half-plane (see, for instance, \cite{V2}).
It has several good properties, including the fact that $\Im G_{X,n}(b)<0$, so that
$F_{X,n}(b):=G_{X,n}(b)^{-1}$ exists and maps elements of positive imaginary part 
into elements of positive imaginary part. Moreover, it has been shown in
\cite{BPV1} that $\Im F_{X,n}(b)\ge\Im b$, so that $h_{X,n}(b):=F_{X,n}(b)-b,$
$\Im b>0$, takes values elements of nonnegative imaginary part.

It has been shown in \cite{ABFN} that for any given selfadjoint $X\in M$ 
and completely positive map $\rho\colon B\to B$ such that $\rho-
\mathrm{Id}_B$ is still completely positive on $B$, there exists
a selfadjoint $X_\rho$ in a possibly larger $W^*$-algebra 
containing $M$ such that $E$ extends to this possibly
larger algebra and the following relations hold:
\begin{equation}\label{semigroup}
G_{X_\rho,n}(b)=G_{X,n}(\omega_\rho(b)),\quad\omega_\rho(b)
=b+(\rho-
\mathrm{Id}_B)h_{X,n}(\omega_\rho(b)),\quad\Im b>0,n\in\mathbb N.
\end{equation}
In terms of the free probability significance of $X_\rho$, we only
mention that $\mu_{X_\rho}=\mu_X^{\boxplus\rho}$, and refer 
the interested reader to \cite{ABFN} for details. We wish to
mention, however, that, thanks to a trick due to Hari Bercovici, 
understanding free convolution powers indexed by 
completely positive maps suffices in order to 
understand free additive convolutions of operator-valued
distributions, so, in a certain sense, $\{\mu_X^{\boxplus\rho}\colon
\rho\textrm{ and }\rho-
\mathrm{Id}_B\textrm{ completely positive}\}$ is the most
general object to be understood in the context of free 
convolutions of operator-valued distributions. 

All of the above has been done for selfadjoint
operators that belong to $M$, that is, bounded 
selfadjoint operators. We will apply our results in 
order to show that, under certain hypotheses,
this can be also done for unbounded operators 
$X=X^*$ affiliated to $M$, generalizing the results of 
\cite{BVIUMJ} to the operator-valued context. We make 
the following hypotheses regarding $X,M$, and $B$:
\begin{enumerate}
\item[(H1)] $B$ and $X$ generate a ${}^*$-algebra of possibly unbounded
densely defined operators $B\langle X\rangle$, such that the spectral 
projections of any selfadjoint element of $B\langle X\rangle$
belong to $M$. In particular, the (classical) distribution of any selfadjoint 
element from $B\langle X\rangle$ with respect to any weakly continuous state on 
$M$ must be a probability measure (we assume weakly continuity for states from now on);
\item[(H2)] $E\left[\Im(b-X)^{-1}\right]<0$ whenever $\Im b>0$ in $B$.
\end{enumerate}
Hypothesis (H1) is very natural, and allows us to extend the notion of
$B$-valued distribution of $X$ with respect to $E$ to the case when $X=X^*$
is unbounded, but affiliated to $M$. Thus,

\begin{defn}\label{distri}
Let $(M, E,B)$ be a $W^*$-noncommutative probability space and $X=X^*$ be
affiliated to $M$, possibly unbounded. Assume that $X,M,$ and $B$ satisfy Hypothesis
(H1). We define the $B$-valued distribution $\mu_X$ of $X$ with respect to $E$ to be the 
collection of Borel probability measures on $\mathbb R$, $\{\mu_{P,\varphi}\colon P=P^*\in B\langle X\rangle,
\varphi\colon B\to\mathbb C\text{ state}\},$
where $\mu_{P,\varphi}$ is determined via the following equality:
$$
\int_{\mathbb R}\frac{1}{z-t}\,{\rm d}\mu_{P,\varphi}(t)=\varphi\left(E\left[(z-P)^{-1}\right]\right),\quad z\in\mathbb C\setminus\mathbb R.
$$
\end{defn}

The above objects are indeed well-defined: by (H1), $(z-P)^{-1}\in M$ is a 
bounded, normal operator. Thus, one may apply $E$ to it in order to obtain
an element of $B$. Since $B$ is a $C^*$-algebra, the set of states on $B$ 
separates its elements: if $b,b'\in B$ are such that $\varphi(b)=\varphi(b')$ 
for all states $\varphi$ on $B$, then $b=b'$. Since for any fixed state $\theta$
on $M$ we have $\lim_{y\to0}\theta(P(iy-P)^{-1})=0$, it follows that
$\mu_{P,\varphi}$ is a probability measure. Moreover, if $X=X^*$ is 
bounded, then the notion of distribution from Definition \ref{distri}
coincides with the classical notion of $B$-valued distribution. Indeed, 
as any element in a $C^*$-algebra is the linear combination of two 
selfadjoint elements of the same  $C^*$-algebra, it follows that 
the collection of moments of $\mu_X$ is determined by the values
$E$ takes on selfadjoint polynomials $P\in B\langle X\rangle$, and
vice-versa. Since elements $b\in B$ are uniquely determined by $\{
\varphi(b)\colon\varphi\colon B\to\mathbb C\text{ state}\}$, we 
conclude that the two definitions are equivalent for bounded variables $X=X^*\in M$.

 Hypothesis (H1) is clearly satisfied under the assumption that $M$ is a finite factor
and $E$ is the trace-preserving conditional expectation.

Hypothesis (H2) deserves a few more comments. It is natural in terms
of allowing for the analytic functions tools (including the $R$-transform
of Voiculescu - see \cite{V*,VFAQ2}) to be deployed. But it can be also
viewed as a measure of nondegeneracy of $E$: indeed, let $b=u+iv,
u=u^*,v>0$. Then 
\begin{eqnarray*}
E\left[\Im(b-X)^{-1}\right] & = & E\left[\Im ((u-X)+iv)^{-1}\right]\\
& = & -v^{-\frac12}
E\left[\left(\left(v^{-\frac12}(u-X)v^{-\frac12}\right)^2+1\right)^{-1}\right]
v^{-\frac12},
\end{eqnarray*}
so that $E\left[\Im(b-X)^{-1}\right]<0$ if and only if 
$E\left[\left(\left(v^{-\frac12}(u-X)v^{-\frac12}\right)^2+1\right)^{-1}\right]>0$. 
It is clear that, since $X$ is unbounded, $\min
\sigma\left(\left(\left(v^{-1/2}(u-X)v^{-1/2}\right)^2+1\right)^{-1}\right)=0$. Also,
$k:=\left\|\left(\left(v^{-1/2}(u-X)v^{-1/2}\right)^2+1\right)^{-1}\right\|\leq1$.
Thus, non-invertibility of $E\left[\Im(b-X)^{-1}\right]$ becomes equivalent
to the equality
$$
\left\|E\left[k-\left(\left(v^{-\frac12}(u-X)v^{-\frac12}\right)^2+1\right)^{-1}\right]\right\|
=\left\|k-\left(\left(v^{-\frac12}(u-X)v^{-\frac12}\right)^2+1\right)^{-1}\right\|.
$$
That is, $E$ is isometric on a positive element which is not in $B$. Thinking in terms of the 
duals of $M$ and $B$, respectively, this tells us that there exists an element $\varphi$
of norm one in the dual $B^*$ of $B$ such that 
$\left(\left(v^{-\frac12}(u-X)v^{-\frac12}\right)^2+1\right)^{-1}-k$ 
reaches its norm on $\varphi\circ E$. Thus, hypothesis (H2) is implied
by the requirement that positive elements in $M$ but not in $B$ do not reach
their norms on $B^*\circ E$. It may be worth mentioning that in the case of
a tracial $W^*$-probability space with normal faithful trace state $\tau$ 
which is left invariant by $E$, hypothesis (H2) comes to stating that 
$\left(\left(v^{-\frac12}(u-X)v^{-\frac12}\right)^2+1\right)^{-1}-k$
does not reach its norm on $L^2(B,\tau)$, and in the case when $B$ is
finite dimensional, (H2) is equivalent to not allowing algebraic relations 
between $X$ and elements in $B$.

In this section, we shall show that 
the fixed point equation \eqref{semigroup}
has a nontrivial solution also when $X$ is unbounded, but
still satisfies hypotheses (H1) and (H2) above. While, according to
 \cite{Wil,PPT}, the noncommutative extension $G_{X,n}(b)=({\rm Id}_{\mathbb C^{n\times n}}\otimes E)\left[(b-I_n\otimes X)^{-1}\right]$, $n\in\mathbb N$, does
not characterize the distribution of $X$ for all possible unbounded random variables $X$ as above, it follows quite easily 
(see, for instance, \cite{HT}) that the extension $\{\alpha\in\mathbb C^{n\times n}\colon\alpha=\alpha^*\}\times H^+_n(B)\ni(\alpha,b)\mapsto G_{\alpha\otimes X}(b)
=({\rm Id}_{\mathbb C^{n\times n}}\otimes E)\left[(b-\alpha\otimes X)^{-1}\right]$, $n\in\mathbb N$,
does encode the distribution of $X$ with respect to $E$, as defined in Definition \ref{distri}. For the reader's comfort, we will sketch the argument in Remark \ref{li} below.

\begin{thm}\label{semigr}
Consider a noncommutative function $h\colon H^+(B)\to B$ such that $\Im h(b)\ge0$  
and $\lim_{y\to+\infty}\frac{\Im h(\Re b+iy\Im b)}{y}=0$ in the {\rm wo}-topology for all $b\in H^+(B)$.
For any given $b>0$, the map $w\mapsto b+h(w)$ has a unique attracting
fixed point in $H^+(B)$, to be denoted by $\omega(b)$, and the correspondence 
$b\mapsto\omega(b)$ is a noncommutative self-map of $H^+(B)$, hence,
in particular, analytic.
\end{thm}
\begin{proof}
It is clearly enough to prove the theorem at level one: fix
$b_0\in B$ such that $\Im b_0>\varepsilon_01>0$. For any $n\ge1$, 
the map $h_0\colon w\mapsto b_0\otimes I_n+h(w)$ sends $H^+(B^{n\times n})$
into $H^+(B^{n\times n})+i\varepsilon_01,$ so that, as a 
noncommutative map, it sends $H^+(B)$ to 
$H^+(B)+i\varepsilon_01.$ We re-write the proof of Corollary \ref{trentatre}
for this context: if $\Im\begin{bmatrix} a & b \\ 0 & c\end{bmatrix}
\in H^+(B^{2\times 2})$, then $\Im h_0\left(\begin{bmatrix} 
a & b \\ 0 & c\end{bmatrix}\right)
\in H^+(B^{2\times 2})+i\varepsilon_01.$ That means
$(\Im h_0(a)-\varepsilon_01)^{-1/2}\Delta h_0(a,c)(b)
(\Im h_0(c)-\varepsilon_01)^{-1}\Delta h_0(a,c)(b)^*
(\Im h_0(a)-\varepsilon_01)^{-1/2}\leq\|(\Im a)^{-1/2}b(\Im c)^{-1/2}\|^2\cdot1$
for all $a,c\in H^+(B)$, $b\in B$. We re-write this as
$$
\Delta h_0(a,c)(b)
(\Im h_0(c)-\varepsilon_01)^{-1}\Delta h_0(a,c)(b)^*\leq
\|(\Im a)^{-\frac12}b(\Im c)^{-\frac12}\|^2(\Im h_0(a)-\varepsilon_01).
$$
Multiplying left and right by $(\Im h_0(a))^{-1/2}$,  we obtain
\begin{eqnarray*}
\lefteqn{(\Im h_0(a))^{-1/2}\Delta h_0(a,c)(b)
(\Im h_0(c)-\varepsilon_01)^{-1}\Delta h_0(a,c)(b)^*(\Im h_0(a))^{-1/2}}\\
& \leq & \|(\Im a)^{-\frac12}b(\Im c)^{-\frac12}\|^2(1-\varepsilon_0(\Im h_0(a))^{-1})\\
& \leq & \|(\Im a)^{-\frac12}b(\Im c)^{-\frac12}\|^2\|1-\varepsilon_0(\Im h_0(a))^{-1}\|\\
& = & \|(\Im a)^{-\frac12}b(\Im c)^{-\frac12}\|^2(1-\varepsilon_0\|(\Im h_0(a))^{-1}\|).
\quad\quad\quad\quad\quad\quad\quad\quad\quad
\end{eqnarray*}
Since $xx^*\leq M\cdot1\iff x^*x\leq M\cdot1$, we immediately 
obtain 
\begin{eqnarray}
\lefteqn{(\Im h_0(a))^{-1/2}\Delta h_0(a,c)(b)
(\Im h_0(c))^{-1}\Delta h_0(a,c)(b)^*(\Im h_0(a))^{-1/2}}\nonumber\\
& \leq & \|(\Im a)^{-\frac12}b(\Im c)^{-\frac12}\|^2
(1-\varepsilon_0\|(\Im h_0(a))^{-1}\|)
(1-\varepsilon_0\|(\Im h_0(c))^{-1}\|).\label{strict-deriv}
\end{eqnarray}
Applying this to $b=a-c$ yields
\begin{eqnarray}
\lefteqn{(\Im h_0(a))^{-1/2}(h_0(a)-h_0(c))
(\Im h_0(c))^{-1}(h_0(a)-h_0(c))^*(\Im h_0(a))^{-1/2}}\nonumber\\
& \leq & \|(\Im a)^{-\frac12}(a-c)(\Im c)^{-\frac12}\|^2
(1-\varepsilon_0\|(\Im h_0(a))^{-1}\|)
(1-\varepsilon_0\|(\Im h_0(c))^{-1}\|).\label{strict}
\end{eqnarray}
It thus follows that if $\omega(b_0)\in H^+(B)+i\varepsilon_01$ 
is a fixed point for $h_0$, then it must be the unique and 
attracting fixed point of $h_0$. Indeed, for an arbitrary 
$a\in H^+(B)$, if we let 
$r=\|(\Im a)^{-1/2}(a-\omega(b_0))(\Im\omega(b_0))^{-1/2}\|$,
it follows that 
$$
h_0(B(\omega(b_0),2r))\subset B(\omega(b_0),2r),
$$
where, as in \cite[Proposition 3.2]{JLMS}, we denote
\begin{equation}\label{pseudoball}
B(c,t)=\{a\in B\colon\|(\Im a)^{-1/2}(a-c)(\Im c)^{-1/2}\|\leq t\}.
\end{equation}
As shown in \cite[Proposition 3.2]{JLMS},
$B(\omega(b_0),2r)$ is bounded in norm in the sense that
$$
\|d\|\leq\|\Re\omega(b_0)\|+\|\Im\omega(b_0)\|\left(
2r^2+1+2r\sqrt{r^2+1}+2r\sqrt{2r^2+1+2r\sqrt{r^2+1}}\right)
,$$
and that it is bounded away from the boundary of $H^+(B)$ in the sense that
$$
\Im d\ge({2+4r^2})^{-1}\Im\omega(b_0), \quad d\in  B(\omega(b_0),2r).
$$


Thus, for any $N\in\mathbb N$, we have, by an iteration of \eqref{strict},
\begin{eqnarray}
\lefteqn{\left\|(\Im h_0^{\circ N}(a))^{-\frac12}(h_0^{\circ N}(a)-\omega(b_0))
(\Im\omega(b_0))^{-\frac12}\right\|^2}\nonumber\\
& = & \left\|(\Im h_0^{\circ N}(a))^{-\frac12}(h_0^{\circ N}(a)-h_0^{\circ N}(\omega(b_0)))
(\Im h_0^{\circ N}(\omega(b_0)))^{-\frac12}\right\|^2\nonumber\\
& \leq & \|(\Im a)^{-\frac12}(a-\omega(b_0))(\Im\omega(b_0))^{-\frac12}\|^2\nonumber\\
& & \mbox{}\times\prod_{j=1}^N(1-\varepsilon_0\|(\Im h_0^{\circ j}(a))^{-1}\|)
(1-\varepsilon_0\|(\Im h_0^{\circ j}(\omega(b_0)))^{-1}\|)\nonumber\\
& = & \|(\Im a)^{-\frac12}(a-\omega(b_0))(\Im\omega(b_0))^{-\frac12}\|^2\nonumber\\
& & \mbox{}\times(1-\varepsilon_0\|(\Im\omega(b_0))^{-1}\|)^N
\prod_{j=1}^N(1-\varepsilon_0\|(\Im h_0^{\circ j}(a))^{-1}\|).\label{iter}
\end{eqnarray}
Letting $N$ go to infinity sends $(1-\varepsilon_0\|(\Im\omega(b_0))^{-1}\|)^N$ to 
zero, so that 
$$
\lim_{N\to\infty}\left\|(\Im h_0^{\circ N}(a))^{-\frac12}(h_0^{\circ N}(a)-\omega(b_0))
(\Im\omega(b_0))^{-\frac12}\right\|^2=0.
$$
Recall that $x^{-1}\ge\frac{1}{\|x\|}$ for any positive operator $x$. Since 
$$
({2+4r^2})^{-1}\Im\omega(b_0)\leq\Im h_0^{\circ N}(a)\leq
\|\Re\omega(b_0)\|+\|\Im\omega(b_0)\|\left(4r+1\right)^2,
$$
we have 
\begin{eqnarray*}
\lefteqn{\left\|(\Im h_0^{\circ N}(a))^{-\frac12}(h_0^{\circ N}(a)-\omega(b_0))
(\Im\omega(b_0))^{-\frac12}\right\|^2}\\
 & \geq & 
\frac{\|h_0^{\circ N}(a)-\omega(b_0)\|^2}{\|\Im\omega(b_0)\|\|\Im h_0^{\circ N}(a)\|}\\
& \ge & \frac{\|h_0^{\circ N}(a)-\omega(b_0)\|^2}{\|\Im\omega(b_0)\|
\|\Re\omega(b_0)\|+\|\Im\omega(b_0)\|^2\left(4r+1\right)^2},
\end{eqnarray*}
which allows us to conclude that 
$$
\lim_{N\to\infty}\left\|h_0^{\circ N}(a)-\omega(b_0))\right\|=0,
$$
uniformly on bounded sets which are at strictly positive 
norm-distance from $B\setminus H^+(B)$.

Iterating in relation \eqref{strict-deriv} for $a=c=\omega(b_0)$ 
yields
\begin{eqnarray*}
\frac{\|[h'_0(\omega(b_0))]^{\circ N}(b)\|}{\|\Im\omega(b_0)\|} & \leq & 
\|(\Im\omega(b_0))^{-\frac12}[h'_0(\omega(b_0))]^{\circ N}(b)(\Im\omega(b_0))^{-\frac12}\|\\
& \leq & 
\|(\Im\omega(b_0))^{-\frac12}b(\Im\omega(b_0))^{-\frac12}\|
(1-\varepsilon_0\|(\Im\omega(b_0))^{-1}\|)^N\\
& \leq & \|(\Im\omega(b_0))^{-1}\|(1-\varepsilon_0\|(\Im\omega(b_0))^{-1}\|)^N\|b\|,
\end{eqnarray*}
which implies that 
$$
\|[h'_0(\omega(b_0))]^{\circ N}(b)\|\leq\|\Im\omega(b_0)\|
\|(\Im\omega(b_0))^{-1}\|(1-\varepsilon_0\|(\Im\omega(b_0))^{-1}\|)^N\|b\|
$$
for all $b\in B$, so that 
$$
\|[h'_0(\omega(b_0))]^{\circ N}\|\leq\|\Im\omega(b_0)\|
\|(\Im\omega(b_0))^{-1}\|(1-\varepsilon_0\|(\Im\omega(b_0))^{-1}\|)^N,
$$
the norm of $[h'_0(\omega(b_0))]^{\circ N}$ being the norm
of a bounded linear self-map of the $C^*$-algebra $B$. Thus, for 
$N>\frac{\log(\|\Im\omega(b_0)\|
\|(\Im\omega(b_0))^{-1}\|)}{-\log(1-\varepsilon_0\|(\Im\omega(b_0))^{-1}\|)}$, 
we have $\|[h'_0(\omega(b_0))]^{\circ N}\|<1$. In general, if a linear 
operator $T$ on a Banach space $B$ satisfies $\|T^N\|<1$, we may write
$\sum_{j=0}^{kN-1}T^j=(1+T+\cdots+T^{N-1})+
T^N(1+T+\cdots+T^{N-1})+T^{2N}(1+T+\cdots+T^{N-1})
+\cdots+T^{(k-1)N}(1+T+\cdots+T^{N-1})=
(1+T+\cdots+T^{N-1})\sum_{j=0}^{k-1}(T^N)^j,$
which tends to $(1+T+\cdots+T^{N-1})(1-T^N)^{-1}$ as $k\to\infty$.
Since $N$ is fixed, it follows easily that in fact so does $\sum_{j=0}^{k}T^j$.
A simple algebraic manipulation shows that $(1+T+\cdots+T^{N-1})(1-T^N)^{-1}=
(1-T)^{-1}.$ Thus, $\mathrm{Id}_B-h'_0(\omega(b_0))$ is invertible as
a linear self-map of the Banach space $B$. By the implicit function theorem 
for analytic maps on Banach spaces, it follows that $\omega$ depents 
analytically on $b_0$. This result, together with the properties of
fixed points for noncommutative maps proved in \cite{AKV} allow us to
conclude that $\omega$ is a noncommutative map on a noncommutative
neighbourhood of $b_0$.

All of the above has been established under the assumption that a fixed
point $\omega(b_0)$ exists. We have not proved its existence, though. 
Relation \eqref{strict} would allow us easily to prove such an existence
along the lines of the above proof if we could somehow guarantee
the boundedness of the iterates $\{h_0^{\circ N}(a)\}_{N\in\mathbb N}$
for some given $a\in H^+(B)$. Unfortunately, this does not seem possible
to do in a direct way. Thus, we show the existence of the fixed point
$\omega(b_0)$ by a perturbative argument, most of which is 
contained in the following proposition, which, we believe, might be
of independent interest. Define 
\begin{equation}\label{kazero}
k_0(a)=-h_0(-a^{-1})^{-1},\quad a\in H^+(B). 
\end{equation}
As $\Im h(a)>\varepsilon_01$, it follows that 
$k_0(H^+(B))\subseteq\{w\colon\|w-i(2\varepsilon_0)^{-1}1\|<
(2\varepsilon_0)^{-1}\},$ the noncommutative ball centered at
an imaginary multiple of the identity.
\begin{prop}\label{horro}
For any $a\in H^+(B)\cup\{0\}$, the fixed-point equation $x=a+k_0(x)$
has a unique solution $x(a)$ in $H^+(B)$. $x(a)$ is a noncommutative 
function of $a$ whenever $a\in H^+(B)$, and $x(0_m\oplus0_n)=
x(0_m)\oplus x(0_n)$ for all $m,n\in\mathbb N$.
\end{prop}
\begin{proof}
Note that the set $a+k_0(H^+(B))$ is bounded and bounded 
away from the complement of $H^+(B)$. Thus, the argument
used above allows us to conclude the existence, uniqueness
and analyticity of $x$ on $H^+(B)$. The existence of $x(0)$
in $H^+(B)$ is the only difficult part of the proof. For this,
we shall use some results from \cite{B-CAOT}, specifically
Proposition 3.1, Remark 3.2(2), and Corollary 3.3, together with the definition
of a noncommutative version of horodisks in the noncommutative
upper half-plane (see \cite[Relation (22)]{B-CAOT}). These results 
have been formulated for functions of a slightly different nature, 
but it is very easy to see that all elements of the proofs involved
adapt to bounded functions like $k_0$ which satisfy $k_0(a^*)^*=k_0(a)$.

We claim that $x(H^+(B))=\{m+in\colon m=m^*,n>\Im k_0(m+in)\}$.
Since $x(a)=a+k_0(x(a))$, the inclusion $\subseteq$ is quite obvious.
To prove $\supseteq$, recall that the map $B^{\rm sa}\ni p\mapsto
\Re x(p+iq)\in B^{\rm sa}$ is a bijection for any given $q>0$ (see
\cite[Corollary 3.3]{B-CAOT}). We also know that there exists a 
smooth function $g_q\colon B^{\rm sa}\to\{b\in B\colon\Im b>0\}$
such that $g_q(\Re x(p+iq))=\Im x(p+iq)$. In particular, for any 
$m\in B^{\rm sa}$, there exists a unique $n>0$ such that $g_q(m)
=n$: we have
$$
m+in=x(p+iq)=p+iq+k_0(x(p+iq))=p+iq+k_0(m+in),
$$
so that $p=m-\Re k_0(m+in), q=n-\Im k_0(m+in)$. This proves $\supseteq$.

Since $k_0(H^+(B))$ is bounded, it folows that for any pair 
$m=m^*,n>0$, we have $yn>k_0(m+iyn)$ for all 
sufficiently large $y\in(0,+\infty)$. Thus, we may define
$$
0\leq t_{m,n}=\inf\{y>0\colon sn>\Im k_0(m+isn)\text{ for all }s>y\}.
$$
We argue that for all $s>t_{m,n}$, we have $sn>\Im k_0(m+isn)$, and
for all $0\le s\le t_{m,n}$, we have $sn\not>\Im k_0(m+isn)$. 
The argument is virtually identical to the one in \cite[Lemma 5.8]{BC}
and is based on related works in the case of scalar, classical distributions 
by Biane \cite{Biane} and by Huang \cite{Huang}, so we will only sketch it. We consider the map
$\mathbb C^+\ni z\mapsto\varphi(m+zn-k_0(m+zn))\in\mathbb C$
for an arbitrary state $\varphi$. If $H(z)=\frac{\varphi(m)}{\varphi(n)}
+z-\frac{\varphi(k_0(m+zn))}{\varphi(n)}$, then $\lim_{y\to+\infty}
\frac{H(iy)}{iy}=1$ and $\Im H(z)\leq\Im z$. Then Huang's version
\cite[Section 3]{Huang} of Biane's results \cite[Lemmas 2 and 4]{Biane}
applies to $H$ to guarantee that if $\frac{\Im\varphi(k_0(m+iy_0n))}{\varphi(n)}
\ge y_0$, then $\frac{\Im\varphi(k_0(m+iyn))}{\varphi(n)}\ge y$ for all 
$y\in(0,y_0]$. Since this holds for any state $\varphi$, our claim follows.

Obviously, there are two possibilities: either $t_{m,n}>0$ or $t_{m,n}=0$.
Consider first the case when $t_{m,n}=0$. Pick a state $\varphi$ on $B$
and $n'>0$. We have
$$
\left\|(yn)^{-\frac12}(yn-yn')(yn')^{-\frac12}\right\|^2\ge
\frac{|\varphi(k_0(m+iyn))-\varphi(k_0(m+iyn'))|^2}{\Im\varphi(k_0(m+iyn))\Im\varphi(k_0(m+iyn'))},
$$
which in its own turn implies
$$
\left\|n^{-\frac12}(n-n')(n')^{-\frac12}\right\|^2\ge
\left|\frac{\Im\varphi(k_0(m+iyn))}{\Im\varphi(k_0(m+iyn'))}-
\frac{\Im\varphi(k_0(m+iyn'))}{\Im\varphi(k_0(m+iyn))}\right|^2.
$$
As $t_{m,n}=0$, we have $\frac{\varphi(\Im k_0(m+iyn))}{y}<\varphi(n)$
for all $y>0$, so that necessarily 
\begin{eqnarray*}
\lefteqn{0\leq\liminf_{y\to0}\frac{\Im\varphi(k_0(m+iyn'))}{y}}\\
& \leq & 
\frac{\varphi(n)}{2}\left(2+\left\|n^{-\frac12}(n-n')(n')^{-\frac12}\right\|^2
+\sqrt{\left(2+\left\|n^{-\frac12}(n-n')(n')^{-\frac12}\right\|^2\right)^2-4}\right).
\end{eqnarray*}
As this holds for any $n'>0$ and any state $\varphi$ on $B$,
we conclude that $k_0$ satisfies the hypotheses of \cite[Theorem 2.3]{JLMS}.
Thus, if there exists a pair $m=m^*,n>0$ such that $t_{m,n}=0$, then for any $n'>0$,
$$
\lim_{y\to0}k_0(m+iyn')=\alpha=\alpha^*
$$
exists in the norm topology. However, observe that since 
$k_0(H^+(B))\subseteq\{w\colon\|w-i(2\varepsilon_0)^{-1}1\|<
(2\varepsilon_0)^{-1}\},$ and the limit is in norm, we must have 
$\alpha=0$.

Now consider the case when $t_{m,n}>0$. As seen above, for any $y>t_{m,n}$,
there exist $p_y=m-\Re k_0(m+iyn),q_y=yn-\Im k_0(m+iyn)$ such that
$x(p_y+iq_y)=p_y+iq_y+k_0(x(p_y+iq_y))=p_y+iq_y+k_0(m+iyn)$ (in particular, $q_y>0$).
This provides the expression of $\lim_{y\to t_{m,n}}x(p_y+iq_y)=m+it_{m,n}n\in H^+(B)$.
Simple continuity guarantees that $x(p_{t_{m,n}}+iq_{t_{m,n}})=
p_{t_{m,n}}+iq_{t_{m,n}}+k_0(x(p_{t_{m,n}}+iq_{t_{m,n}}))$. Thus, 
$H^+(B)\ni w\mapsto p_{t_{m,n}}+iq_{t_{m,n}}+k_0(w)\in H^+(B)$ 
has a fixed point in $H^+(B)$. Since the range of this map is bounded
in the unbounded set $H^+(B)$, the fixed point is necessarily unique 
and attracting (indeed, one can apply the argument from the first part of the proof 
of Theorem \ref{semigr}, for ex., to the map 
$\{w+p_{t_{m,n}}+iq_{t_{m,n}}\colon\|w-i\varepsilon_0^{-1}1\|<
\varepsilon_0^{-1}\}\ni w\mapsto p_{t_{m,n}}+iq_{t_{m,n}}+k_0(w)\in
\{w\colon\|w-i(2\varepsilon_0)^{-1}1\|<(2\varepsilon_0)^{-1}\}$
to conclude uniqueness and norm-convergence of iterates to the
fixed point, or one can appeal to Proposition \ref{ka}). Thus, $x$ extends to a norm-neighbourhood 
of $p_{t_{m,n}}+iq_{t_{m,n}}$.

To summarize: either $t_{m,n}=0$, and then, by an application of \cite[Theorem 2.3]{JLMS},
$k_0$ has a Julia-Carath\'eodory derivative at $m$, and $\lim_{y\to0}k_0(m+iyn')=0$ in 
norm for all $n'>0$, or $t_{m,n}>0$, and then $x$ extends analytically around
$p_{t_{m,n}}+iq_{t_{m,n}}=m-\Re k_0(m+it_{m,n}n)+i
(t_{m,n}n-\Im k_0(m+it_{m,n}n))$. We apply this to $m=0$. 
Assume towards contradiction that $t_{0,n}=0$ for some $n>0$.
Recall from \cite[Relation (22)]{B-CAOT} the definition of the 
pseudo-horodisks at zero in ``direction'' $n>0$: 
\begin{eqnarray*}
\mathcal H(0,n) & = & \{w\in H^+(B)\colon(w-0)^*(\Im w)^{-1}(w-0)\leq n\}\\
& = & \{w\in H^+(B)\colon n^{-1/2}\Im wn^{-1/2}+n^{-1/2}\Re w(\Im w)^{-1}\Re wn^{-1/2}\leq1\},
\end{eqnarray*}
and 
\begin{eqnarray*}
\mathring{\mathcal H}(0,n) & = & \{w\in H^+(B)\colon(w-0)^*(\Im w)^{-1}(w-0)<n\}\\
& = & \{w\in H^+(B)\colon n^{-1/2}\Im wn^{-1/2}+n^{-1/2}\Re w(\Im w)^{-1}\Re wn^{-1/2}<1\}.
\end{eqnarray*}
Note that the only selfadjoint element in $\mathcal H(0,n)$ is zero. Indeed,
by definition, if $w\in\mathcal H(0,n),$ then $n\ge\Re w(\Im w)^{-1}\Re w+\Im w$, so that if
$\|\Im w\|\to0$, then necessarily $\|\Re w\|\to0$ (in fact one can easily obtain the estimate
$\|\Im w\|\ge\|\Im w\|(n-\Im w)>\|\Im w\|\Re w(\Im w)^{-1}\Re w\ge(\Re w)^2$). Consider
$B(iyn,y^{-1/2}),y>0$, with $B$ defined in relation \eqref{pseudoball}. We have:
$$
\mathring{\mathcal H}(0,n)\subseteq\bigcap_{0<t<1}\bigcup_{0<y<t}B(iyn,y^{-1/2})\subseteq
{\mathcal H}(0,n).
$$
This has been shown in \cite{B-CAOT}, but we will provide a sketch of the proof below.
Thus, assume towards contradiction that $a\in\mathring{\mathcal H}(0,n)$, but 
$a\not\in\bigcap_{0<t<1}\bigcup_{0<y<t}B(iyn,y^{-1/2})$. Then there exist
a $t_0\in(0,1)$ such that $a\not\in B(iyn,y^{-1/2})$ for any $y\in(0,t_0)$. That is,
$(a-iyn)^*(\Im a)^{-1}(a-iyn)\not\leq n$ for all $y\in(0,t_0)$. At the same time,
there exists an $\epsilon_{a,n}\in(0,+\infty)$ such that $a^*(\Im a)^{-1}a\leq n-
\epsilon_{a,n}\cdot1$. However, $(a-iyn)^*(\Im a)^{-1}(a-iyn)=a^*(\Im a)^{-1}\!a
+y(in(\Im a)^{-1}\!a\!-ia^*\!(\Im a)^{-1}n+yn(\Im a)^{-1}n)\leq a^*\!(\Im a)^{-1}a
+y(2\|n\|\|(\Im a)^{-1}\|\|a\|+y\|n\|^2\|(\Im a)^{-1}\|)<a^*(\Im a)^{-1}a+
\epsilon_{a,n}\cdot1\leq n$ for all $y\in(0,\|n\|\|(\Im a)^{-1}\|(\sqrt{\|a\|^2+\epsilon_{a,n}}
-\|a\|))$. This is a contradiction. Thus the first inclusion holds.
The second inclusion is equally simple: $a\in B(iy_jn,y_j^{-1/2})$ for some sequence
$y_j$ decreasing to zero is equivalent to 
$a^*(\Im a)^{-1}a+y_j(in(\Im a)^{-1}a-ia^*(\Im a)^{-1}n+y_jn(\Im a)^{-1}n)
\leq n$ for all $j\in\mathbb N$, which implies $a^*(\Im a)^{-1}a\leq n$, that is,
$a\in\mathcal H(0,n)$. We have
\begin{equation}\label{something}
k_0\left(\mathring{\mathcal H}(0,n)\right)\subseteq
k_0\left(\bigcap_{0<t<1}\bigcup_{0<y<t}B(iyn,y^{-1/2})\right)
\end{equation}
$$
\subseteq\bigcap_{0<t<1}k_0\left(\bigcup_{0<y<t}B(iyn,y^{-1/2})\right)
=\bigcap_{0<t<1}\bigcup_{0<y<t}k_0(B(iyn,y^{-1/2})).
$$
Recall that $iyn=x(p_y+iq_y)=p_y+iq_y+k_0(x(p_y+iq_y))=p_y+iq_y+k_0(iyn)$,
that is, $iyn=x(p_y+iq_y)$ is a fixed point for $w\mapsto p_y+iq_y+k_0(w)$.
Thus, 
\begin{equation}\label{someotherthing}
k_0(B(yn,y^{-1/2}))\subseteq B(iyn,y^{-1/2})-(p_y+iq_y).
\end{equation}
We have seen that $p_y=-\Re k_0(iyn)$ tends to zero in norm (in fact 
$\|p_y/y\|$ is bounded as $y\to0$), and $q_y=yn-\Im k_0(iyn)\to0$ in norm
as $y\to0$ (in fact, $\|q_y/y\|$ is uniformly bounded for $y\in(0,1)$).

We claim that 
$$
\bigcap_{0<t<1}\bigcup_{0<y<t}(B(iyn,y^{-1/2})-(p_y+iq_y))\subseteq\mathcal H(0,n).
$$
Assume that is not the case. Then there exists 
$$
a_0\in\bigcap_{0<t<1}\bigcup_{0<y<t}(B(iyn,y^{-1/2})-(p_y+iq_y))\setminus\mathcal H(0,n),
$$
that is, for all $t\in(0,1)$, there exists $0<y<t$ such that $a_0\in B(iyn,y^{-1/2})-(p_y+iq_y)$,
and yet $a^*_0(\Im a_0)^{-1}a_0\not\leq n$. So (representing $B$ on a Hilbert space
via the GNS construction), there exists a unit vector $\xi$ and a number $\eta>0$
such that 
\begin{equation}\label{et}
\langle(\Im a_0)^{-1}a_0\xi,a_0\xi\rangle>\langle n\xi,\xi\rangle+\eta
\end{equation}
and $a_0=\alpha_0-p_y-iq_y$, where $(\alpha_0-iyn)^*(\Im\alpha_0)^{-1}(\alpha_0-iyn)
\leq n$. Thus, we found a sequence $\{y_j\}_{j\in\mathbb N}$ decreasing to zero
such that 
$$
(a_0+p_{y_j}+iq_{y_j}-iy_jn)^*(\Im a_0+q_{y_j})^{-1}(a_0+p_{y_j}+iq_{y_j}-iy_jn)\leq n;
$$
in particular,
$$
\left\langle (\Im a_0+q_{y_j})^{-1}(a_0+p_{y_j}+iq_{y_j}-iy_jn)\xi,
(a_0+p_{y_j}+iq_{y_j}-iy_jn)\xi\right\rangle\leq\langle n\xi,\xi\rangle.
$$
Expanding, we obtain
\begin{eqnarray}
\lefteqn{\left\langle (\Im a_0+q_{y_j})^{-1}a_0\xi,a_0\xi\right\rangle+2\Re\left\langle
(\Im a_0+q_{y_j})^{-1}a_0\xi,(p_{y_j}+iq_{y_j}-iy_jn)\xi\right\rangle}\nonumber\\
& & \mbox{}+
\left\langle(\Im a_0+q_{y_j})^{-1}(p_{y_j}+iq_{y_j}-iy_jn)\xi,(p_{y_j}+iq_{y_j}-iy_jn)\xi\right\rangle
\nonumber\\
& \leq & \langle n\xi,\xi\rangle.\label{lang}
\end{eqnarray}
From \eqref{et} and \eqref{lang} together we obtain (by cancelling $\langle n\xi,\xi\rangle$)
\begin{eqnarray}
\lefteqn{\left\langle(\Im a_0)^{-1}a_0\xi,a_0\xi\right\rangle-\eta}\nonumber\\
& > & \left\langle (\Im a_0+q_{y_j})^{-1}a_0\xi,a_0\xi\right\rangle+2\Re\left\langle
(\Im a_0+q_{y_j})^{-1}a_0\xi,(p_{y_j}+iq_{y_j}-iy_jn)\xi\right\rangle\nonumber\\
& & \mbox{}+
\left\langle(\Im a_0+q_{y_j})^{-1}(p_{y_j}+iq_{y_j}-iy_jn)\xi,(p_{y_j}+iq_{y_j}-iy_jn)\xi\right\rangle.
\nonumber
\end{eqnarray}
We re-arrange this relation to get 
\begin{eqnarray}
\lefteqn{\left\langle(\Im a_0)^{-1}a_0\xi,a_0\xi\right\rangle-
\left\langle (\Im a_0+q_{y_j})^{-1}a_0\xi,a_0\xi\right\rangle-\eta}\nonumber\\
& = & \left\langle (\Im a_0+q_{y_j})^{-1}q_{y_j}(\Im a_0)^{-1}a_0\xi,a_0\xi\right\rangle-\eta\nonumber\\
& > & 2\Re\left\langle
(\Im a_0+q_{y_j})^{-1}a_0\xi,(p_{y_j}+iq_{y_j}-iy_jn)\xi\right\rangle\nonumber\\
& & \mbox{}+\left\langle(\Im a_0+q_{y_j})^{-1}(p_{y_j}+iq_{y_j}-iy_jn)\xi,(p_{y_j}+iq_{y_j}-iy_jn)\xi
\right\rangle.
\nonumber
\end{eqnarray}
Since $\lim_{j\to\infty}\|q_{y_j}\|=\lim_{j\to\infty}\|p_{y_j}\|=\lim_{j\to\infty}y_j=0$, when we take
limit as $j\to\infty$ in the above inequality, we obtain $-\eta>0$, a contradiction. Thus,
$$
\bigcap_{0<t<1}\bigcup_{0<y<t}(B(iyn,y^{-1/2})-(p_y+iq_y))\subseteq\mathcal H(0,n).
$$
Combining this with relations \eqref{something} and \eqref{someotherthing},
we obtain 
$$
k_0(\mathcal H(0,n))\subseteq\mathcal H(0,n).
$$
Quite trivially, 
\begin{eqnarray*}
& & a\in\mathcal H(0,n)\iff n^{-1/2}a^*(\Im a)^{-1}an^{-1/2}\le1\\
& & \iff n^{-1/2}\Im an^{-1/2}+n^{-1/2}\Re a(\Im a)^{-1}\Re an^{-1/2}\le 1\\
& & \iff (\Im a+\Re a(\Im a)^{-1}\Re a)^{-1}\ge n^{-1}\\
& & \iff \Im(-a^{-1})\ge n^{-1}.
\end{eqnarray*}
Thus, $\mathcal H(0,n)$ is mapped bijectively (and as a noncommutative set)
onto the set $\{a\in H^+(B)\colon \Im a\ge n^{-1}\}$ by the correspondence
$a\mapsto -a^{-1}.$ By the definition of $k_0$ (see \eqref{kazero}), it follows that
$h(\{a\in H^+(B)\colon \Im a\ge n^{-1}\})\subseteq\{a\in H^+(B)\colon \Im a\ge n^{-1}\}$.

However, our hypothesis on $h$ states that $\lim_{y\to\infty}\frac{\langle\Im h(\Re a+iy\Im a)\xi,\xi
\rangle}{y}=0$ for any $a\in H^+(B)$ and unit vector $\xi$. That is, given $u=u^*,v>0$, there 
exists an $y_{u,v,\xi}>0$ depending on $u$, $v$ and $\xi$ such that $\langle\Im h(u+iyv)\xi,\xi\rangle<y\langle
v\xi,\xi\rangle/2$ whenever $y\ge y_{u,v,\xi}$. But, choosing $n^{-1}=yv$, we obtain $\Im h(u+iyv)\ge\Im(u+iyv)=yv$, a contradiction. This concludes the proof of our 
proposition.
\end{proof}
By Proposition \ref{horro}, $a\mapsto-(b_0+h(-a^{-1}))^{-1}$ has an attracting fixed point
in $H^+(B),$ call it $a(b_0)$. Then $\omega(b_0):=-a(b_0)^{-1}$ is the fixed point of $w\mapsto b_0+h(w)$. 
The results of \cite{AKV} allow us to conclude the proof of Theorem \ref{semigr}.
\end{proof}

In order to argue that Theorem \ref{semigr} solves the problem of defining free convolution 
powers of distributions of unbounded selfadjoint random variables in the context of Definition \ref{distri} and Hypotheses (H1), (H2), let
us show that if $X=X^*\in M$, then $h_X(b)=E[(b-X)^{-1}]^{-1}-b$ satisfies the
hypothesis of Theorem \ref{semigr}. Fix $b=u+iv$, $u=u^*,v>0$. 
Then 
\begin{eqnarray*}
h_X(u+zv) & = & E\left[(u-X+zv)^{-1}\right]^{-1}-u-zv\\
& = & v^{1/2}E\left[\left(z+v^{-1/2}(u-X)v^{-1/2}\right)^{-1}\right]^{-1}v^{1/2}-u-zv\\
& = &  v^\frac12\left\{E\left[\left(z+v^{-\frac12}(u-X)v^{-\frac12}\right)^{-1}\right]^{-1}-z-
v^{-1/2}uv^{-1/2}\right\} v^\frac12.
\end{eqnarray*}
We argue that $h_X$ satisfies the hypothesis of Theorem \ref{semigr}. This
means (via a polarization argument) to show that 
$\lim_{y\to+\infty}\frac{\langle\Im h_X(u+iyv)\xi,\xi\rangle}{y}=0$. Denote $Y=v^{-\frac12}(u-X)v^{-\frac12}$.
Since $\Im (m+in)^{-1}=-(mn^{-1}m+n)^{-1},$ $\Re(m+in)^{-1}
=n^{-1}m(mn^{-1}m+n)^{-1}$, we have
$$
\Im E\left[\frac{1}{z-Y}\right]=-E\left[\frac{y}{y^2+(x-Y)^2}\right]<0,\quad
\Re E\left[\frac{1}{z-Y}\right]=E\left[\frac{x-Y}{y^2+(x-Y)^2}\right],
$$
where $z=x+iy$. Thus,
\begin{eqnarray*}
\lefteqn{\Im E\left[\left(z-Y\right)^{-1}\right]^{-1}=}\\
& & \left\{E\left[\frac{x-Y}{y^2+(x-Y)^2}\right]E\left[\frac{y}{y^2+(x-Y)^2}\right]^{-1}
E\left[\frac{x-Y}{y^2+(x-Y)^2}\right]\right.\\
& & \left.\mbox{}+E\left[\frac{y}{y^2+(x-Y)^2}\right]\right\}^{-1}\\
& \leq & E\left[\frac{y}{y^2+(x-Y)^2}\right]^{-1},
\end{eqnarray*}
which makes 
$$
\Im E\left[\left(z-Y\right)^{-1}\right]^{-1}-y
\leq y\left(E\left[\frac{y^2}{y^2+(x-Y)^2}\right]^{-1}-1\right).
$$
Dividing by $y$ provides us with the majorizing term $E\left[\frac{y^2}{y^2+(x-Y)^2}\right]^{-1}-1$.
This, as a function of $y$, is decreasing, as it can be seen by taking the (classical) derivative
with respect to $y$:
\begin{eqnarray*}
\lefteqn{
\partial_yE\left[\frac{y^2}{y^2+(x-Y)^2}\right]^{-1}=}\\
& & \mbox{}-E\left[\frac{y^2}{y^2+(x-Y)^2}\right]^{-1}
E\left[\frac{2y(x-Y)^2}{(y^2+(x-Y)^2)^2}\right]E\left[\frac{y^2}{y^2+(x-Y)^2}\right]^{-1}\leq0.
\end{eqnarray*}
Thus, $E\left[\frac{y^2}{y^2+(x-Y)^2}\right]^{-1}-1$ is a decreasing function of $y$. 
If it does not decrease to zero, then there exists a positive operator $0\neq c\ge0$
which belongs to the von Neumann algebra $B$
such that $\lim_{y\to\infty}E\left[\frac{y^2}{y^2+(x-Y)^2}\right]^{-1}=1+c$
in the strong operator topology. 
Multiplying left and right by $(1+c)^{-1/2}$ allows us to conclude that
$(1+c)^{1/2}E\left[\left(z-Y\right)^{-1}\right](1+c)^{1/2}$
belongs to the norm-ball of center $-i/(2y)$ and radius $1/(2y)$. Taking the imaginary part
and multiplying by $y$ yields
$$
\lim_{y\to\infty}(1+c)^{1/2}E\left[\frac{y^2}{y^2+(x-Y)^2}\right](1+c)^{1/2}=1
$$
in the so-topology. Thus\footnote{We use here that if $0<b_j^{-1}$ decreases to 1,
then $0<b_j$ increases to 1; this can be seen by evaluating 
$\langle(1-b_j)^{1/2}\xi,\xi\rangle^2=\langle b_j^{1/2}(b_j^{-1}-1)^{1/2}\xi,\xi\rangle^2
\leq\langle b_j\xi,\xi\rangle\langle(b_j^{-1}-1)\xi,\xi\rangle$.}
$$
\lim_{y\to\infty}E\left[\frac{y^2}{y^2+(x-Y)^2}\right]=(1+c)^{-1}.
$$
Composing this with any wo-continuous state $\varphi$ on the 
algebra of $B$ provides us with a state $\varphi\circ E$ on $M$
with respect to which the distribution of $Y$ is not a probability, contradicting (H1).

\begin{remark}\label{li}
\begin{trivlist}
\item[\ (1)]
We provide here a very brief sketch of a selfadjoint version of the classical realization-linearization trick
that allows us to conclude that, under Hypotheses (H1), (H2), the noncommutative map 
$$
\{\alpha\in\mathbb C^{n\times n}\!\colon\!\alpha=\alpha^*\!\}\times H^+_n(B)\ni(\alpha,b)\mapsto 
G_{\alpha\otimes X}(b)\!=\!({\rm Id}_{\mathbb C^{n\times n}}\otimes E)\!\left[(b-\alpha\otimes X)^{-1}\right]\!, n\in\mathbb N,
$$
determines the distribution of $X$ as defined in Definition \ref{distri}.
 Direct computation using the Schur complement shows that the matrix
$$
\begin{bmatrix}
z & 0 &  \cdots & 0 & 0 & b_n\\
0 & 0 &  \cdots & 0 & X & -1\\
0 & 0 &  \cdots & b_{n-1} & -1 & 0\\
\vdots &  \vdots &  \reflectbox{$\ddots$} & \vdots & \vdots  & \vdots\\
0 & X &  \cdots & 0 & 0 & 0\\
b_0 & -1  & \cdots & 0 & 0 & 0
\end{bmatrix}
$$
is invertible, and its inverse has $(z-b_0X\cdots b_{n-1}Xb_n)^{-1}$ as its $(1,1)$
 entry. We call this matrix the realization of the monomial $b_0X\cdots b_{n-1}Xb_n$.
We note that its lower right $(2n)\times(2n)$ corner of this matrix is invertible 
regardless of the monomial $b_0X\cdots b_{n-1}Xb_n$.
If $\begin{bmatrix}
z & u_j\\
v_j & Q_j
\end{bmatrix}$ are realizations as above of the monomials $m_j$, $j=1,2$, 
then  $\begin{bmatrix}
z & u_1 & u_2\\
v_1 & Q_1 & 0\\
v_2 & 0 & Q_2
\end{bmatrix}$ is a realization of $m_1+m_2$: its inverse has $(z-m_1-m_2)^{-1}$ 
as its $(1,1)$ entry, and its lower right corner is invertible regardless of the 
choice of $m_1$ and $m_2$. Finally, if $\begin{bmatrix}
z & u\\
v & Q
\end{bmatrix}$ is a realization of the polynomial $P\in B\langle X\rangle$, then $\begin{bmatrix}
z & u & v^*\\
u^* & 0 & Q^*\\
v & Q & 0 
\end{bmatrix}$ is a {\em selfadjoint} realization of the selfadjoint polynomial
 $P+P^*$. Of course, it is quite possible that applying the expectation $E$ 
entrywise to the inverse of such a matrix is impossible (for instance, $E[X]$ 
might not be defined for a given unbounded $X$ affiliated with $M$). However, 
if $Q=Q^*$, then for any number $\epsilon>0$, the matrix $\begin{bmatrix}
z & v^*\\
v & Q+i\epsilon
\end{bmatrix}$ is invertible in the space of matrices over $M$ 
of the corresponding size, and its inverse is bounded in $M$.
We take $\lim_{\epsilon\to0}\left(E\begin{bmatrix}
z & v^*\\
v & Q+i\epsilon
\end{bmatrix}\right)_{1,1}$ to obtain $E\left[(z-P)^{-1}\right]$. 
Composing with any state $\varphi$ on $B$ yields the claimed result.

\item[\ (2)] In view of the above, applying Theorem  \ref{semigr} to 
solve our problem in free probability requires only one more ingredient: 
we must show that if $X=X^*$ satisfies Hypotheses (H1) and (H2), then 
so does $\alpha\otimes X$ for any $\alpha=\alpha^*\in\mathbb C^{n\times n}$, 
$n\in\mathbb N$. It is quite easy to show that $\alpha\otimes X$ satisfies (H1): 
elementary matrix arithmetics rules guarantee that $B^{n\times n}$ and 
$\alpha\otimes X$ form a ${}^*$-algebra of possibly unbounded 
operators affiliated to $M^{n\times n}$.

We show next that 
$({\rm Id}_{\mathbb C^{n\times n}}\otimes E)\left[\Im(b-\alpha\otimes X)^{-1}\right]<0$
for any $b\in H^+_n(B)$. This is again quite straightforward. 
First, since $\alpha$ is a scalar selfadjoint matrix, it can be diagonalized 
by conjugating with a unitary matrix $U\in\mathbb C^{n\times n}$:
$$
U^*({\rm Id}_{\mathbb C^{n\times n}}\otimes E)\left[\Im(b\!-\!\alpha\otimes X)^{-1}\right]U\!=\!
({\rm Id}_{\mathbb C^{n\times n}}\otimes E)\left[\Im(U^*\!bU\!-\!\mathrm{diag}(\alpha_1 X,\dots,\alpha_nX))^{-1}\right],
$$
where $\alpha_1,\dots,\alpha_n$ are the real eigenvalues of $\alpha$. Thus, it is enough to prove the statement for diagonal matrices $\alpha$.
Second, by a simple maximum principle argument, if this inequality happens at $iv$ for some $v>0$, then it must happen everywhere on $H^+_n(B)$.
Thus, it is enough to prove the statement for $b=iI_n\otimes1$. But in this case, the statement is immediate: 
$({\rm Id}_{\mathbb C^{n\times n}}\otimes E)\left[\Im(iI_n\otimes1-\mathrm{diag}(\alpha_1 X,\dots,\alpha_nX))^{-1}\right]
=\textrm{diag}(E\left[(i1-\alpha_1X)^{-1}\right],\dots,E\left[(i1-\alpha_nX)^{-1}\right])<0$, according to (H2).
\end{trivlist}
\end{remark}

Denote $({\rm Id}_{\mathbb C^{n\times n}}\otimes E)\left[(b-\alpha\otimes X)^{-1}\right]$, $\alpha=\alpha^*\in\mathbb C^{n\times n}$, $n\in\mathbb N$, by $G_X$,
and the same for $F$ and $h$.

\begin{cor}
Under hypotheses {\rm (H1)} and {\rm (H2)}, $\mu_X^{\boxplus\rho}$ is well-defined for all 
cp maps $\rho\colon B\to B$ such that $\rho-{\rm Id}_B$ is still cp.
\end{cor}
\begin{proof}
Apply Theorem \ref{semigr} to $h(w)=(\rho-\mathrm{Id}_B)h_{X}(w)$.
\end{proof}

Let us conclude with a brief comment on the Nevanlinna representation of 
$h_X$. If $X\in M$, results of \cite{PV} guarantee the
existence of an extension $B\langle\mathcal X\rangle$ of $B$ in which there exists a bounded 
selfadjoint element $\mathcal X,$ and of a completely positive map
$\rho\colon B\langle\mathcal X\rangle\to B$ such that 
$h_X(b)=-E[X]+\rho\left[(\mathcal X-b)^{-1}\right]$, $b\in H^+(B)$.
As in the case of the classical Nevanlinna representation, for unbounded
operators $X$, the cp map $\rho$ is not the appropriate completely positive map anymore. We
define $\eta\colon B\langle\mathcal X\rangle\to B$ by $\eta[a]=
\rho\left[(\mathcal X-i)^{-1}a(\mathcal X+i)^{-1}\right].$ The 
correspondence becomes now
$$
h_X(b)=\Re h_X(i)+\eta\left[(\mathcal X-b)^{-1}+b+b(\mathcal X-b)^{-1}b\right],\quad\Im b>0.
$$
Observe that indeed $\Im h(i)=\eta[i]$. Rewriting this map as
\begin{equation}\label{GenNev}
h_X(b)=
\Re h_X(i)+\eta\left[(\mathcal X-b)^{-1}-\mathcal X+\mathcal X(\mathcal X-b)^{-1}\mathcal X\right]
\end{equation}
makes it clear that it maps $H^+(B)$ in its closure. It would be interesting to 
determine whether this expression is equivalent in some sense to the one
obtained in \cite{PPT} for maps from the nc unit ball of a $C^*$-algebra to its 
right half-plane.


\end{document}